\documentclass{amsart}

\usepackage[foot]{amsaddr}
\usepackage[numbers]{natbib}

\usepackage{lineno}
\usepackage{amsmath,amssymb}
\usepackage[utf8]{inputenc} 
\usepackage{accents}				
\usepackage{todonotes}
\usepackage{subcaption}
\usepackage{stmaryrd} 
\usepackage{amsthm} 
\usepackage[makeroom]{cancel} 
\usepackage{bookmark} 

\usepackage[margin=1in]{geometry}

\newcommand{\norm}[1]{\left\lVert#1\right\rVert}
\newcommand{\mat}[1]{\underline{\mathbf{#1}}}
\newcommand\acclrvec[1]{\accentset{\,\leftrightarrow}{#1}}	
\newcommand{\blocktensor}[1]{\acclrvec{{\mathbf #1}}}

\newcommand{\Nabla} {\vec{\nabla}}
\newcommand{\numfluxb}[1]{\hat{\mathbf{#1}} }

\newcommand\threeMatrix[1]{\underline{ #1}}
\newcommand{\ec}{{\mathrm{EC}}}

\newcommand{\bigpartialderiv}[2]{ \frac{\partial {#1}}{\partial {#2} } }

\newcommand\stateG[1]{\boldsymbol #1}

\newcommand\state[1]{\mathbf{#1}}

\newcommand{\supEuler}{{\mathrm{Euler}}}
\newcommand{\supMHD}{{\mathrm{MHD}}}
\newcommand{\supGLM}{{\mathrm{GLM}}}
\newcommand{\DG}{{\mathrm{DG}}}
\newcommand{\FV}{{\mathrm{FV}}}
\newcommand{\LLF}{{\mathrm{LLF}}}

\newcommand{\vn}{\vec{\tilde{n}}}

\newcommand{\avg}[1]{\left\{\hspace*{-3pt}\left\{#1\right\}\hspace*{-3pt}\right\}}

\def\D{\mathtt{D}}
\def\Didp{\widehat{\D}}

\def\cc{\hat{\mathtt{c}}} 
\def\mm{\mathtt{m}} 
\def\NN{\mathcal{N}} 
\def\dd{\hat{\mathtt{d}}} 

\newtheorem{proposition}{Proposition}

\title[Subcell limiting strategies for DGSEM]{Subcell Limiting Strategies for Discontinuous Galerkin Spectral Element Methods}

\author[Rueda-Ramírez]{Andrés M. Rueda-Ramírez$^{*,1}$}
\address{$^1$Department of Mathematics and Computer Science, University of Cologne, 50931 Cologne, Germany}

\author[Pazner]{Will Pazner$^2$}
\address{$^2$Center for Applied Scientific Computing, Lawrence Livermore National Laboratory, Livermore, CA, USA}

\author[Gassner]{Gregor J. Gassner$^{1,3}$}
\address{$^3$Center for Data and Simulation Science, University of Cologne, 50931 Cologne, Germany}

\begin{document}

\begin{abstract}
We present a general family of subcell limiting strategies to construct robust high-order accurate nodal discontinuous Galerkin (DG) schemes.
The main strategy is to construct compatible low order finite volume (FV) type discretizations that allow for convex blending with the high-order variant with the goal of guaranteeing additional properties, such as bounds on physical quantities and/or guaranteed entropy dissipation.
For an implementation of this main strategy, four main ingredients are identified that may be combined in a flexible manner: (i) a nodal high-order DG method on Legendre-Gauss-Lobatto nodes, (ii) a compatible robust subcell FV scheme, (iii) a convex combination strategy for the two schemes, which can be element-wise or subcell-wise, and (iv) a strategy to compute the convex blending factors, which can be either based on heuristic troubled-cell indicators, or using ideas from flux-corrected transport methods.

By carefully designing the metric terms of the subcell FV method, the resulting methods can be used on unstructured curvilinear meshes, are locally conservative, can handle strong shocks efficiently while directly guaranteeing physical bounds on quantities such as density, pressure or entropy.
We further show that it is possible to choose the four ingredients to recover existing methods such as a provably entropy dissipative subcell shock-capturing approach or a sparse invariant domain preserving approach.

We test the versatility of the presented strategies and mix and match the four ingredients to solve challenging simulation setups, such as the KPP problem (a hyperbolic conservation law with non-convex flux function), turbulent and hypersonic Euler simulations, and MHD problems featuring shocks and turbulence.
\end{abstract}

\maketitle

\makeatletter
\def\blfootnote{\xdef\@thefnmark{}\@footnotetext}
\makeatother

\blfootnote{$^*$Corresponding author}

\section{Introduction}

There is a vast literature on the construction of robust high-order numerical methods (and in particular, discontinuous Galerkin (DG) methods) for hyperbolic conservation laws and other advection-dominated problems.
A large class of such methods is based on the idea of combining, through some type of limiting procedure, the high-order discretization with a related, more robust, low-order method.
In this paper, we describe a general methodology and framework for constructing robust nodal DG spectral element methods (DGSEM) using subcell convex limiting strategies.
A number of different methods from the literature can be recast in the context of this framework, and it can also be used to demonstrate that some seemingly different methods are, in fact, equivalent.
Furthermore, this framework allows for the formulation of new classes of methods by combining, in a flexible way, different components of the limiting algorithm that we enumerate in this paper.

In this paper, we consider a number of subcell limiting strategies for DGSEM on Legendre-Gauss-Lobatto (LGL) nodes on unstructured curvilinear quadrilateral/hexahedral meshes.
Broadly speaking, there are two classes of subcell-based limiting approaches.
The first class of approaches identifies a troubled elements, and in those elements, switches over to a robust discretization on a refined subcell grid.
The robust discretizations may be based on, for example, TVD (total variation diminishing) finite volume (FV) methods \cite{Sonntag2017,Krais2019,Sonntag2017a,Sonntag2014,Gaburro2021,Dumbser2016} or can be even high-order accurate WENO-type approximations \cite{Boscheri2019,Dumbser2014}.

The second class of approaches also identifies troubled elements, but then computes a convex blending of the high-order DG scheme with a compatible low-order discretization.
In this work, we focus on this second approach.
By compatible low-order discretization, we refer to a low-order scheme that uses the same degrees of freedom (DOF) as the high-order DG method, such that a convex combination of the two methods can be formulated.
This general approach of convex blending dates back to the development of the flux-corrected transport (FCT) methods introduced by \citet{boris1976flux} in 1976, since which many variations have been developed, e.g., \cite{fct2012,lohner1987finite, kuzmin2010failsafe}.
In 1994, \citet{Giannakouros1994} formulated an FCT approach for spectral element methods for the compressible Euler equations; in that work, the authors rewrote the spectral element method as FV-type update scheme, which enabled convex blending with the low-order scheme.
Such a so-called conservative reformulation of the high-order method was also presented for spectral (very high order) methods in, e.g., \citet{Sidilkover1993}.
More recently, \citet{Fisher2013} and \citet{Carpenter2014} demonstrated that diagonal-norm summation-by-parts (SBP) operators may be equivalently reformulated as a subcell based FV-type conservative update.
Furthermore, they established a general approach to obtain entropy-consistent high-order SBP discretizations based on split-formulations of the original partial differential equations.
Since DGSEM with LGL nodes are themselves diagonal-norm SBP operators, these ideas can be applied directly to DGSEM \cite{gassner_skew_burgers}, resulting in robust entropy-stable high-order discretizations, cf.~\cite{Carpenter2014,Gassner2021,chan2018,Renac2019,ranocha2019,Bohm2018}.

\citet{vilar2019posteriori} introduced a way to re-write a DG scheme into a compatible high-order FV scheme on an arbitrary subcell distribution and used this in combination with a low-order FV approximation for shock capturing.
Shock capturing for DGSEM with LGL based on convex blending was recently proposed with an element-wise approach in \citet{Hennemann2020}.
In that work, the goal was to retain the provable entropy consistency of the hybrid scheme.
The blending coefficient is adjusted with a shock indicator such that the robust low-order method is activated only in the presence of shocks.
\citet{Rueda-Ramirez2020} extended the method to compressible magnetohydrodynamics (MHD), and showed that higher-order reconstructions can be used to enhance the accuracy of the low-order method, while retaining the property of provable entropy consistency.
More recently, \citet{Rueda-Ramirez2021} showed that the subcell FV method of \citet{Hennemann2020} can be used to preserve positivity of density and pressure for standard and split-form DGSEM discretizations of the Euler equations.

Recently, the concept of invariant domain preserving (IDP) methods has been used to formulate robust higher-order methods for conservation laws.
In 2008, \citet{Berthon2008} used the concept of invariant domain preservation to construct a second order MUSCL-type FV scheme that, for example, guarantees positivity of the solution for all times.
IDP methods have been developed in the context of continuous finite element methods in the work of \citet{Guermond2016,Guermond2017} and \citet{guermond2019invariant}.
In these works, the authors proposed to construct IDP low-order methods for higher-order nodal continuous finite element methods (FEM) by adding a so-called graph viscosity term, analogous to (local) Lax-Friedrichs (LLF) dissipation, allowing them to prove the IDP property on general unstructured grids.
\citet{Pazner2020} extended this approach to DGSEM with LGL.
Instead of using the full SEM operators to construct a low order IDP scheme analogous to the continuous FEM, a sparse, compatible low-order IDP discretization was introduced, substantially reducing the dissipation for approximations with high polynomial degrees, when compared with the full low-order graph viscosity approach.
In addition to an element-wise convex blending approach, Pazner introduced a methodology for consistently blend the high-order DGSEM on a local subcell basis for each individual LGL node within a DG element, demonstrating that such a localized blending allows for the recovery of subcell accuracy within high-order elements, even when limiting is required.

In the remainder of this paper, we describe a systematic framework for formulating methods of this type.
In Section~\ref{sec:strategies}, we enumerate four components, or ``building blocks\rlap{,}'' needed to implement the subcell based limiting.
These components may be combined in a flexible manner, allowing one to recover existing methods, as well as generate new methods.
In Section~\ref{sec:equivalence}, we use this framework to show that the sparse IDP method of \citet{Pazner2020} is equivalent to the subcell FV methods of \citet{Hennemann2020} when the standard DGSEM is used with the Rusanov or LLF numerical flux function.
In Section~\ref{sec:options}, a number of possible options and choices are listed and summarized, describing the benefits and trade-offs of each combination.
Section~\ref{sec:results} includes numerical results, in which we apply a variety of limiting strategies for a range of challenging test cases, including the KPP problem, a variety of compressible Euler test cases, and the equations of ideal MHD.
We draw our conclusions in the last section, Section \ref{sec:conclusion}.

\section{Discretization Building Blocks}\label{sec:strategies}

In this section, we introduce the four ingredients that we will use to mix and match for our limiting strategies. The goal is to construct robust high-order DG discretizations of hyperbolic conservation laws,
\begin{equation} \label{eq:ConsLaw}
\bigpartialderiv{\state{u}}{t} + \nabla \cdot \blocktensor{f}(\state{u}) = \state{0},
\end{equation}
describing the evolution of a state quantity, $\state{u}$, in time, $t$, where $\blocktensor{f}$ is a flux function.

We focus on the spatial discretization in following sections, and make use of explicit strong stability-preserving (SSP) Runge-Kutta (RK) methods (cf.~\cite{gottlieb2005high,gottlieb2011strong,shu1988efficient}) for the temporal discretization with typical explicit CFL time step restrictions.

For brevity, we present the numerical methods in two space dimensions and add references to three-dimensional extensions in the literature in cases where the extension is not straightforward.

\subsection{Ingredient (I): The High-Order Method}\label{sec:DGSEM}

The first ingredient is the choice of our high-order DG method as a collocated nodal method based on spectral element  ansatz function with Legendre-Gauss-Lobatto nodes, the DGSEM with LGL,  cf.~\cite{Kopriva:2009nx,Gassner2021,Gassner_BR1,gassner2016split}. All variables are approximated within each element by piecewise Lagrange interpolating polynomials of degree $N$ on tensor-product LGL nodes. For the conservation law \eqref{eq:ConsLaw} in two spatial dimensions, the time derivative of the state quantities at node $ij$ of a curvilinear quadrilateral element reads
\begin{align} \label{eq:DGSEM}
J_{ij} \omega_{ij} \dot{\state{u}}^{\DG}_{ij}
&=
 \omega_{j} \,\,\,\, \left(
 \state{F}^{\mathrm{Vol(1)}}_{ij}
 \,\,\,\,\,\,
+ \delta_{i0} \numfluxb{f}_{(0,L)j}
- \delta_{iN} \numfluxb{f}_{(N,R)j}
\right)
\nonumber\\
&+ \omega_{i}
\underbrace{
 \left(
 \state{F}^{\mathrm{Vol(2)}}_{ij}
 \right.
}_{\mathrm{volume~integral}}
+
\underbrace{
 \left.
  \delta_{j0} \numfluxb{f}_{i(0,L)}
- \delta_{jN} \numfluxb{f}_{i(N,R)}
 \right),
}_{\mathrm{surface~integral}}
\end{align}
where $J_{ij}$ denotes the geometry mapping Jacobian from reference space to physical space, $\vec{\xi} \in [-1,1]^2 \rightarrow \vec{x} \in \Omega_k$, $\Omega_k$ denotes the quadrilateral element $k$ under consideration, $\omega_i$ and $\omega_j$ denote the reference-space quadrature weights in $\xi^1$ and $\xi^2$, respectively, $\omega_{ij} := \omega_i \omega_j$ is simply a short-hand notation for the product of quadrature weights, $\delta_{ij}$ denotes Kronecker's delta function with node indexes $i$ and $j$, and $\numfluxb{f}$ denotes the surface numerical flux function, which are typically based on approximate Riemann solvers and thus depend on the local values and the values from the neighbor elements.

The volume integral terms for the standard, e.g., \cite{Kopriva:2009nx} and split-form DGSEM, e.g., \cite{gassner2016split} in the first coordinate direction are respectively
\begin{equation}
\state{F}^{\mathrm{Vol(1),Std}}_{ij} =
- \sum_{m=0}^{N} \bar{S}_{im} \tilde{\state{f}}^{1}_{mj},
~~~~ \mathrm{and} ~~~
\state{F}^{\mathrm{Vol(1),Split}}_{ij} =
- \sum_{m=0}^{N} S_{im} \tilde{\state{f}}^{1*}_{(i,m)j},
\end{equation}
where $\tilde{\state{f}}^1_{mj}$ is the so-called contravariant flux in the first reference coordinate direction at node $mj$, and $\tilde{\state{f}}^{1*}_{(i,m)j}$ is the first contravariant numerical volume flux evaluated between nodes $ij$ and $mj$. The numerical volume flux is a two-point numerical flux function that needs to be consistent to the continuous flux and symmetric in its two arguments.
Note that by choosing a central arithmetic two-point flux, the split-form DGSEM reduces to the standard DGSEM \cite{gassner2016split}.

The volume terms in the other coordinate directions are defined in an analogous way.

The volume integral matrices are defined as
\begin{equation}
    \bar{\mat{S}} = \mat{Q} - \mat{B},
    ~~~~~
    {\mat{S}} = 2 \mat{Q} - \mat{B},
\end{equation}
with $Q_{ij} := \omega_i \ell'_j(\xi_i)$ the SBP derivative matrix, defined in terms of the Lagrange interpolating polynomials, $\{ \ell_i \}_{i=0}^N$, and $\mat{B} := \text{diag} (-1, 0, \ldots, 0, 1)$ the so-called boundary evaluation matrix.

The contravariant fluxes are defined using the Jacobian and the contravariant basis vectors of the element mapping, $\vec{a}^m_{ij} := \Nabla \xi^m$, as
\begin{equation}
    \tilde{\state{f}}^1_{ij} = \sum_{m=1}^2 \left( J{a}^1_m \right)_{ij} \, \state{f}^m_{ij},
    ~~~~~
    \tilde{\state{f}}^2_{ij} = \sum_{m=1}^2 \left(J{a}^2_m \right)_{ij} \, \state{f}^m_{ij},
\end{equation}
and the volume numerical two-point fluxes are defined with the metric terms as
\begin{align}
\tilde{\state{f}}^{1*}_{(i,m)j} &:= \blocktensor{f}^{*}(\state{u}_{ij}, \state{u}_{mj}) \cdot \avg{J\vec{a}^1}_{(i,m)j}, ~~~~~
\tilde{\state{f}}^{2*}_{i(j,m)} := \blocktensor{f}^{*}(\state{u}_{ij}, \state{u}_{im}) \cdot \avg{J\vec{a}^2}_{i(j,m)},
\end{align}
where $\avg{\cdot}_{(i,m)j}$ denotes the average operator between nodes $ij$ and $mj$. As mentioned, $\blocktensor{f}^*(\cdot,\cdot)$ is a two-point flux function and its choice leads to possible desirable properties of the final DG discretization. There are choices of two-point numerical volume fluxes that give kinetic energy preservation \cite{gassner2016split}, entropy conservation/dissipation \cite{ismail2009affordable,Chandrashekar2013}, pressure equilibrium preservation \cite{shima2021preventing}, or all of these properties together \cite{ranocha2018generalised,ranocha2021preventing}.

As mentioned, the standard DGSEM can be written in the form of a split-form DGSEM using the standard average two-point fluxes,
\begin{align}
\tilde{\state{f}}^{1*}_{(i,m)j} &:=
\avg{ \blocktensor{f} \cdot J\vec{a}^1 }_{(i,m)j}, ~~~~~
\tilde{\state{f}}^{2*}_{i(j,m)} := \avg{ \blocktensor{f} \cdot J\vec{a}^2 }_{i(j,m)}.
\end{align}

Fisher and Carpenter et al.~\cite{Fisher2013a,Carpenter2014} proved the key property that diagonal norm SBP discretizations (and hence also the DGSEM considered in this work) can be rewritten in so-called \textit{flux-differencing} form,
\begin{equation} \label{eq:DGSEM_fluxDiff}
J_{ij} \dot{\state{u}}^{\DG}_{ij} =
\frac{1}{\omega_i}
\left(
  \numfluxb{f}^{\DG}_{(i-1,i)j}
- \numfluxb{f}^{\DG}_{(i,i+1)j}
\right)
+
\frac{1}{\omega_j}
\left(
  \numfluxb{f}^{\DG}_{i(j-1,j)}
- \numfluxb{f}^{\DG}_{i(j,j+1)}
\right)
,
~~~~~~
\forall i,j=0, \ldots, N,
\end{equation}
where indexes $-1$ and $N+1$ refer to the outer states: across the left and right boundaries, respectively.
Equation \eqref{eq:DGSEM_fluxDiff} implies local subelement-wise conservation due to the well-known theorem of \citet{laxwendroff}.
We define the high-order fluxes such that the boundary fluxes match with the surface numerical fluxes, and \eqref{eq:DGSEM} is recovered with the flux-differencing formula, \eqref{eq:DGSEM_fluxDiff}.
The fluxes in $\xi^1$ are then defined as
\begin{align}
\numfluxb{f}^{\DG}_{(-1,0)j}  &= \numfluxb{f}_{(0,L)j}
\\
\numfluxb{f}^{\DG}_{(i,i+1)j} &= - \sum_{l=0}^i \state{F}^{\mathrm{Vol(1)}}_{lj}, & i=0, \ldots, N-1,
\label{eq:inteFlux}\\
\numfluxb{f}^{\DG}_{(N,N+1)j} &= \numfluxb{f}_{(N,R)j},
\end{align}
and the fluxes in $\xi^2$ (and $\xi^3$) in are defined analogously.

Note that the above-defined fluxes differ slightly from those of \citet{Fisher2013a}, as we not only transform the volume integral, but also the entire discretization including surface terms. Moreover, note that the inner fluxes \eqref{eq:inteFlux} can be computed in an efficient recursive manner by also using the symmetry of the numerical volume fluxes:
\begin{align*}
\numfluxb{f}^{\DG}_{(0,1)j} &= -\state{F}^{\mathrm{Vol(1)}}_{0j},
\\
\numfluxb{f}^{\DG}_{(i,i+1)j} &= \numfluxb{f}^{\DG}_{(i-1,i)j} - \state{F}^{\mathrm{Vol}}_{ij}, & i=1, \ldots, N-1.
\end{align*}

\subsection{Ingredient (II): The Compatible Low-Order Method}\label{sec:FV}

The second ingredient is the choice of our compatible subcell FV method for the (split-form) DGSEM on LGL nodes. As in \cite{Hennemann2020,Rueda-Ramirez2020,Rueda-Ramirez2021}, we construct the compatible subcell FV method by interpreting the LGL nodal values of the state quantities as `cell-centered' values of the FV subcells (though we note that the LGL nodes themselves are not geometrically located in the center of the subcells),
\begin{equation} \label{eq:FV}
J_{ij} \dot{\state{u}}^{\FV}_{ij} =
\frac{1}{\omega_i}
\left(
  \numfluxb{f}^{\FV}_{(i-1,i)j}
- \numfluxb{f}^{\FV}_{(i,i+1)j}
\right)
+
\frac{1}{\omega_j}
\left(
  \numfluxb{f}^{\FV}_{i(j-1,j)}
- \numfluxb{f}^{\FV}_{i(j,j+1)}
\right).
\end{equation}
As is typical for FV methods, the numerical fluxes are based on approximate Riemann solvers, cf.~\cite{Toro:1999yq},
\begin{align} \label{eq:FV_fluxes}
\numfluxb{f}^{\FV}_{(i,m)j} :=&
\norm{\vn_{(i,m)j}} \numfluxb{f}^{\FV} \left(\state{u}_{ij}, \state{u}_{mj}, \ldots ; \frac{\vn_{(i,m)j}}{\norm{\vn_{(i,m)j}}} \right),
\end{align}
resulting in very robust low-order discretizations.
The approximate Riemann solvers in particular are designed to be dissipative, and thus the straightforward idea is that such a compatible (overly) dissipative low-order method can be used to make the low-dissipation high-order DG method more robust.

To enable the approximation on curvilinear meshes, the choice of the subcell normal vectors of the FV method is crucial. \citet{Hennemann2020} showed how to derive the subcell normal vectors, $\vn_{(\cdot,\cdot)}$, from the high-order flux-differencing formula \eqref{eq:DGSEM_fluxDiff} to ensure a compatible and watertight subcell FV discretization on high-order curvilinear meshes,
\begin{align} \label{eq:SubcellMetrics}
\vn_{(i,i+1)j}  &= J\vec{a}^1_{0j} + \sum_{l=0}^{i} \sum_{m=0}^{N} Q_{lm} (J\vec{a}^1)_{mj}, &
\vn_{i(j,j+1)}  &= J\vec{a}^2_{i0} + \sum_{l=0}^{j} \sum_{m=0}^{N} Q_{lm} (J\vec{a}^2)_{im}.
\end{align}
This particular choice of the subcell metrics \eqref{eq:SubcellMetrics} also ensures free-stream preservation \cite{Hennemann2020} and thus provably discrete subcell metric identities
\begin{equation} \label{eq:metric_ids_subcell}
    \frac{ \vn_{(i-1,i)j} - \vn_{(i,i+1)j} }{\omega_i}
  + \frac{ \vn_{i(j-1,j)} - \vn_{i(j,j+1)} }{\omega_j}
  = \vec{0},
\end{equation}
which are obtained by evaluating a constant solution in \eqref{eq:FV}-\eqref{eq:FV_fluxes}.

The interpretation as a classical FV method on subcell allows us to apply standard techniques to improve the fidelity of the subcell method by e.g.\ choosing accurate approximate Riemann solvers \eqref{eq:FV_fluxes} and/or by using reconstruction techniques to improve the piecewise constant solution interpretation on the subcell (see, e.g., Figure \ref{fig:FV_reconstruction}).
\begin{figure}
\centering
\includegraphics[trim=20 10 20 10 ,clip,width=0.5\linewidth]{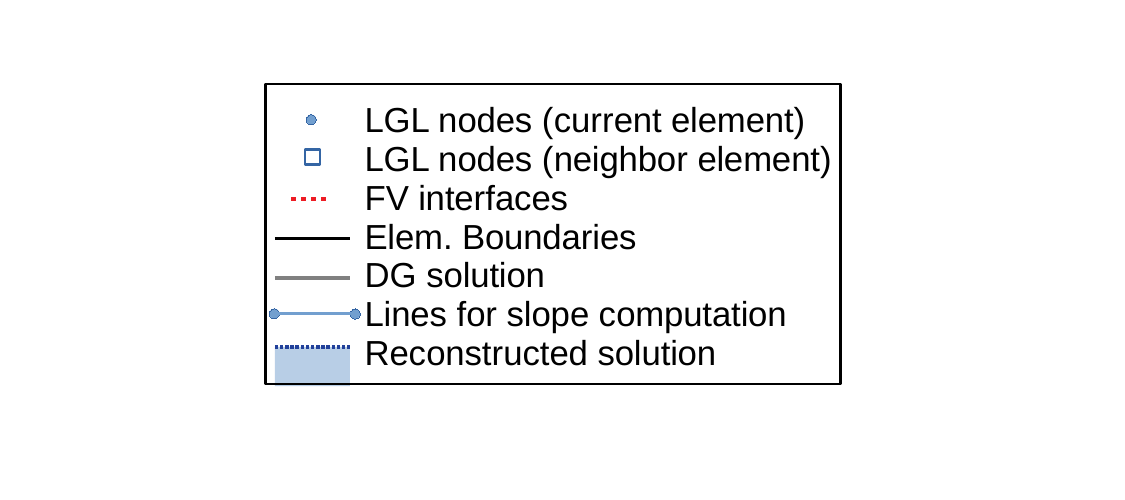}\\
\begin{subfigure}[b]{0.45\linewidth}
    \centering
	\includegraphics[trim=50 10 50 5 ,clip,width=0.75\linewidth]{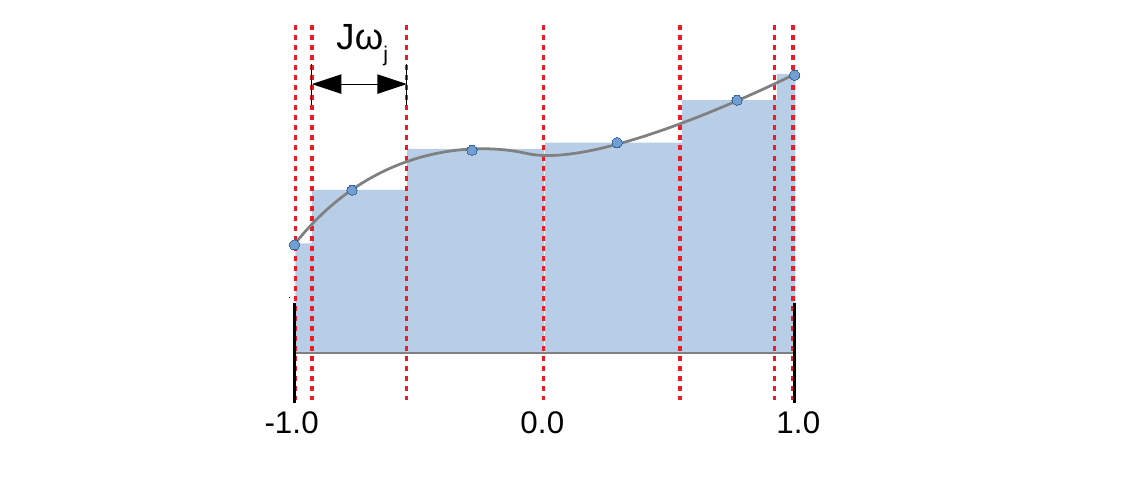}
	\caption{$1^{st}$ order reconstruction}
	\label{fig:FV_1st}
\end{subfigure}%
\begin{subfigure}[b]{0.45\linewidth}
    \centering
	\includegraphics[trim=50 10 50 5 ,clip,width=0.75\linewidth]{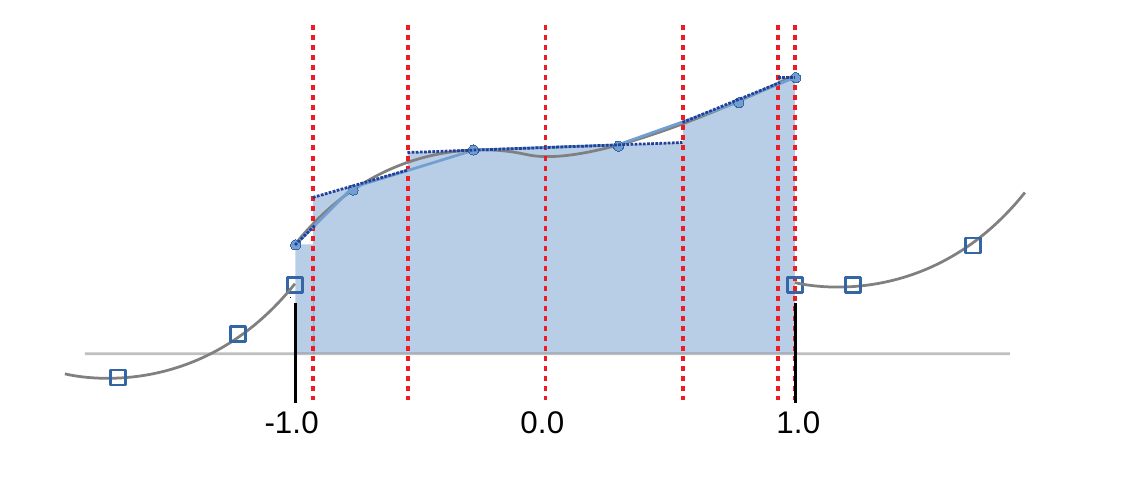}
	\caption{$2^{nd}$ order reconstruction using \textit{minmod} limiter}
	\label{fig:FV_2nd}
\end{subfigure}
\caption{Illustration of first and second-order FV-reconstruction procedures for the subcells.}
\label{fig:FV_reconstruction}
\end{figure}
So-called total variation diminishing (TVD) reconstructions based on primitive variables (e.g., $\rho$, $\vec{v}$ and $p$ for the Euler equations) can be used to enhance accuracy while complying with positivity constraints.
Moreover, TVD reconstructions based on the \textit{minmod} limiting of so-called \textit{scaled entropy variables} \cite{fjordholm2012arbitrarily} can be used to obtain higher-order entropy-stable FV schemes. As shown by \citet{Rueda-Ramirez2020}, any element-wise convex combination of such a higher-order entropy-stable FV scheme with an entropy-stable high-order DGSEM scheme remains provably entropy-stable as long as the surface terms of both schemes match.

It is clear that there are many choices and strategies to construct the low-order method, resulting in different computational complexities and numerical properties such as dissipativity, positivity preservation of physical quantities, and guaranteed entropy dissipation.

\subsection{Ingredient (III): Convex Blending Strategies}\label{sec:blending}

The third ingredient is the strategy to blend the high-order DG method from Section \ref{sec:DGSEM} with its compatible low-order FV method from Section \ref{sec:FV}.
The simplest way to blend the two schemes is to compute an \textbf{element-wise} convex combination of the spatial discretization operators, as for instance presented in \cite{Hennemann2020,Rueda-Ramirez2020,Rueda-Ramirez2021}:
\begin{equation} \label{eq:Elem_blend}
\dot{\state{u}}_{ij} =
(1-\alpha) \dot{\state{u}}_{ij}^{\DG}
+ \alpha \dot{\state{u}}_{ij}^{\FV},
\end{equation}
where $\alpha$ is an element-local blending coefficient, also known as \textit{limiting factor}.
Thus, every element has one constant value of $\alpha$ that determines the amount of FV method added to the high-order DG scheme.

Alternatively, it is possible to blend the high-order DG method with the lower-order FV method using a \textbf{subcell-wise} strategy. The subcell-wise blending was suggested by both \citet{Hennemann2020} in the context of shock capturing and \citet{Pazner2020} in the context of IDP independently, and was tested by the latter author for the combination of a standard DGSEM scheme with a low-order IDP method. The subcell-wise convex combination reads as
\begin{equation} \label{eq:Subcell_blend}
J_{ij} \dot{\state{u}}_{ij} =
\frac{1}{\omega_i}
\left(
  \numfluxb{f}_{(i-1,i)j}
- \numfluxb{f}_{(i,i+1)j}
\right)
+
\frac{1}{\omega_j}
\left(
  \numfluxb{f}_{i(j-1,j)}
- \numfluxb{f}_{i(j,j+1)}
\right),
\end{equation}
where each interface flux $\numfluxb{f}_{(\cdot,\cdot)}$ is computed as a convex combination of the DG and FV fluxes,
\begin{equation}
    \numfluxb{f}_{(a,b)} = (1-\alpha_{(a,b)}) \numfluxb{f}^{\DG}_{(a,b)} + \alpha_{(a,b)} \numfluxb{f}^{\FV}_{(a,b)},
\end{equation}
and a different blending coefficient, $\alpha_{(a,b)}$, can be selected for each interface between any two adjacent nodes $a$ and $b$. Thus, with this subcell strategy, every element has a collection of local blending factors $\alpha_{(a,b)}$ that determine locally the amount of FV mixed to the high-order DG scheme.

\subsection{Ingredient (IV): Strategies to Compute the Blending Coefficients}

The fourth and final ingredient is the computational strategy used to compute the element-wise or subcell-wise blending coefficients $\alpha$ or $\alpha_{(a,b)}$, respectively.
The goal of the blending is to retain as much of the high-order method as possible, while increasing the robustness (shock capturing, positivity of the solution, guaranteed entropy dissipation, etc.) where necessary.
In general, the blending coefficients can be computed before or after a time step (explicit RK stage) has been taken.
In this work, we will refer to the former technique as the \textit{a priori} computation and to the latter as the \textit{a posteriori} computation.

\paragraph{\textbf{The \textit{a priori} computation of the blending coefficient}} is perhaps the simplest strategy, as it only requires the evaluation of a troubled cell indicator function, e.g., \cite{Persson2006,Klockner2011}, for the element-wise convex combination strategy; or a nodal feature-based indicator function, e.g., \cite{Fernandez2018,Ciuca2020}, for the subcell-wise convex combination strategy.
This strategy was successfully used by \citet{Hennemann2020} and \citet{Rueda-Ramirez2020} in the context of entropy-stable shock capturing with element-wise hybrid FV/DG methods (which be recast in the context of the present framework) and a modification of the shock indicator of \citet{Persson2006}. The general idea is to measure the modal energy for a given indicator quantity $\epsilon$, e.g., density, or pressure. We transform the quantity $\epsilon$ from (collocated) nodal space to hierarchical Legendre space with modal coefficients $\{\hat{\epsilon}\}_{j=0}^N$. The modal energy of the highest (and the second highest) modes are compared to the total modal energy of the polynomial
\begin{equation}
\mathbb{E} = \max\left(\frac{\hat{\epsilon}_N^2}{\sum_{j=0}^{N} \hat{\epsilon}_j^2}, \frac{\hat{\epsilon}_{N-1}^2}{\sum_{j=0}^{N-1} \hat{\epsilon}_j^2}\right).
\end{equation}
The blending coefficient is then estimated as
\begin{equation}\label{eq:alpha}
\alpha = \frac{1}{1+\exp \left( \frac{-s}{\mathbb{T}}(\mathbb{E}-\mathbb{T})\right)},
\end{equation}
with $s=9.21024$, see \cite{Hennemann2020}, to obtain $\alpha(\mathbb{E}=0)=0.0001$, and
\begin{equation}
\mathbb{T}(N)=0.5 \cdot 10^{-1.8 (N+1)^{0.25}}.
\end{equation}
This blending coefficient can be further refined to increase computational efficiency of the hybrid method. For instance, one can introduce the threshold $\alpha_{max}\in(0,1]$, which limits the maximum amount of allowed FV scheme
\begin{equation}\label{eq:alpha_max}
    \alpha = \min(\alpha, \alpha_{max}).
\end{equation}
We refer to \cite{Hennemann2020,Rueda-Ramirez2020} for details and other modifications.

The two main advantages of the \textit{a priori} technique are that it only requires information from the previous time step, and that it allows for the most efficient use of the computational resources. For instance, if $\alpha = 0$ uniformly within an element, only the DG operator has to be computed, whereas if $\alpha = 1$, only the FV operator has to be computed. On the other hand, the main drawback of the \textit{a priori} approach is that it is typically a heuristic approach with user-chosen parameters.

\paragraph{\textbf{The \textit{a posteriori} computation of the blending coefficient}} is a more complicated but versatile strategy, which can be used to impose bounds to the numerical solution, e.g., positivity of the solution or general discrete maximum principles, entropy inequalities, etc.
The \textit{a posteriori} technique has its roots in the flux corrected transport (FCT) methods, e.g., \cite{boris1976flux,zalesak1979fully,lohner1987finite,kuzmin2010failsafe}, where the main idea is to start from a bounds-preserving low-order method and carefully blend in as much as possible of a less dissipative operator, which is not necessarily bounds preserving, such that blended solution is bounds-satisfying.
In the terminology of FCT methods, the application of a low-order bounds-preserving scheme is commonly known as the first or \textit{transport} stage, and the addition of a higher-order operator is commonly known as the second, \textit{antidiffusive} or \textit{flux-correcting} stage.

The \textit{a posteriori} computation requires both a strategy to compute the bounds and a method to compute the blending coefficient from those bounds.
In principle, any physically-relevant bounds can be imposed, as long as the subcell FV method delivers a solution that is within bounds.
For instance, \citet{Rueda-Ramirez2021} proposed a maximum allowable deviation from the FV solution in the context of positivity-preservation for the Euler equations.

It is also common to compute the target bounds for FCT methods from local minimum or maximum principles, e.g., \cite{kuzmin2010failsafe,lohner1987finite},
\begin{equation} \label{eq:lowStencil_bounds}
    \rho_{ij}^{\max} = \max_{k \in \NN (ij)} \rho (\state{u}^{*}_k),
    ~~~~~~
    \rho_{ij}^{\min} = \min_{k \in \NN (ij)} \rho (\state{u}^{*}_k),
\end{equation}
where $\rho(\cdot)$ denotes any quantity that needs to be bounded, $\NN (ij)$ is the low-order stencil,
\begin{equation} \label{eq:lowStencil}
    \NN (ij) := \{ ij, (i-1)j, (i+1)j, i(j-1), i(j+1) \},
\end{equation}
and $\state{u}^{*}$ can be the solution at the previous time step, i.e., $\state{u}_{ij}^{*} = \state{u}_{ij}^{n}$, or the low-order solution in the next time step, i.e., $\state{u}_{ij}^{*} = \state{u}_{ij}^{\FV,n+1}$, when the low-order method is monotonic.

Another popular alternative to compute the bounds comes from the work of \citet{Guermond2016}, who showed that a low-order method with a specific graph Laplacian viscosity operator can be written as a convex combination of auxiliary two-point states in the low-order stencil, often referred to as \textit{bar states}
\begin{equation} \label{eq:barstates}
    \bar{\state{u}}_{(a,b)} = \frac{1}{2} \left( \state{u}_a + \state{u}_b \right) + \frac{1}{2 \lambda_{(a,b)}} \left( \state{f}_a - \state{f}_b \right).
\end{equation}
In the above, $\lambda_{(a,b)}$ is the maximum wavespeed of the one-dimensional Riemann problem computed between two nodes, here $a$ and $b$, along the normal direction at the interface between those nodes.
Bounds for the maximum wavespeed can be computed efficiently using the algorithm developed in \cite{guermond2016fast}.
Moreover, \citet{Guermond2016} showed that the bar states preserve all \textit{convex invariants} of the governing hyperbolic systems (e.g. positivity and entropy conditions), and hence can be used to compute the target bounds, i.e., $\state{u}_{k}^{*} = \bar{\state{u}}_{(ij,k)}$.

Since SSP-RK methods can be written as a convex combination of forward Euler steps, we can compute the blending coefficient that guarantees the target bounds for a simple Euler evolution.
In FCT methods, an Euler step is usually rewritten as
\begin{equation} \label{eq:ForwardEuler}
 \state{u}_{ij}^{n+1} =
 \state{u}_{ij}^{\FV,n+1} +
 \sum_{k \in \NN(ij) \setminus \lbrace ij \rbrace} (1 - \alpha_{(ij,k)} ) \bar{\state{f}}_{(ij,k)},
\end{equation}
where $\state{u}_{ij}^{\FV,n+1}$ is the forward Euler increment for the low order method, and $\bar{\state{f}}_{(ij,k)}$ is the so-called mass-weighted \textit{anti-diffusive} flux between nodes $ij$ and $k$, which is defined as the high-order flux minus the low-order flux.
For instance, the anti-diffusive flux between nodes $ij$ and $(i+1)j$ reads as
\begin{equation}
    \bar{\state{f}}_{(i,i+1)j} := -
    \frac{\Delta t}{J_{ij} \omega_i}
    \left(
    \numfluxb{f}^{\DG}_{(i,i+1)j} - \numfluxb{f}^{\FV}_{(i,i+1)j}
    \right).
\end{equation}

If we blend the operators in an element-wise manner, the computation of the blending coefficient is straight-forward: every element must use the maximum blending coefficient that is determined for each of its nodes independently \cite{Rueda-Ramirez2021,Pazner2020}.
The blending coefficient for linear constraints (i.e., when we want to impose a bound on one of the state quantities) can be obtained directly from \eqref{eq:ForwardEuler} by replacing the element of $\state{u}_{ij}^{n+1}$ that needs to be bounded with its maximum and minimum values.
For nonlinear convex constraints, the blending coefficient is also obtained from \eqref{eq:ForwardEuler}, but a nonlinear problem must be solved; see \cite{Rueda-Ramirez2021} for details.
The solution of the nonlinear problems can be approximated efficiently with a Newton--bisection or Newton--secant algorithm \cite{guermond2019invariant,maier2021efficient}.

If we blend the operators in a subcell-wise manner, more sophisticated techniques are required to compute the local blending coefficient.
For instance, a Zalesak-type limiter \cite{zalesak1979fully,Pazner2020} can be used to compute a provisional blending coefficient for each node that imposes lower and upper bounds on a state quantity (here noted $\rho$),
\begin{align}\label{eq:Zalesak}
    P^+ &= \sum_{k \in \NN(ij) \setminus \lbrace ij \rbrace} \max \left(0, \bar{\state{f}}^{\rho}_{(ij,k)}\right), &
    P^- &= \sum_{k \in \NN(ij) \setminus \lbrace ij \rbrace} \min \left(0, \bar{\state{f}}^{\rho}_{(ij,k)}\right),
    \nonumber \\
    \alpha^+_{ij} &= 1 - \min \left( 1, \frac{\rho^{\max}_{ij} - \rho (\state{u}^{\FV,n+1}_{ij})}{P^+} \right), &
    \alpha^-_{ij} &= 1 - \min \left( 1, \frac{\rho^{\min}_{ij} - \rho (\state{u}^{\FV,n+1}_{ij})}{P^-} \right)
    \nonumber \\
    \tilde{\alpha}_{ij} &= \max (\alpha^-_{ij}, \alpha^+_{ij}).
\end{align}

For convex (nonlinear) constraints, the blending coefficient for each node can be obtained by solving $2d$ nonlinear problems, where $d$ is the number of dimensions of the problem \cite{guermond2019invariant,Pazner2020}.
We find provisional blending coefficients for each node, such that the solution of the provisional Euler increment,
\begin{align}\label{eq:interface_nonlin_update}
    \tilde{\state{u}}_{ij}^{n+1} =
    \state{u}_{ij}^{\FV,n+1} +
    \Gamma (1 - \tilde{\alpha}_{ij} ) \bar{\state{f}}_{(ij,k)},
    ~~~~
    \forall
    k \in \NN(ij) \setminus \lbrace ij \rbrace,
\end{align}
fulfills the bound that is being enforced.
In the above equation, $\Gamma$ is a parameter that depends on the number of dimensions of the problem.
A value of $\Gamma := 2d$ guarantees the fulfillment of the bound \cite{Pazner2020}.
The blending coefficient at each subcell interface can be then computed as the maximum between the provisional coefficients of the two nodes it connects, e.g., $\alpha_{(i,m)j} = \max (\tilde{\alpha}_{ij} ,  \tilde{\alpha}_{mj})$.

In general, the subcell-wise techniques to obtain the blending coefficient are stricter than the element-wise methods since they compute the blending coefficient for each subcell interface independently.
Therefore, they have to guarantee that the bounds are always met in each node, regardless of the blending coefficient in the other subcell interfaces.

\section{On the Hybrid Subcell FV/DG and the Sparse IDP DG Equivalence}\label{sec:equivalence}

In this section, we show that the sparse invariant domain preserving DG method proposed by \citet{Pazner2020} can be equivalently recovered within the collection or proposed subcell limiting strategies by particular choices of the ingredients.

Ingredient (I): \cite{Pazner2020} uses the high-order standard DGSEM scheme on LGL points, which clearly belongs to the present high-order framework as a special case.

Ingredient (III): \cite{Pazner2020} presents either an element-wise or an subcell-wise convex blending. It is straight forward to derive that these blending strategies directly correspond to the ones described in this actual work.

Ingredient (IV): The strategy to compute the blending coefficient used in \cite{Pazner2020} is the \textit{a posteriori} approach based on a Zalesak type limiter to ensure positivity of density and specific (physical) entropy. Furthermore, for DG schemes with high-order polynomials, Pazner added an a priori step based on an element-wise troubled cell indicator by \citet{Persson2006} to pre-assess if the high-order element really needs limiting as the bounds from the bar states of the low order IDP discretization are typically very restrictive and thus resulting in very dissipative solutions.

Ingredient (II): What remains to show is that the low-order IDP method of \cite{Pazner2020} based on local graph Laplacian viscosity is equivalent to the compatible subcell FV scheme on LGL nodes by \citet{Hennemann2020};
see Section \ref{sec:FV}.

\begin{proposition}
The subcell FV discretization on high-order curvilinear grids detailed in Section \ref{sec:FV} is equivalent to the low-order sparse invariant domain preserving scheme introduced by \citet{Pazner2020} when a first-order reconstruction and the Rusanov (also known as local Lax-Friedrichs, or short LLF) flux are used.
\end{proposition}

\begin{proof}
Using the LLF flux and a first-order FV reconstruction (Figure \ref{fig:FV_1st}) in \eqref{eq:FV_fluxes}, the compatible FV discretization \eqref{eq:FV} reads,
\begin{align} \label{eq:FVLLF}
J_{ij} \dot{\state{u}}^{\LLF}_{ij} =&
+ \frac{1}{\omega_i}
\left[
  \frac{\vn_{(i-1,i)j}}{2} \cdot \left( \blocktensor{f}_{(i-1)j} + \blocktensor{f}_{ij} \right)
  - \frac{\norm{\vn_{(i-1,i)j}} \lambda_{(i-1,i)j}}{2} \left( \state{u}_{ij} - \state{u}_{(i-1)j} \right)
\right]
\nonumber\\
&
- \frac{1}{\omega_i}
\left[
  \frac{\vn_{(i,i+1)j}}{2} \cdot \left( \blocktensor{f}_{ij} + \blocktensor{f}_{(i+1)j} \right)
  - \frac{\norm{\vn_{(i,i+1)j}} \lambda_{(i,i+1)j}}{2} \left( \state{u}_{(i+1)j} - \state{u}_{ij} \right)
\right]
\nonumber\\
&
+ \frac{1}{\omega_j}
\left[
  \frac{\vn_{i(j-1,j)}}{2} \cdot \left( \blocktensor{f}_{i(j-1)} + \blocktensor{f}_{ij} \right)
  - \frac{\norm{\vn_{i(j-1,j)}} \lambda_{i(j-1,j)}}{2} \left( \state{u}_{ij} - \state{u}_{i(j-1)} \right)
\right]
\nonumber\\
&
- \frac{1}{\omega_j}
\left[
  \frac{\vn_{i(j,j+1)}}{2} \cdot \left( \blocktensor{f}_{ij} + \blocktensor{f}_{i(j+1)} \right)
  - \frac{\norm{\vn_{i(j,j+1)}} \lambda_{i(j,j+1)}}{2} \left( \state{u}_{i(j+1)} - \state{u}_{ij} \right)
\right],
\end{align}
where $\lambda_{(.,.)}$ is again the maximum wave speed computed between two nodes.

We now define the quantities
\begin{equation} \label{eq:PaznerOperators}
    \cc_{(i,m)j}:= -\frac{\omega_j \vn_{(i,m)j}}{2},
    ~~~~~
    \cc_{i(j,m)}:= -\frac{\omega_i \vn_{i(j,m)}}{2},
    ~~~~~
    \cc_{(ij,ij)}:= \vec{0},
    ~~~~~ \text{and} ~~
    \mm_{ij} := J_{ij} \omega_{ij}.
\end{equation}
By virtue of the discrete subcell metric identities \eqref{eq:metric_ids_subcell}, we remove all occurrences of $\blocktensor{f}_{ij}$ in \eqref{eq:FVLLF}.
Moreover, realizing that the normal vectors are uniquely defined inside each element, $\vn_{(a,b)} = - \vn_{(b,a)}$, we obtain
\begin{align} \label{eq:FVLLF2}
\mm_{ij} \dot{\state{u}}^{\LLF}_{ij} =&
+
\left[
  \cc_{(i,i-1)j} \cdot  \blocktensor{f}_{(i-1)j}
  + \norm{\cc_{(i,i-1)j}} \lambda_{(i-1,i)j} \left( \state{u}_{(i-1)j} - \state{u}_{ij} \right)
\right]
\nonumber\\
&
+
\left[
  \cc_{(i,i+1)j} \cdot  \blocktensor{f}_{(i+1)j}
  + \norm{\cc_{(i,i+1)j}} \lambda_{(i,i+1)j} \left( \state{u}_{(i+1)j} - \state{u}_{ij} \right)
\right]
\nonumber\\
&
+
\left[
  \cc_{i(j,j-1)} \cdot  \blocktensor{f}_{i(j-1)}
  + \norm{\cc_{i(j,j-1)}} \lambda_{i(j-1,j)} \left( \state{u}_{i(j-1)} - \state{u}_{ij} \right)
\right]
\nonumber\\
&
+
\left[
  \cc_{i(j,j+1)} \cdot  \blocktensor{f}_{i(j+1)}
  + \norm{\cc_{i(j,j+1)}} \lambda_{i(j,j+1)} \left( \state{u}_{i(j+1)} - \state{u}_{ij} \right)
\right].
\end{align}

Equation \eqref{eq:FVLLF2} can be further simplified by using the definition of the low-order stencil \eqref{eq:lowStencil}, defining the graph viscosity coefficients, $\dd_{(a,b)} := \norm{\cc_{(a,b)}} \lambda_{(a,b)}$, and re-indexing, $ij \rightarrow i$, to obtain
\begin{equation}
    \mm_i \dot{\state{u}}^{\LLF}_{i} = \sum_{j \in \NN(i)} \cc_{ij} \cdot \blocktensor{f}_{j} + \sum_{j \in \NN(i)} \dd_{ij} \left( \state{u}_j - \state{u}_i \right),
\end{equation}
which is exactly the sparse invariant domain preserving scheme of \citet[eq. (35)]{Pazner2020}.
\end{proof}

The factor $1/2$ in operator $\cc$ \eqref{eq:PaznerOperators} corresponds to the entries of the sparsified derivative matrix,
\begin{equation} \label{eq:didp}
     \Didp_{1D} :=
     \left(\begin{array}{ccccccc}
          -\frac{1}{2} & \frac{1}{2} & 0 & 0 & 0 & \cdots & 0 \\
          -\frac{1}{2} & 0 & \frac{1}{2} & 0 & 0 & \cdots & 0 \\
          0 & -\frac{1}{2} & 0 & \frac{1}{2} & 0 & \cdots & 0 \\
          \vdots & \vdots & \vdots & \vdots & \vdots & \ddots & \vdots
     \end{array}\right),
\end{equation}
introduced by \citet{Pazner2020} to construct the low-order IDP scheme.
Therefore, the other terms in $\cc$ must correspond to Pazner's \textit{modified metric terms} for the FV-LLF and low-order IDP schemes to be equivalent.
In \cite{Pazner2020}, the modified metric terms are computed by perturbing the contravariant basis vectors and solving a minimization problem with a singular value decomposition, such that the resulting low-order scheme fulfills a conservative condition.

\begin{proposition} \label{lemma:SubcellMetrics}
The subcell metrics \eqref{eq:SubcellMetrics} introduced by \citet{Hennemann2020} are exact analytical solutions to the conservative condition to the modified metric terms of the low-order IDP method of \citet[Section 3.1]{Pazner2020}.
\end{proposition}

\begin{proof}
The conservative condition to the modified metric terms of the low-order IDP method can be rewritten in terms of the operator $\cc$ as \cite[Proposition 5]{Pazner2020}
\begin{equation} \label{eq:FSP_Pazner}
    \sum_{j \in \NN(i)} \cc_{ij} = \vec{0}.
\end{equation}

Inserting the definition of $\cc$ \eqref{eq:PaznerOperators} into \eqref{eq:FSP_Pazner} and re-indexing back, we recover a weighted version of the discrete subcell metric identities,
\begin{equation}
  \sum_{k \in \NN(ij)} \cc_{(ij)k} =
  \frac{\omega_{ij}}{2}
  \left(
    \frac{ \vn_{(i-1,i)j} - \vn_{(i,i+1)j} }{\omega_i}
  + \frac{ \vn_{i(j-1,j)} - \vn_{i(j,j+1)} }{\omega_j}
  \right)
  = \vec{0}.
\end{equation}

\end{proof}

By virtue of Proposition \ref{lemma:SubcellMetrics}, it is not necessary to solve an underdetermined system of equations (e.g., using a singular value decomposition) to obtain the modified metric terms.
It suffices to evaluate the exact analytical subcell metric terms as described in equation \eqref{eq:SubcellMetrics}.

\section{Construction of Subcell Limiting Strategies for DGSEM}\label{sec:options}

The strategies enumerated in Section \ref{sec:strategies} generate a family of potential methods that can be combined in a flexible manner, depending on one's specific needs.
Since in Section \ref{sec:equivalence} we further established that the entropy-stable hybrid FV/ DG methods and the sparse IDP DG schemes belong to this unified framework, it is possible to borrow ideas from one approach and transfer it to the other method.
For instance, the following options are possible and have been investigated by the authors:

\begin{enumerate}
    \item It is possible to construct sparse IDP DG methods and use split-form DG volume integral operators to enhance robustness of the high-order discretization, e.g., by guaranteeing entropy dissipation and/or kinetic energy dissipation. We will demonstrate in the numerical results section that an entropy-stable split-form high-order DG scheme ensures convergence to the unique entropic solution of the KPP problem. In general, it is further interesting to note that a robust split-form DG high-order scheme is desirable as it helps to reduce the amount of necessary blending stabilization: \citet{Rueda-Ramirez2021} showed that the split-form DGSEM in combination with the entropy-stable flux of \citet{Chandrashekar2013} and a subcell FV method is more stable and requires less stabilization (a smaller amount of low-order method) than the hybrid standard DGSEM/FV for simulations when applied to problems with under-resolved turbulence and shocks.

    \item It is possible to use hybrid FV/DG methods for shock capturing, cf.~\cite{Hennemann2020,Rueda-Ramirez2020}, with a \textbf{subcell-wise limiting} strategy, where an element local shock indicator has to be used to determine the local blending coefficients within each element. However, we remark, that guaranteed entropy dissipation of the semi-discretization has only been proven for the element-wise blending of entropy-stable DG and FV methods (in \cite{Hennemann2020} and \cite{Rueda-Ramirez2020} for the Euler and MHD equations, respectively) and that the proof for subcell-wise blending is still missing.

    \item Since the low-order sparse IDP method can be interpreted as a subcell FV scheme with LLF flux, it is possible to exchange the numerical flux function with more accurate numerical fluxes. There are a wide variety of numerical fluxes in the FV literature, some of which can enforce certain desirable properties with reduced dissipation compared to the LLF scheme.
    For instance, the HLLE \cite{Einfeldt1988} and HLLEM \cite{einfeldt1991godunov} schemes preserve positivity of density and pressure for the Euler equations of gas dynamics.
    However, it is important to compute target bounds that are consistent with the low-order method used.
    For instance, the bar states of the IDP approach are specific to the LLF-type graph Laplacian dissipation only. Hence, these bar states and their bounds do of course not directly carry over to other numerical Riemann solvers.

    \item It is possible to apply \textbf{high-resolution FV reconstruction procedures} within the convex blending strategies. For instance, a second-order reconstruction procedure (such as the one illustrated in Figure \ref{fig:FV_2nd}) on primitive variables can guarantee monotonic solutions, which can be used to enforce positivity of the solution.
\end{enumerate}

\section{Numerical experiments}\label{sec:results}

In this section, we apply various subcell limiting strategies for the high-order DGSEM scheme for different hyperbolic PDEs. We provide examples of the so-called KPP problem, the compressible Euler equations of gas dynamics, and the ideal MHD equations. The KPP problem is simulated with the open source Julia code Trixi.jl (\url{github.com/trixi-framework/Trixi.jl}), e.g., \cite{ranocha2022adaptive,schlottkelakemper2021purely} while the rest of the numerical results are generated with the open source Fortran code Fluxo (\url{github.com/project-fluxo/fluxo}). All examples make use of the third-order three-stages SSP-RK method introduced by \citet{shu1988efficient} with a typical explicit timestep determined by $\textrm{CFL}=0.5$, the element geometry, and the maximum wave speed.

\subsection{KPP Problem}

We consider the so-called KPP problem, first introduced by Kurganov, Petrova, and Popov in \cite{kurganov2007adaptive}.
This two-dimensional conservation law features a non-convex flux function, and is given by
\begin{equation} \label{eq:kpp}
    \bigpartialderiv{u}{t} + \bigpartialderiv{\sin(u)}{x} + \bigpartialderiv{\cos(u)}{y} = 0.
\end{equation}
The spatial domain is taken to be $[-2,2]^2$, and the initial condition is given by the piecewise constant function
\begin{equation} \label{eq:kpp-ic}
    u_0(x,y) = \begin{cases}
        7 \pi / 2, &\qquad x^2 + y^2 \leq 1,\\
        \pi/4, &\qquad \text{otherwise.}
    \end{cases}
\end{equation}
This problem is challenging for many high-order numerical methods, which do not correctly resolve the composite wave structure, and may converge to incorrect (nonentropic) solutions.
Even monotone methods that result in non-oscillatory, maximum-principle-satisfying solutions may converge to nonentropic solutions if they fail to satisfy discrete entropy inequalities \cite{Guermond2016,Kuzmin2021}.
Therefore, it is natural to make use of discretizations that satisfy discrete entropy inequalities to guarantee convergence of the method to the unique entropy solution.
Although DG discretizations of scalar conservation laws satisfy a cell entropy inequality for the square entropy \cite{Jiang1994}, this requires exact computation of the volume and surface integrals, which may not be feasible for the non-polynomial flux functions of \eqref{eq:kpp}.

In this section we consider the effect of using an entropy-stable DGSEM discretization with flux differencing to discretize this problem.
Consider the square entropy $U = \frac{1}{2} u^2$.
The entropy variables for this entropy are the same as the conserved variables, i.e., $v = u$.
The unique two-point entropy conservative flux is given by
\begin{equation}
    \numfluxb{f}_{\ec}(u_L, u_R) = \begin{cases}
        \frac{\stateG{\psi}_R - \stateG{\psi}_L}{v_R - v_L} &\qquad u_L \neq u_R,\\
        \state{f}(u_L) &\qquad u_L = u_R,
    \end{cases}
\end{equation}
where $\stateG{\psi}$ is the potential flux given by
\begin{equation}
    \stateG{\psi} = \begin{bmatrix}
        -\cos(u) \\
        \sin(u)
    \end{bmatrix}.
\end{equation}
The split-form DGSEM method using this two-point flux for the volume terms, and the local Lax--Friedrichs flux for surface integrals will satisfy the discrete entropy inequality, and will therefore converge to the unique entropic solution.

\begin{figure}[h!]
    \centering
    \includegraphics[trim=0 1600 1440 0 ,clip,width=0.4\linewidth]{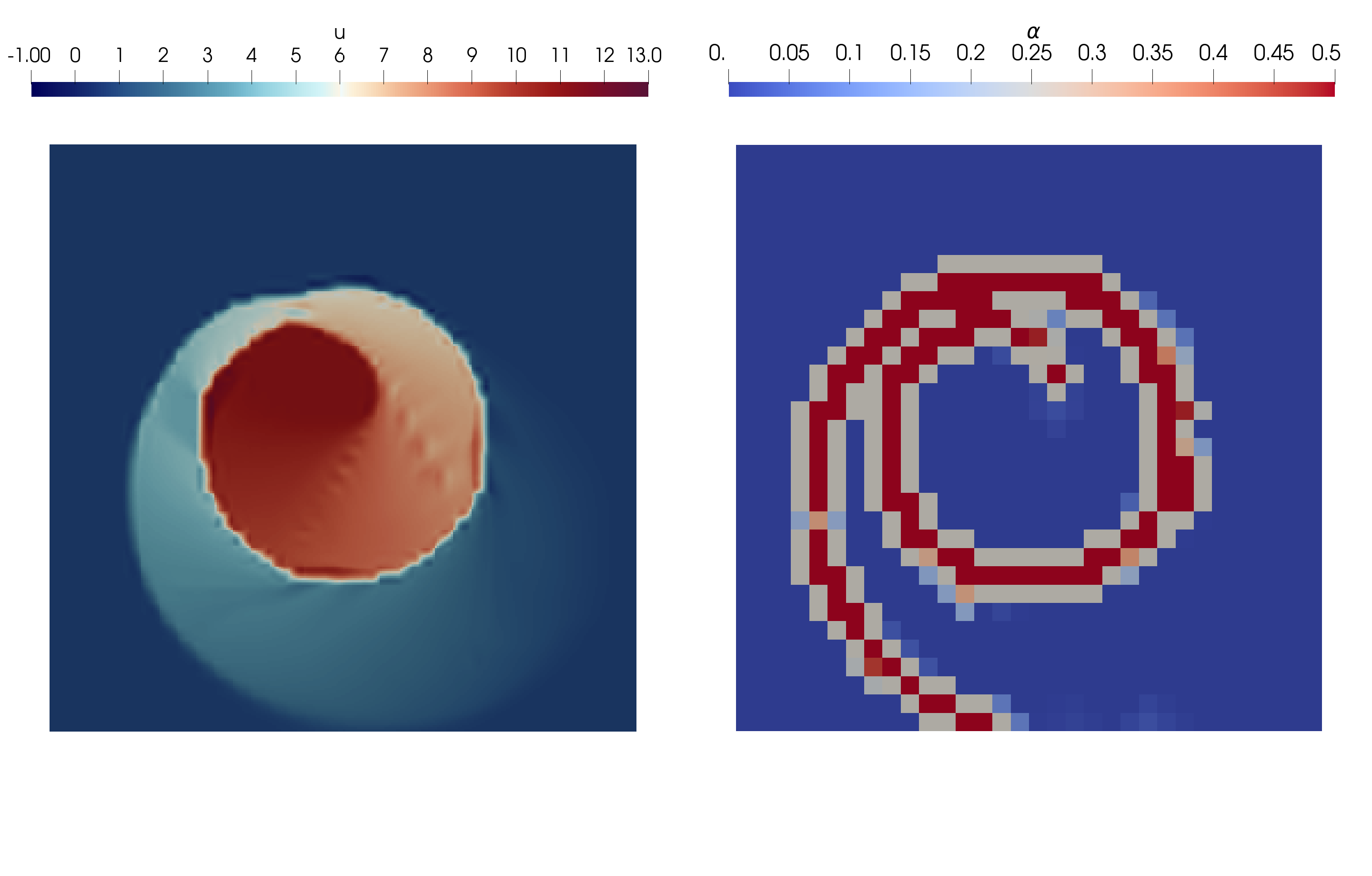}
\begin{subfigure}[b]{\linewidth}
    \centering
	\includegraphics[trim=0 300 1440 300 ,clip,width=0.28\linewidth]{figs/KPP/Standard_DG_grid_0000.png}
	\hspace{10pt}
	\includegraphics[trim=0 300 1440 300 ,clip,width=0.28\linewidth]{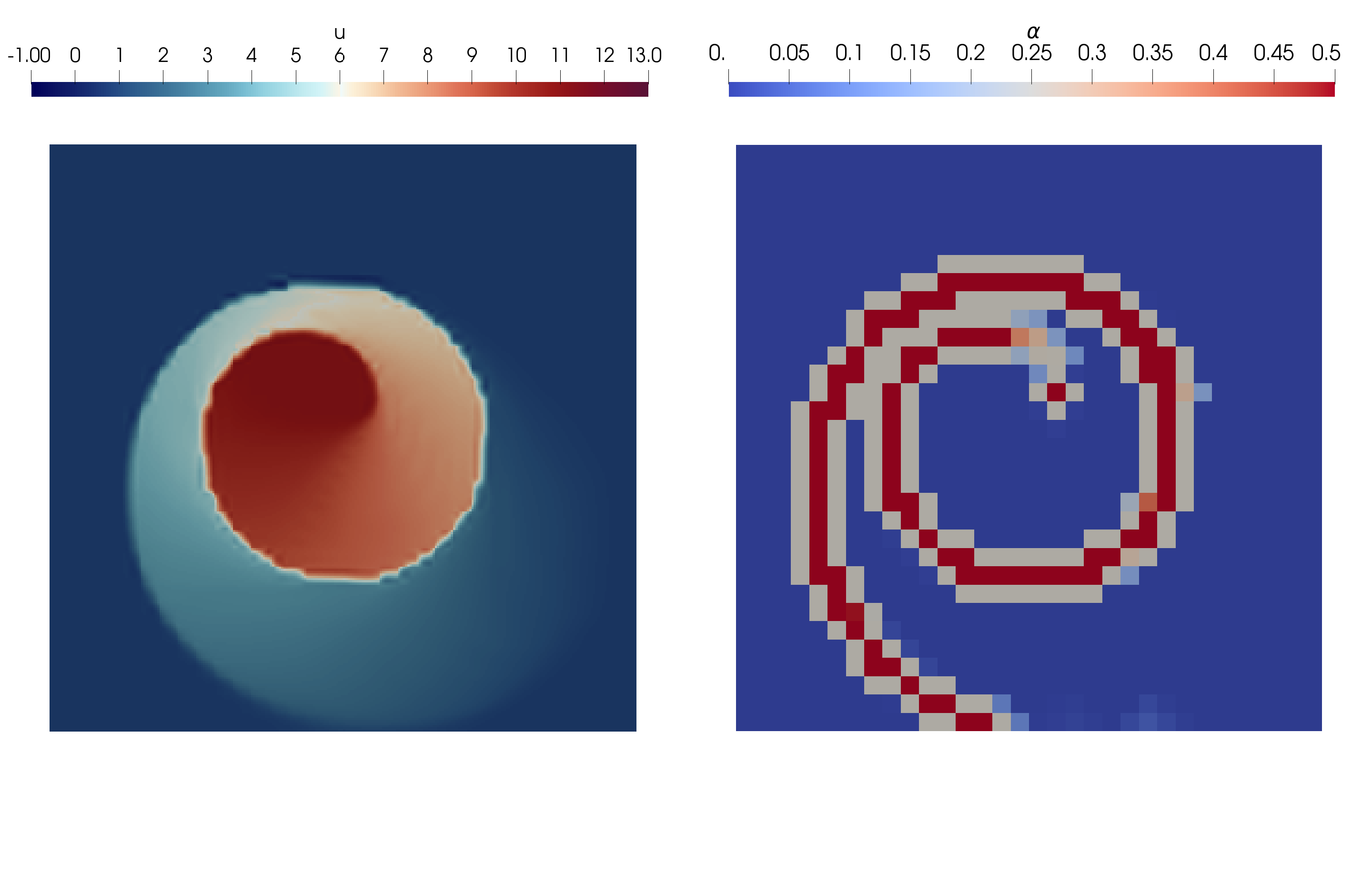}
	\caption{$32 \times 32$ elements mesh}
	\label{fig:kpp_grid0}
\end{subfigure}
\begin{subfigure}[b]{\linewidth}
    \centering
	\includegraphics[trim=0 300 1440 300 ,clip,width=0.28\linewidth]{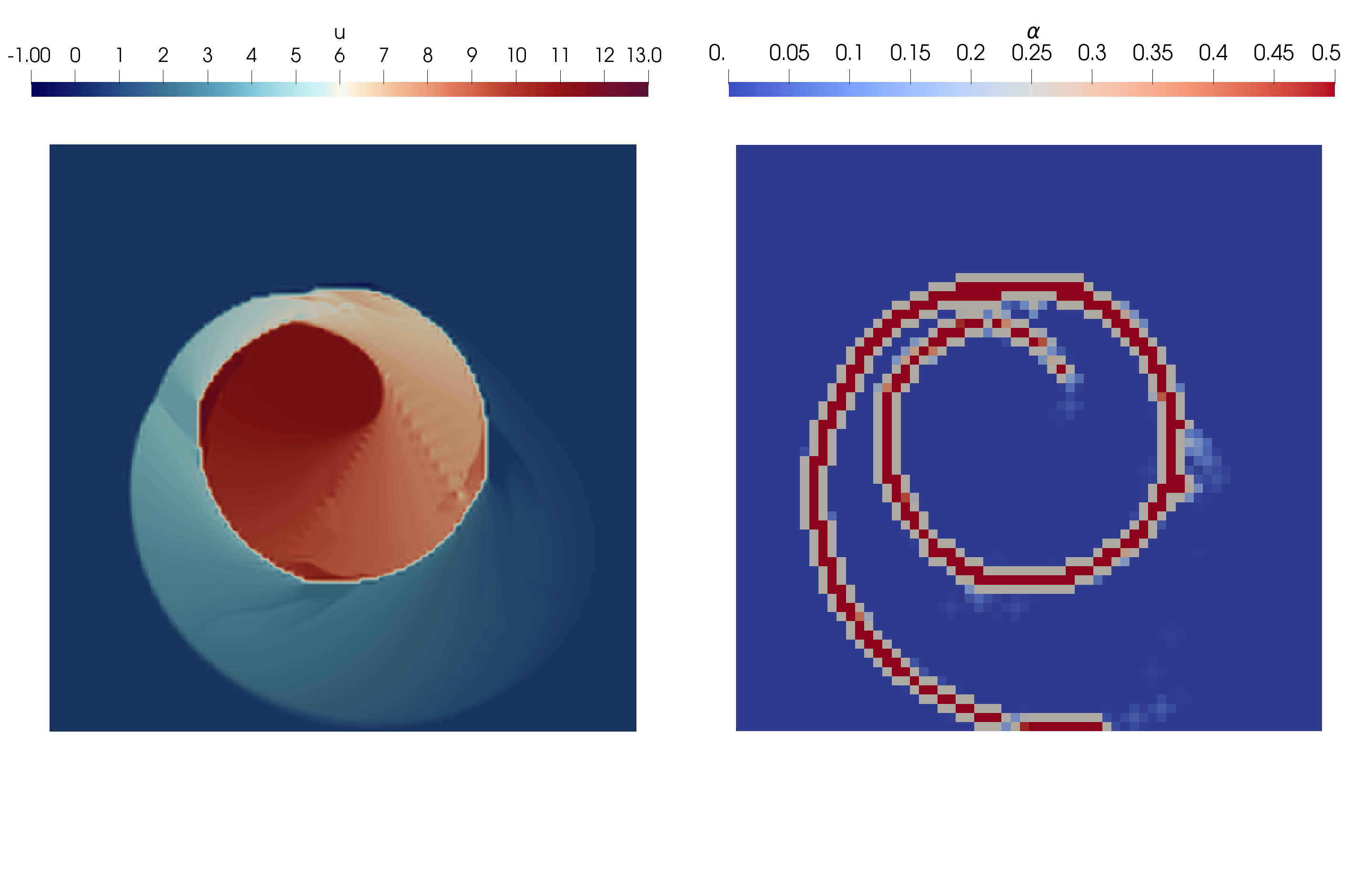}
	\hspace{10pt}
	\includegraphics[trim=0 300 1440 300 ,clip,width=0.28\linewidth]{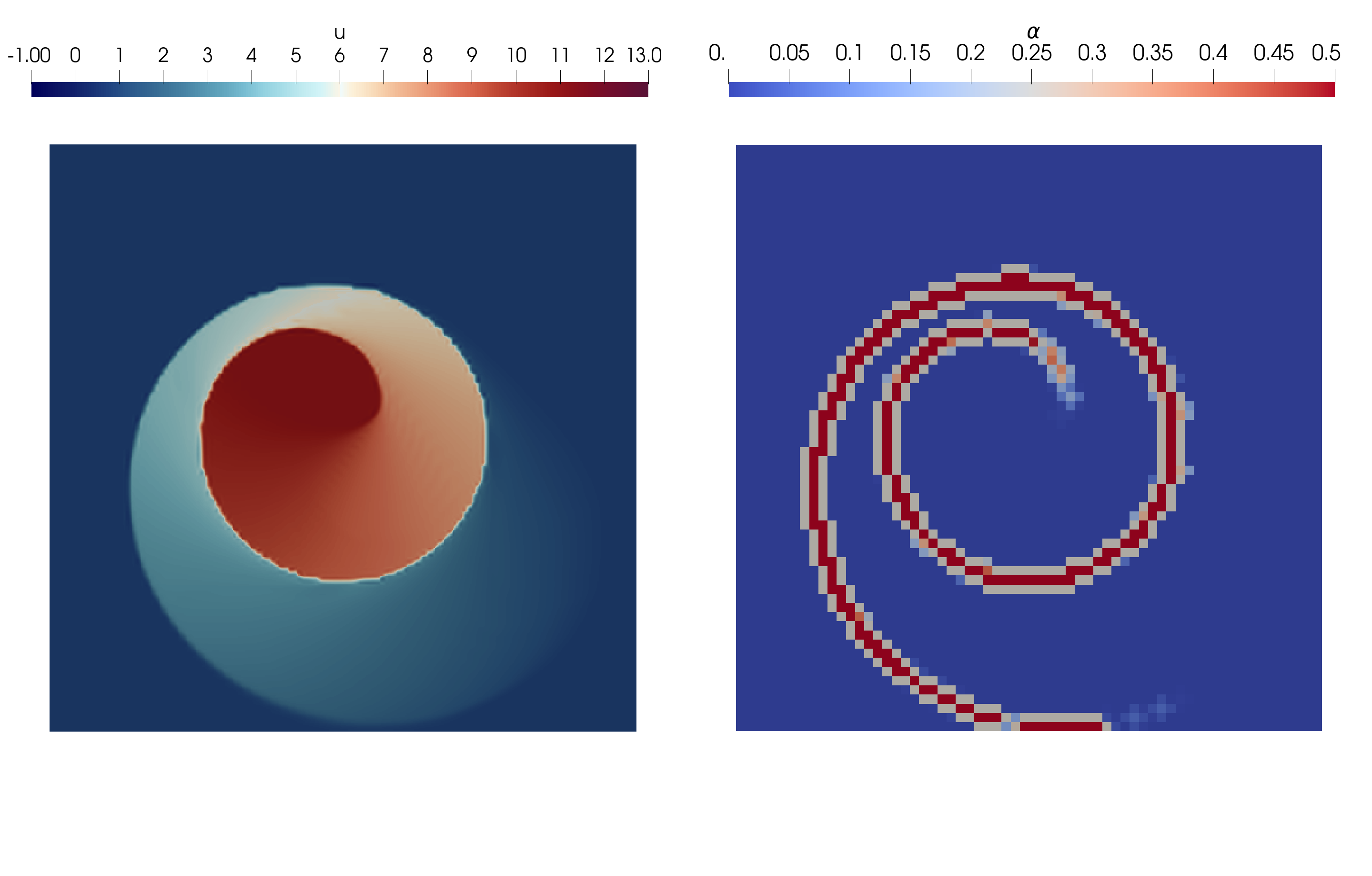}
	\caption{$64 \times 64$ elements mesh}
	\label{fig:kpp_grid1}
\end{subfigure}
\begin{subfigure}[b]{\linewidth}
    \centering
	\includegraphics[trim=0 300 1440 300 ,clip,width=0.28\linewidth]{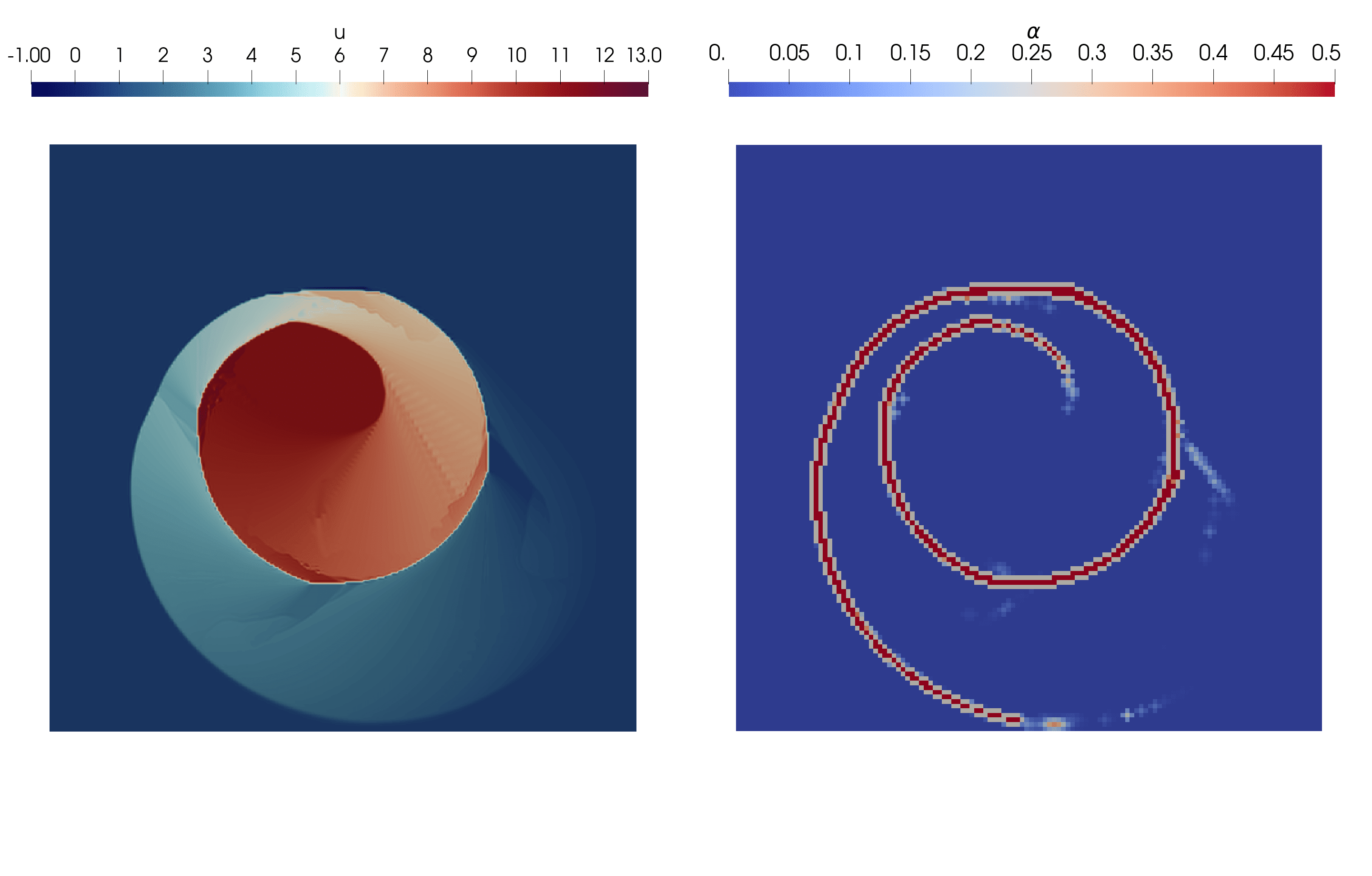}
	\hspace{10pt}
	\includegraphics[trim=0 300 1440 300 ,clip,width=0.28\linewidth]{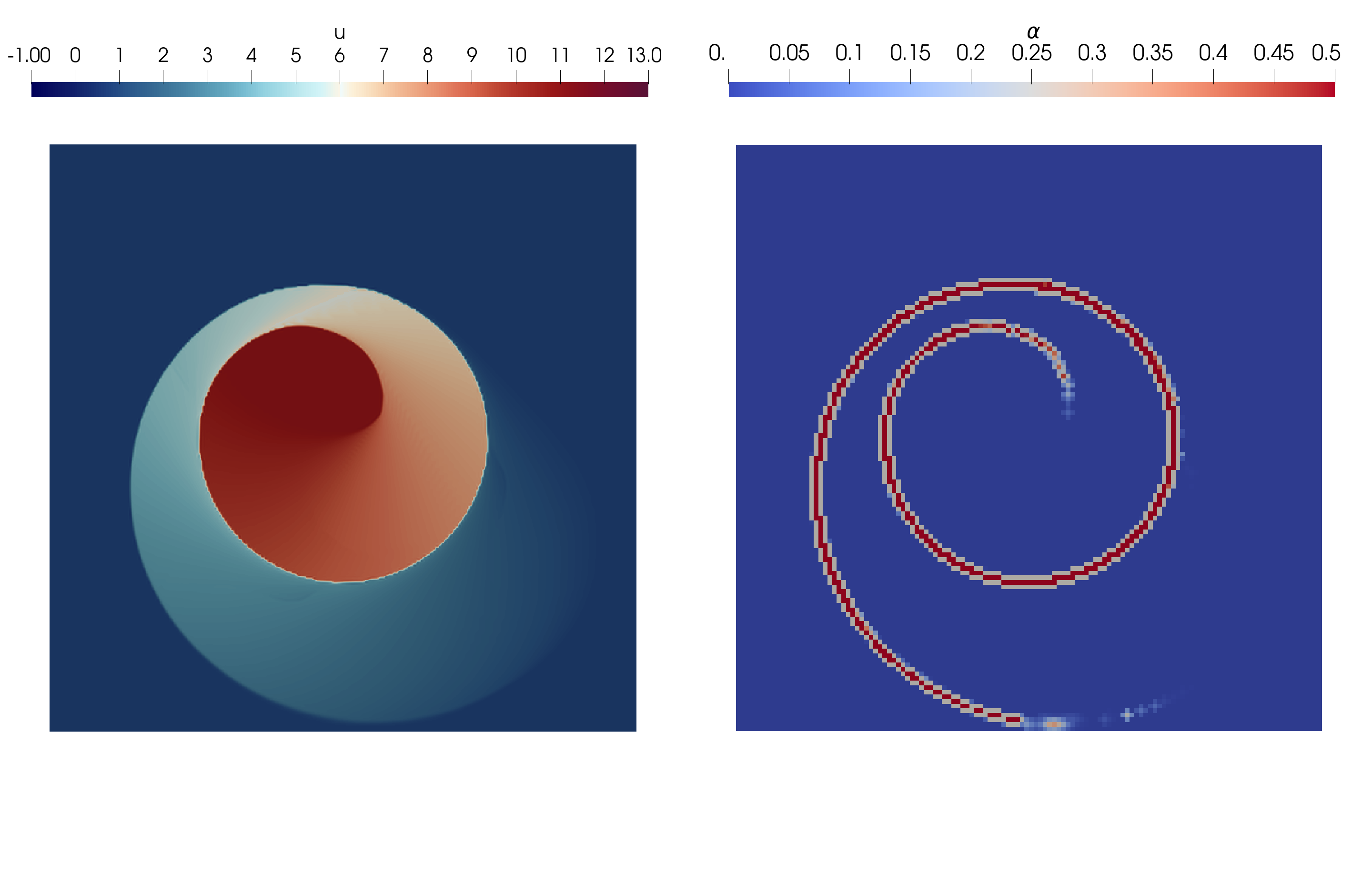}
	\caption{$128 \times 128$ elements mesh}
	\label{fig:kpp_grid2}
\end{subfigure}
\begin{subfigure}[b]{\linewidth}
    \centering
	\includegraphics[trim=0 300 1440 300 ,clip,width=0.28\linewidth]{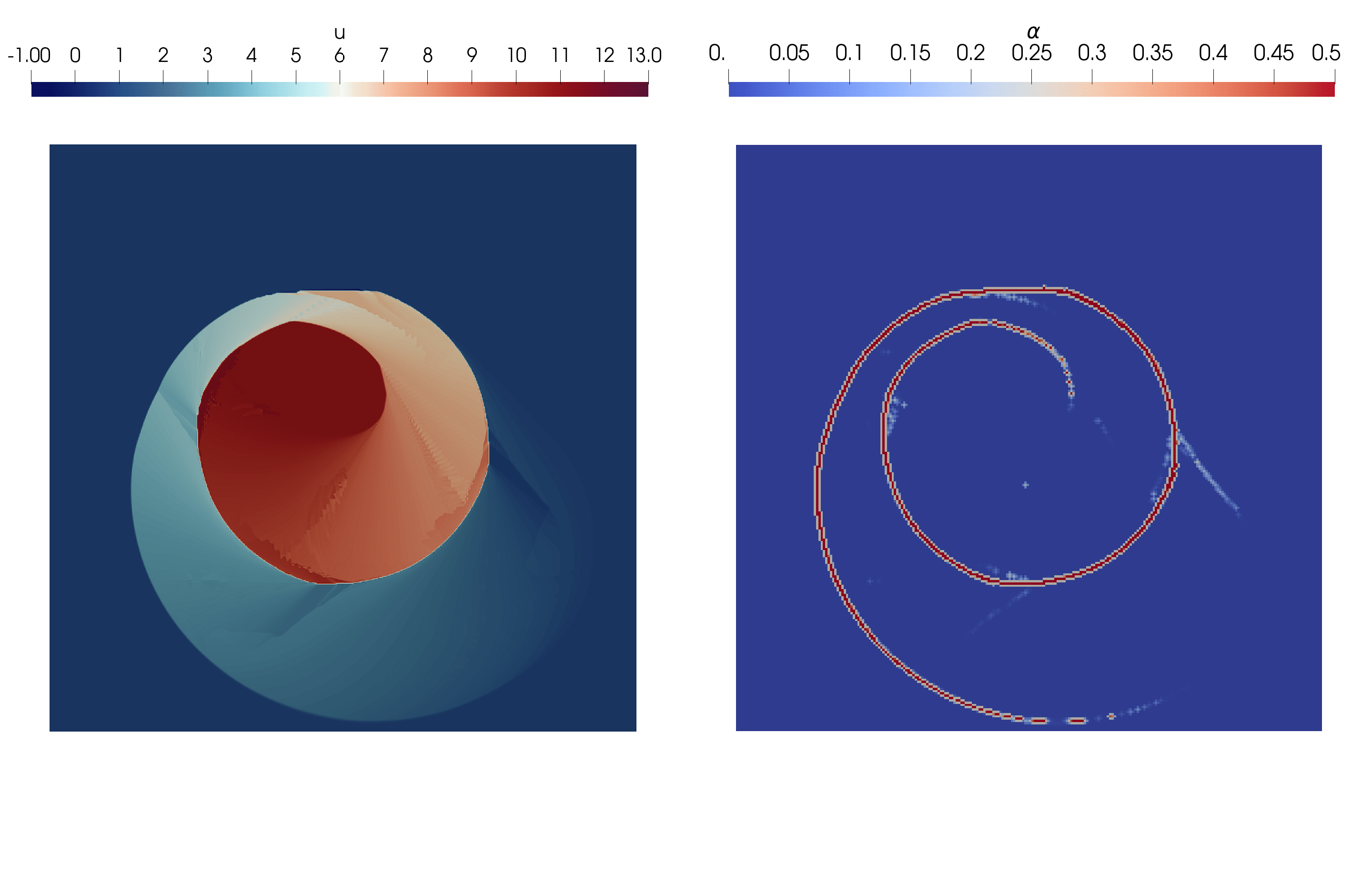}
	\hspace{10pt}
	\includegraphics[trim=0 300 1440 300 ,clip,width=0.28\linewidth]{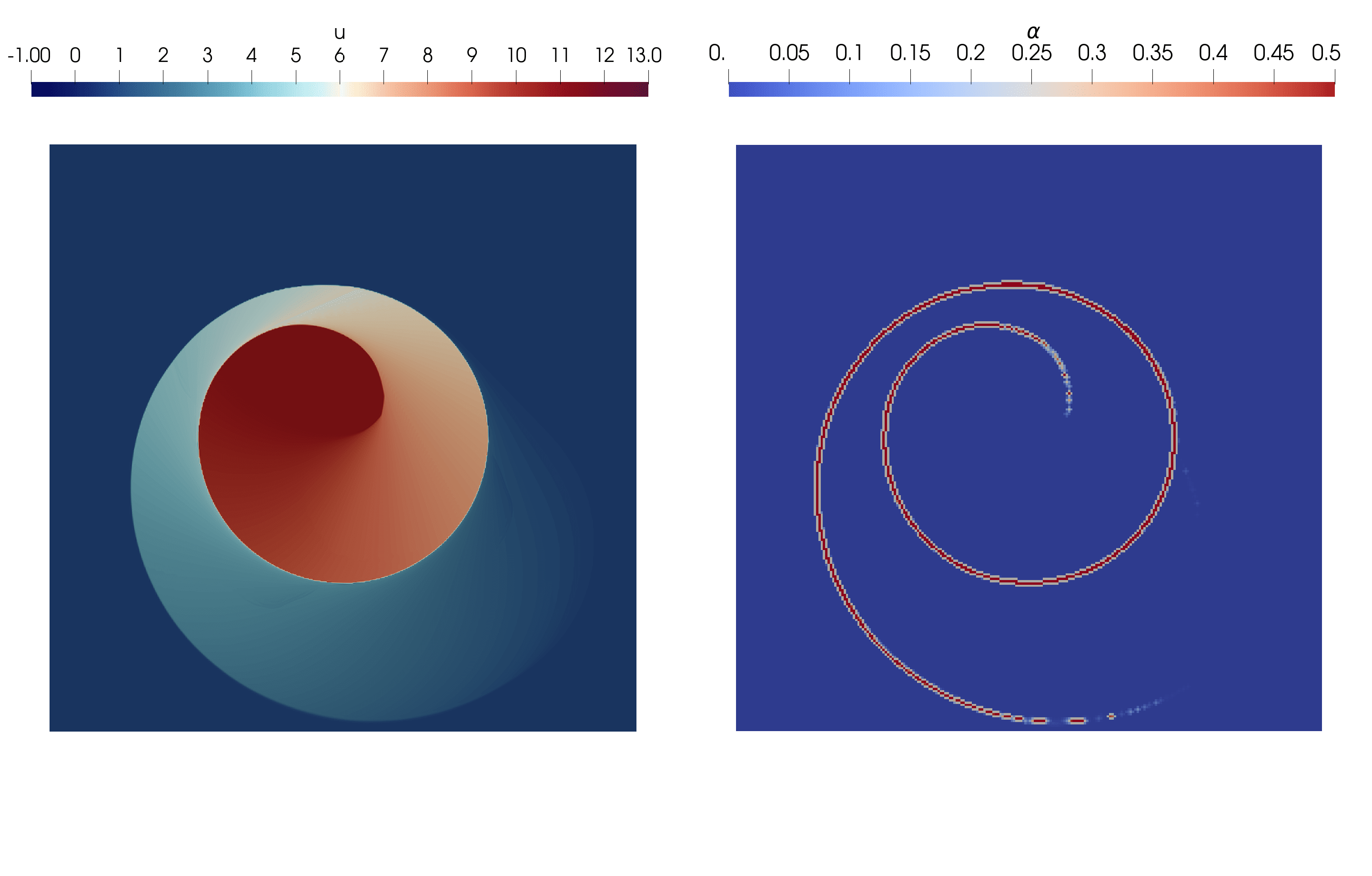}
	\caption{$256 \times 256$ elements mesh}
	\label{fig:kpp_grid3}
\end{subfigure}
    \caption{Comparison of solution to the KPP problem using the standard DGSEM method (left) and the entropy-stable DGSEM (right) for a sequence of different grids. The polynomial degree $N=3$ and the end time $t=1$. The standard DGSEM methods converges to a nonentropic solution, whereas the entropy-stable DGSEM converges to the correct entropy solution.}
    \label{fig:kpp}
\end{figure}

We discretize this problem using the DGSEM method, with the element-wise entropy-stable shock capturing scheme of \cite{Hennemann2020} based on the indicator Eq.\eqref{eq:alpha} and \eqref{eq:alpha_max} with the maximum alpha threshold $\alpha_{max} = 0.5$, allowing at most 50\% FV mixed in. The equations are integrated until $t=1$ using the initial condition given by \eqref{eq:kpp-ic}. We generate a sequence of uniform Cartesian meshes, from $32\times 32$ elements up to $256\times 256$. The polynomial degree $N=3$. The solutions obtained using the standard DGSEM method and the entropy-stable DGSEM method are shown in figure \ref{fig:kpp}. Although the element-wise shock capturing scheme is sufficient to give clean essentially nonoscillatory solutions, the standard DGSEM method, which does not satisfy a discrete entropy inequality, converges to an incorrect, nonentropic solution. The entropy-stable DGSEM method converges to the entropy solution, and correctly resolves the wave structures with sharp interfaces at the discontinuities.

Finally, figure \ref{fig:kpp_alpha} shows the distribution of the element-wise blending coefficients $\alpha$ that nicely paint along the discontinuous flow features. It is also interesting to note that erroneous wave structures from the nonentropic solution of the standard scheme trigger the shock capturing, which is not the case in the entropy-stable DGSEM result.

\begin{figure}[h!]
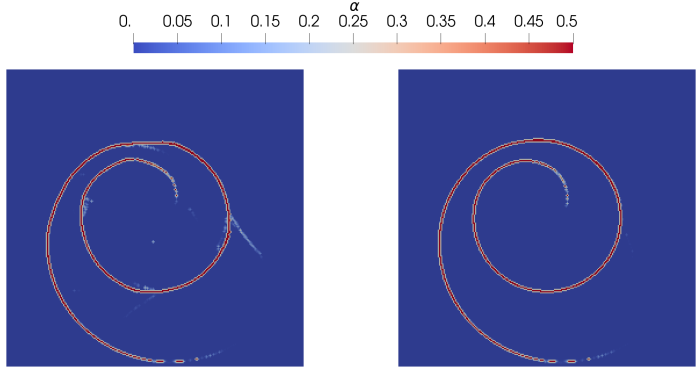

    \centering
    \includegraphics[trim=1440 1600 0 0 ,clip,width=0.4\linewidth]{figs/KPP/Standard_DG_grid_0000.png}

	\includegraphics[trim=1440 300 0 300 ,clip,width=0.28\linewidth]{figs/KPP/Standard_DG_grid_0003.png}
	\hspace{10pt}
	\includegraphics[trim=1440 300 0 300 ,clip,width=0.28\linewidth]{figs/KPP/ES_DG_grid_0003.png}

    \caption{Comparison of the blending coefficient distribution for the KPP problem using the standard DGSEM method (left) and entropy-stable DGSEM (right) using a $256 \times 256$ elements grid with polynomial degree $N=3$ at time $t=1$. The shock capturing sensor of \cite{Hennemann2020} based on the equations \eqref{eq:alpha} and \eqref{eq:alpha_max} with the maximum alpha threshold $\alpha_{max} = 0.5$ is used for the indicator quantity $u$.}
    \label{fig:kpp_alpha}
\end{figure}

\subsection{The Euler Equations of Gas Dynamics} \label{sec:NumEuler}

The Euler equations of gas dynamics can be written as a system of conservation laws, where the conserved variables are the mass, momentum and total energy per unit volume, $\state{u} = [\rho, \rho \vec{v}, \rho E]^T$.
The flux function reads
\begin{equation}
\blocktensor{f} (\mathbf{u})=
\begin{bmatrix}
\rho \vec{v} \\
\rho \vec{v} \otimes \vec{v} + \mat{I} p \\
\vec{v} (\rho E + p)
\end{bmatrix},
\end{equation}
where $\mat{I}$ is the identity matrix, the pressure is computed with the calorically perfect gas assumption,
\begin{equation} \label{eq:press}
p=(\gamma-1) \rho e,
\end{equation}
$\gamma$ is the heat capacity ratio, and $e = E - \norm{\vec{v}}^2/2$ is the internal energy.

\subsubsection{Kelvin-Helmholtz Instability}

We consider the inviscid two-dimensional Kelvin-Helmholtz instability (KHI) setup, e.g., presented by \citet{Rueda-Ramirez2021}.
This setup is very challenging for nodal high-order DG methods, as it contains severely under-resolved vortical structures with an effective Reynolds number Re $=\infty$ and a Mach number of up to Ma $\approx 0.6$.
As a result, both the standard and the split-form DGSEM methods require the use of a positivity-preserving limiter to ensure robustness.

The aim of this test is to apply the positivity technique of \citet{Rueda-Ramirez2021} in a subcell-wise manner.
More specifically, we impose lower bounds for density and pressure that depend on the FV solution,
\begin{equation} \label{eq:rhop_restriction}
\rho_{ij} \ge \beta \rho^{\FV}_{ij},
~~~~
p_{ij} \ge \beta p^{\FV}_{ij},
\end{equation}
with $\beta = 0.1$ (which is more stringent requirement than just positivity).
We compute local blending coefficients that ensure the fulfillment of condition \eqref{eq:rhop_restriction} in the entire simulation domain.
This requires the use of a one-sided Zalesak-type limiter for the density limiting (i.e., only $\alpha^-$ is computed in \eqref{eq:Zalesak}) and the solution of interface non-linear equations for the pressure limiting (as in \eqref{eq:interface_nonlin_update}).

The initial condition is given by
\begin{align} \label{eq:KHI_IniCond}
\rho_0 (x,y) &= \frac{1}{2}
+ \frac{3}{4} B,
&
p_0 (x,y) &= 1,
\nonumber \\
v_{1,0} (x,y) &= \frac{1}{2} \left( B-1 \right),
&
v_{2,0} (x,y)) &= \frac{1}{10} \sin(2 \pi x),
\end{align}
with
$
B=
\tanh \left( 15 y + 7.5 \right) - \tanh(15y-7.5).
$

The simulation domain, $[-1,1]^2$, is equipped with periodic boundary conditions.
We tessellate the domain using $64 \times 64$ quadrilateral elements, represent the solution with polynomials of degree $N=7$, and run the simulation until the final time $t=10$.
We discretize the Euler equations using the split-form DGSEM, the entropy-conserving and kinetic energy preserving flux of \citet{Chandrashekar2013} for the volume numerical fluxes, $\state{f}^*$, and the traditional Rusanov scheme for the surface numerical fluxes, $\numfluxb{f}$.

\begin{figure}[htb]
\centering

\begin{subfigure}[b]{0.48\linewidth}
    \centering
 	\includegraphics[trim=0 0 0 0 ,clip,width=\linewidth]{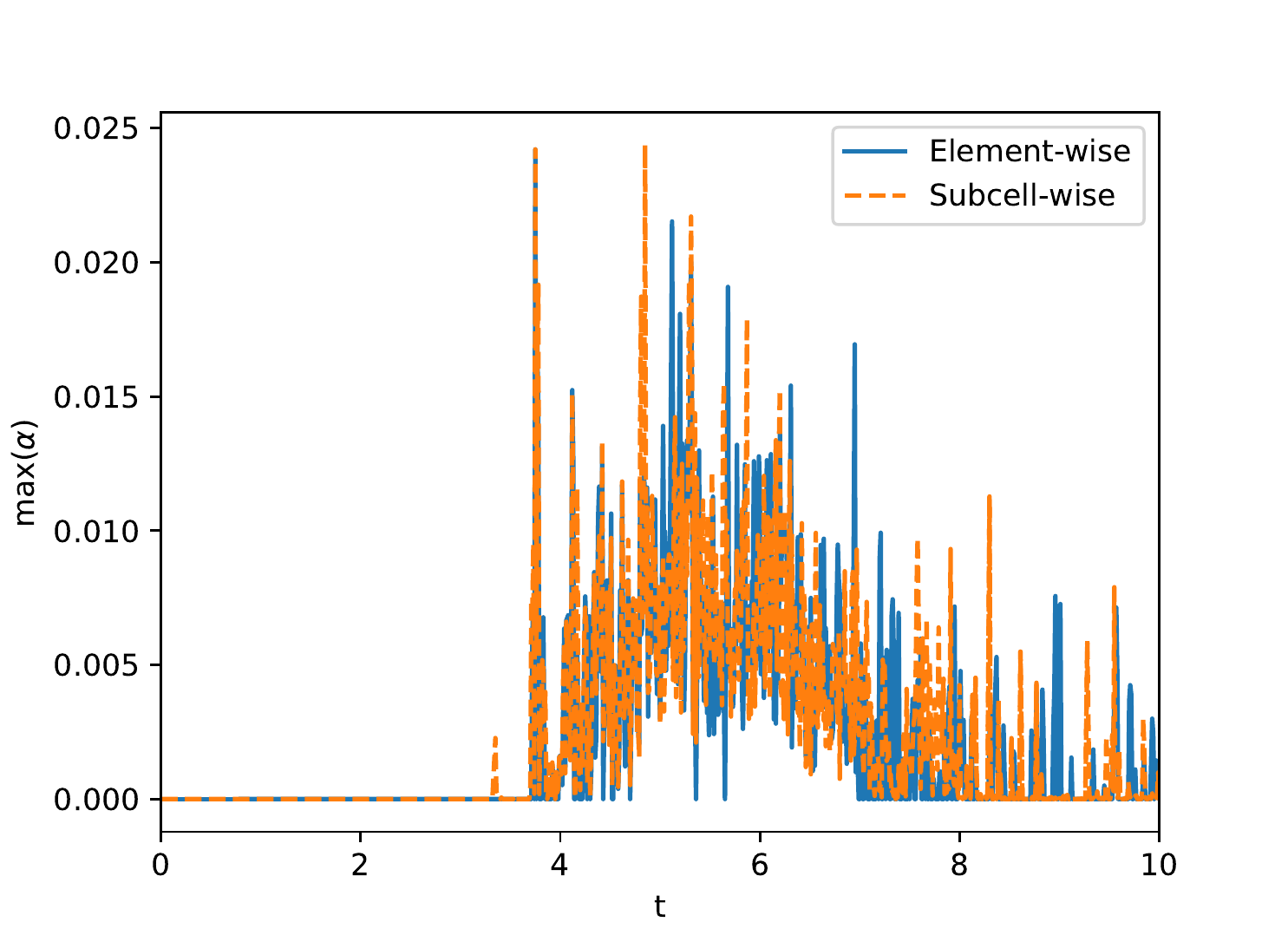}
	\caption{Maximum blending coefficient}
	\label{fig:khi_alphamax}
\end{subfigure}
\begin{subfigure}[b]{0.48\linewidth}
    \centering
 	\includegraphics[trim=0 0 0 0 ,clip,width=\linewidth]{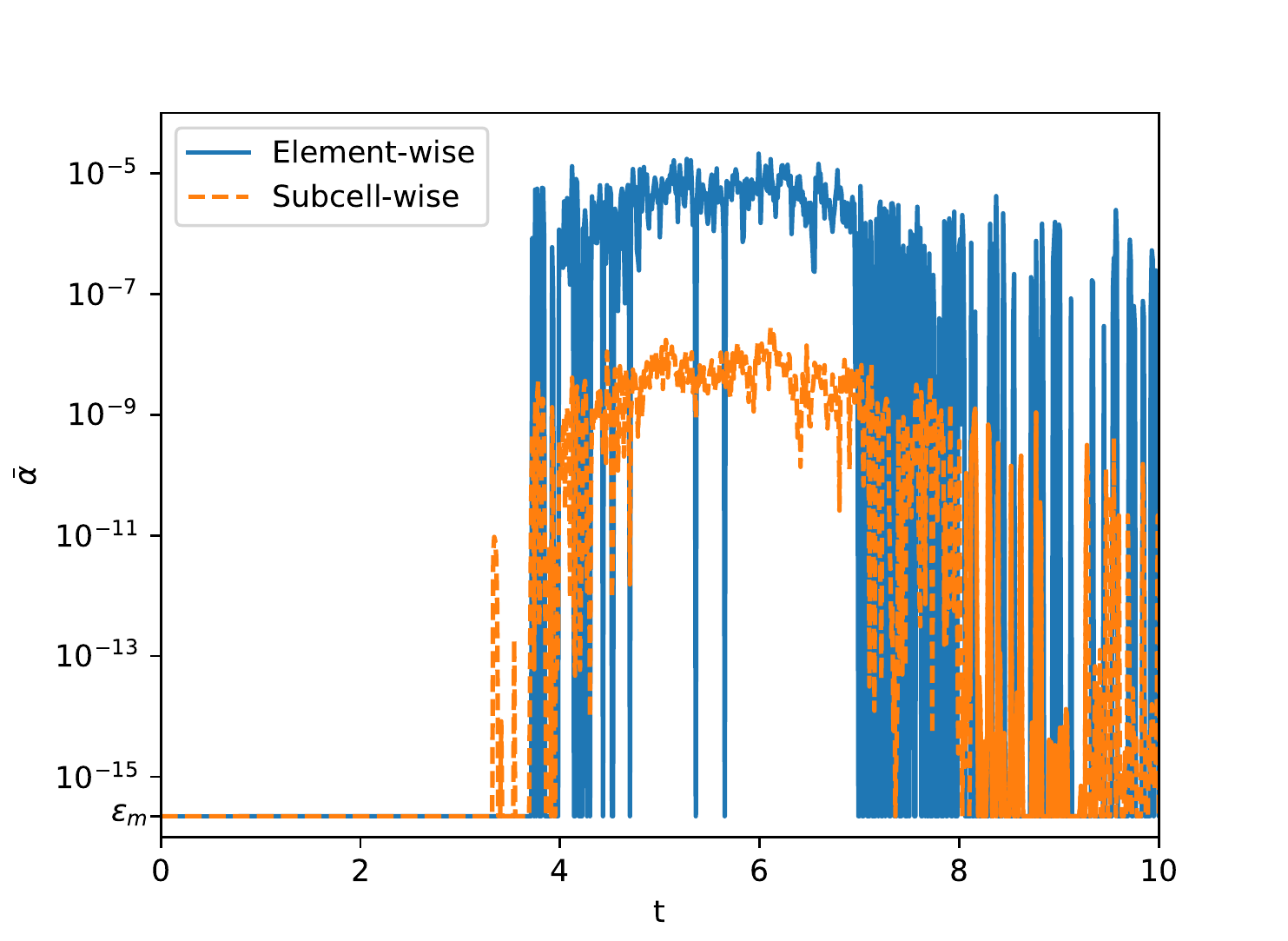}
	\caption{Amount of low-order method in the domain}
	\label{fig:khi_alphabar}
\end{subfigure}
\caption{
Evolution of the blending coefficient over time for the Kelvin-Helmholtz instability simulations.
$\varepsilon_m \approx 2.22 \times 10^{-16}$ corresponds to the double-precision machine epsilon.
}
\label{fig:khi_alpha}
\end{figure}

Figure~\ref{fig:khi_alpha} shows the evolution of the blending coefficient over time for the Kelvin-Helmholtz instability simulation.
To quantify how dissipative the resulting scheme is (that is, how much of the low-order method was required to stabilize the high order method), we plot a generalization of the metrics introduced by \citet{Rueda-Ramirez2021},
\begin{equation} \label{eq:maxAlpha}
\max (\alpha) (t) = \max_{\tau \in [t-\Delta \tau, t]} \left( \max_{e=1}^K \max_{i,j=0}^N \alpha^e_{ij}(\tau) \right),
~~~~
\bar \alpha (t) = \frac{1}{n_s} \sum_{s=1}^{n_s} \left( \frac{1}{V} \sum_{e=1}^K \sum_{i,j=0}^N J_{ij} \omega_{ij} (\alpha^e_{ij})^s \right),
\end{equation}
where $e \in [1,K]$ denotes the element index, $K$ is the number of elements of the domain, $i,j \in [0,N]$ are the node indexes, $N$ is the polynomial degree, $(\alpha^e_{ij})^s$ is the blending coefficient of node $ij$ of element $e$ at the RK stage $s$, $n_s$ is the number of RK stages taken from $t-\Delta \tau$ to $t$, and $V$ is the area of the domain. The value of $\bar \alpha$ represents a volume-weighted average of the blending coefficient to get an impression of the amount of FV mixed to the DG approximation. A value $\bar \alpha = 1$ means that the entire domain uses a first-order FV method, whereas $\bar \alpha = 0$ means that it uses only the high-order DG method.

Figure~\ref{fig:khi_alphamax} shows that the magnitude of the maximum blending coefficient of both blending strategies is similar for this example.
Moreover, from Figure~\ref{fig:khi_alphabar} it is clear that the subcell-wise blending strategy requires a much lower \textit{amount} of low-order method to stabilize the simulation than the element-wise strategy (about three orders of magnitude lower), leading to a less dissipative scheme.

We plot the density contours for the Kelvin-Helmholtz simulations at times $t=3.7$, $t=6.7$ and $t=10$ in Figure~\ref{fig:khi_dens}.
The evolution of the density in the domain is very similar for the element-wise and subcell-wise limiting techniques at the beginning of the simulation, which is expected due to the extremely low values of $\alpha$ in both setups.
Nevertheless, we remark that already at time $t=6.7$, small differences in the density distribution of the two limiting techniques can be appreciated visually, which make the two simulations evolve to substantially different solutions at $t=10$.

\begin{figure}[h!]
\centering
\includegraphics[trim=0 1400 1450 0 ,clip,width=0.48\linewidth]{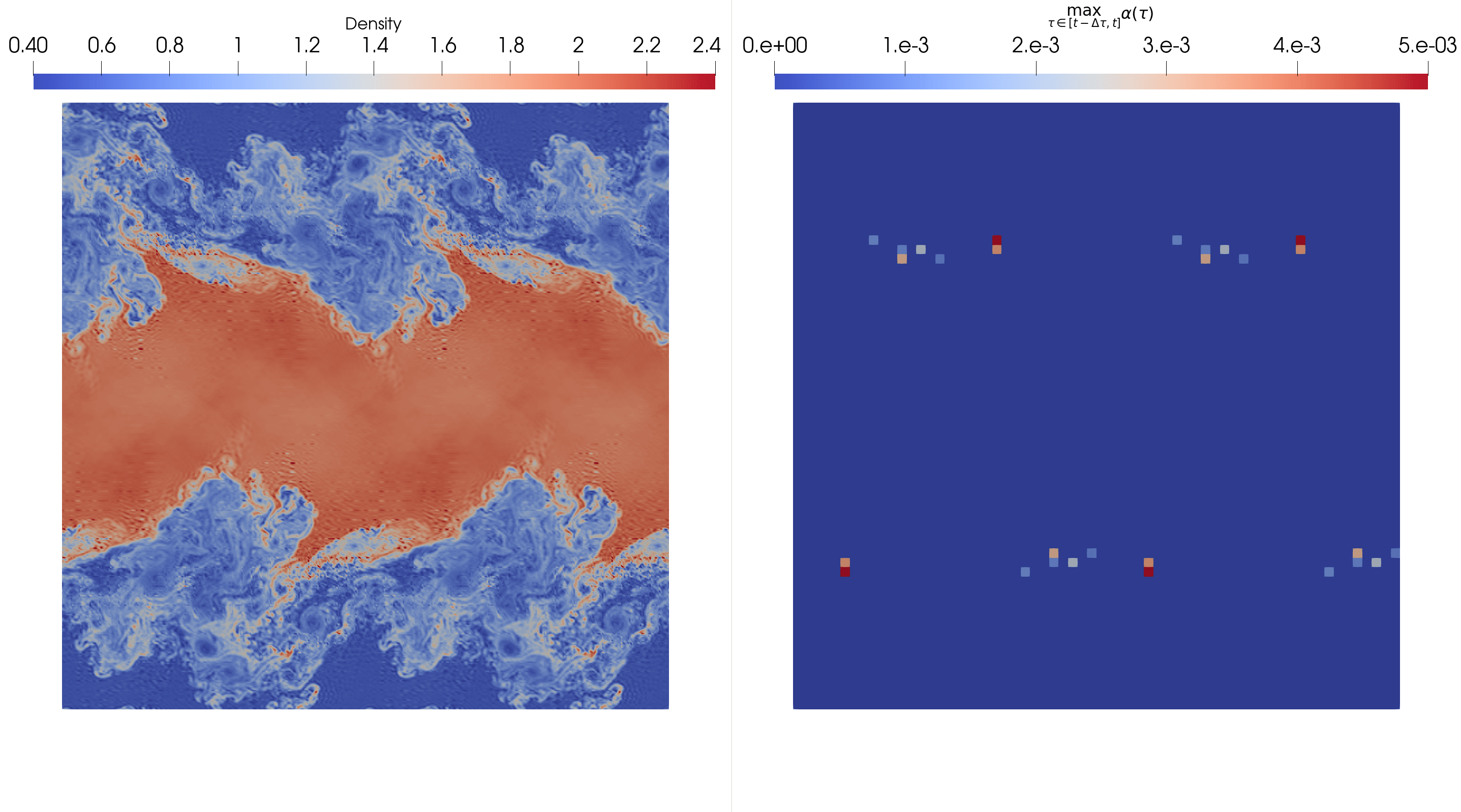}
\\
\begin{subfigure}[b]{0.33\linewidth}
    \centering
	\includegraphics[trim=110 200 1550 200 ,clip,width=\linewidth]{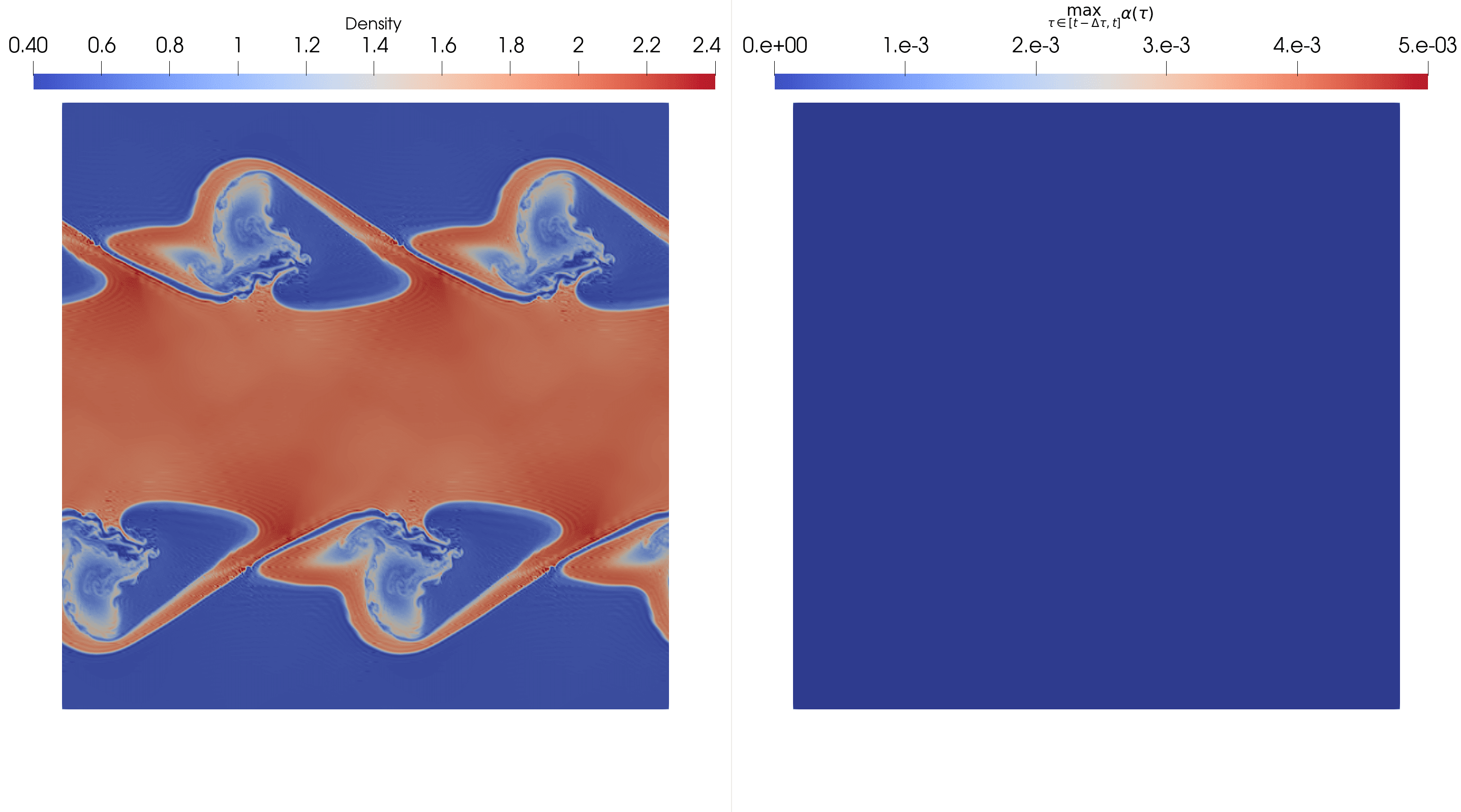}
	\hspace{50pt}
	\caption{Element-wise blending, $t=3.7$}
	\label{fig:khi_dens_elem_1}
\end{subfigure}
\begin{subfigure}[b]{0.33\linewidth}
    \centering
	\includegraphics[trim=110 200 1550 200 ,clip,width=\linewidth]{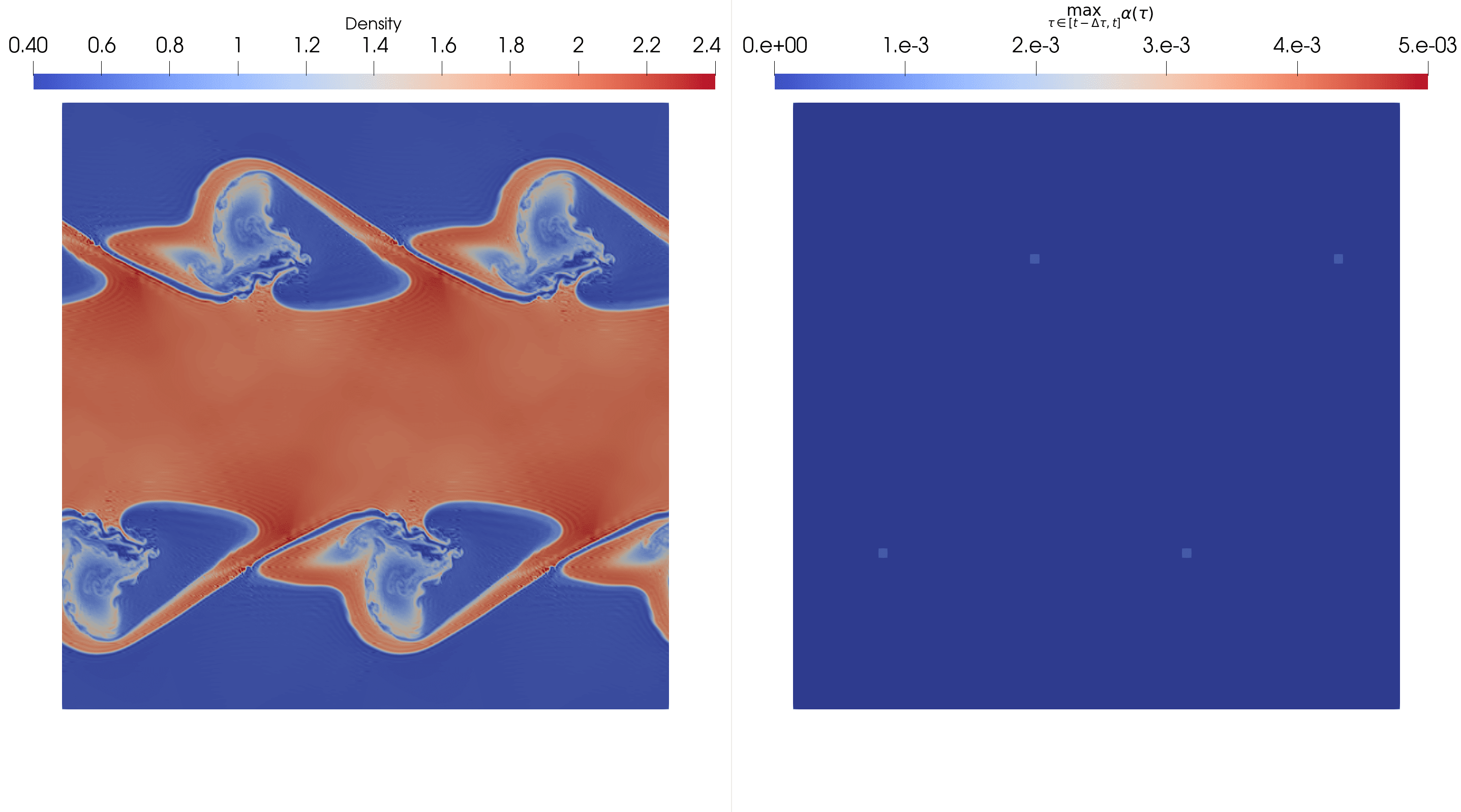}
	\hspace{50pt}
	\caption{Subcell-wise blending, $t=3.7$}
	\label{fig:khi_dens_subc_1}
\end{subfigure}
\begin{subfigure}[b]{0.33\linewidth}
    \centering
	\includegraphics[trim=110 200 1550 200 ,clip,width=\linewidth]{figs/KHI/KHI_elemwise_6_7.png}
	\hspace{50pt}
	\caption{Element-wise blending, $t=6.7$}
	\label{fig:khi_dens_elem_2}
\end{subfigure}
\begin{subfigure}[b]{0.33\linewidth}
    \centering
	\includegraphics[trim=110 200 1550 200 ,clip,width=\linewidth]{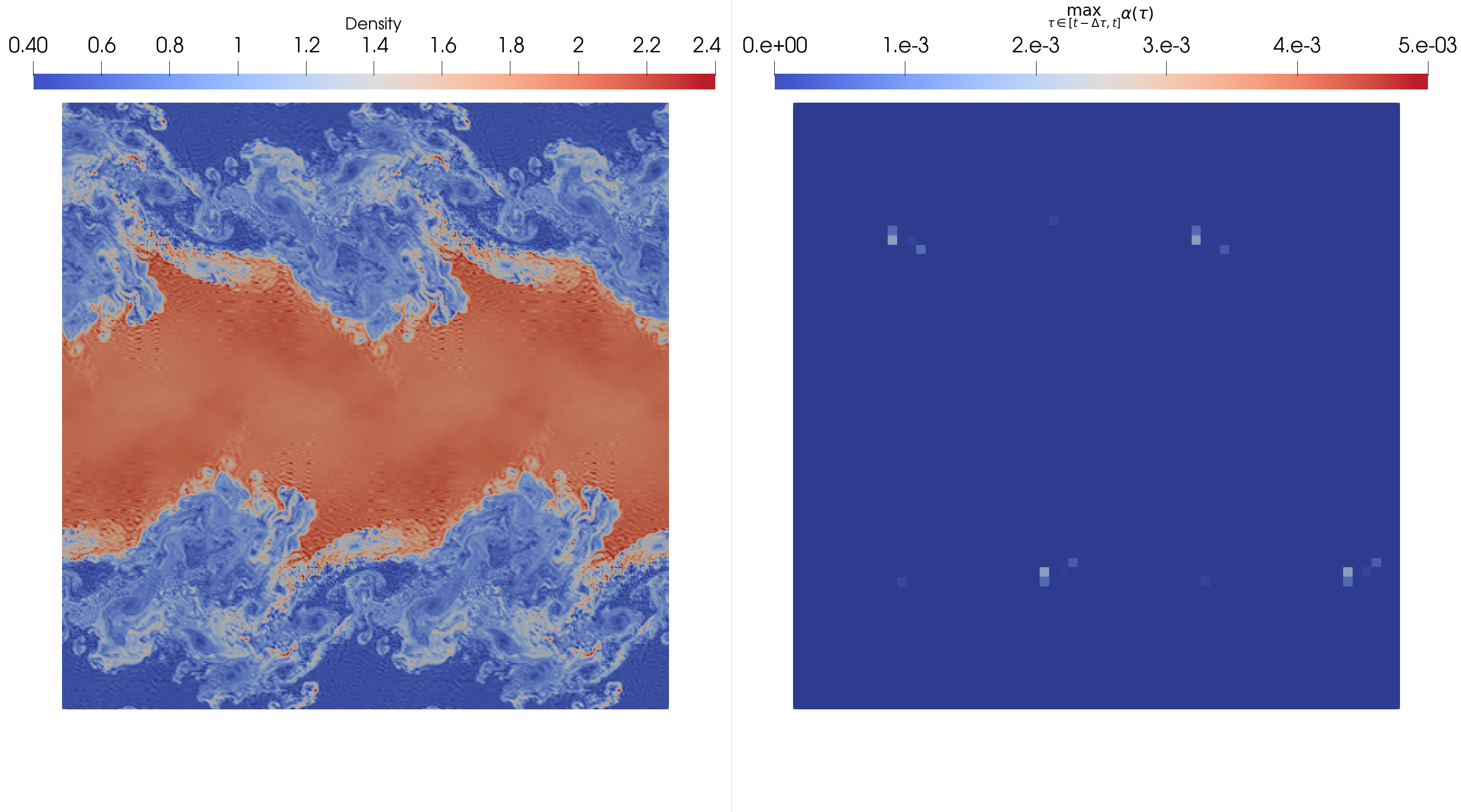}
	\hspace{50pt}
	\caption{Subcell-wise blending, $t=6.7$}
	\label{fig:khi_dens_subc_2}
\end{subfigure}
\begin{subfigure}[b]{0.33\linewidth}
    \centering
	\includegraphics[trim=110 200 1550 200 ,clip,width=\linewidth]{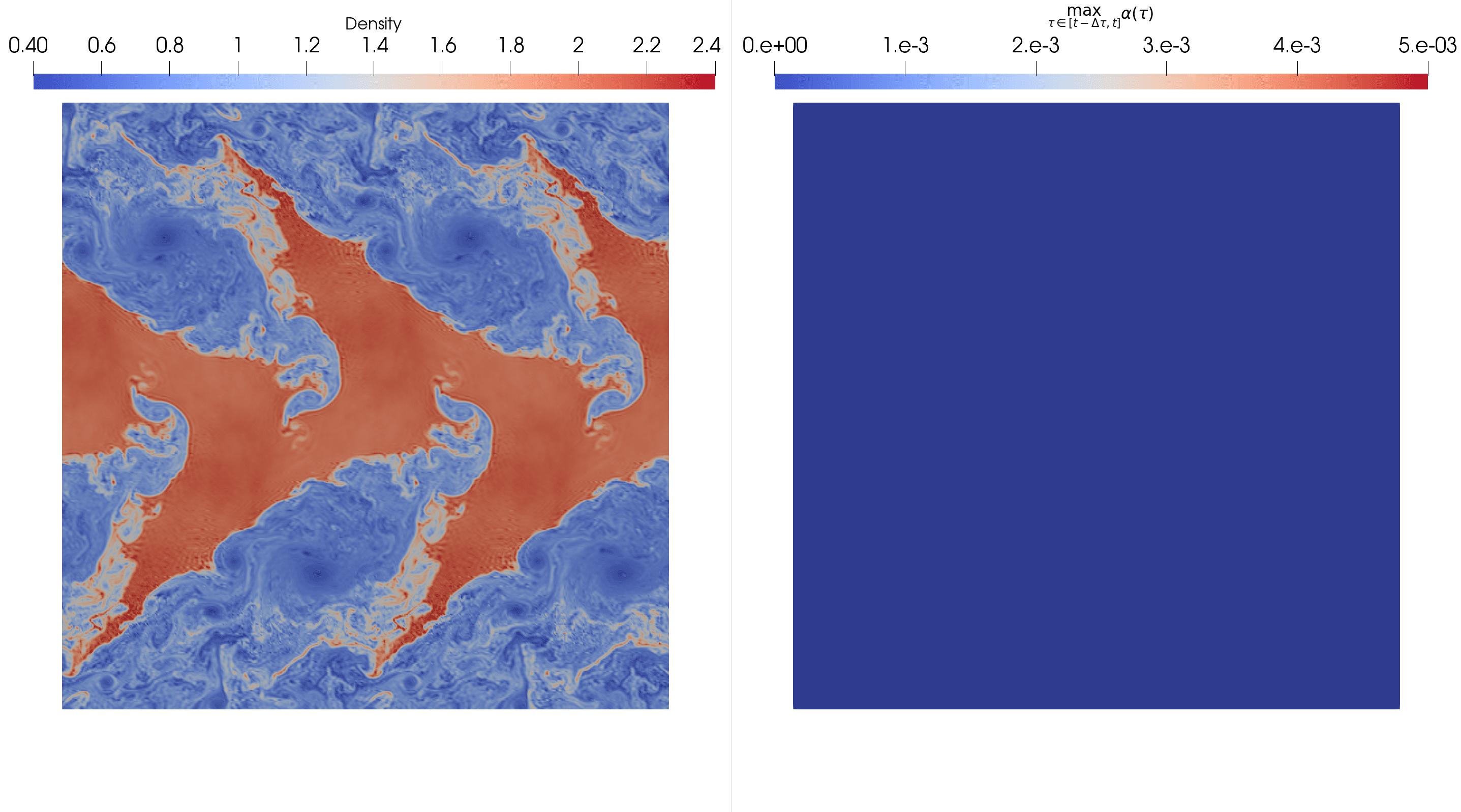}
	\hspace{50pt}
	\caption{Element-wise blending, $t=10$}
	\label{fig:khi_dens_elem_3}
\end{subfigure}
\begin{subfigure}[b]{0.33\linewidth}
    \centering
	\includegraphics[trim=110 200 1550 200 ,clip,width=\linewidth]{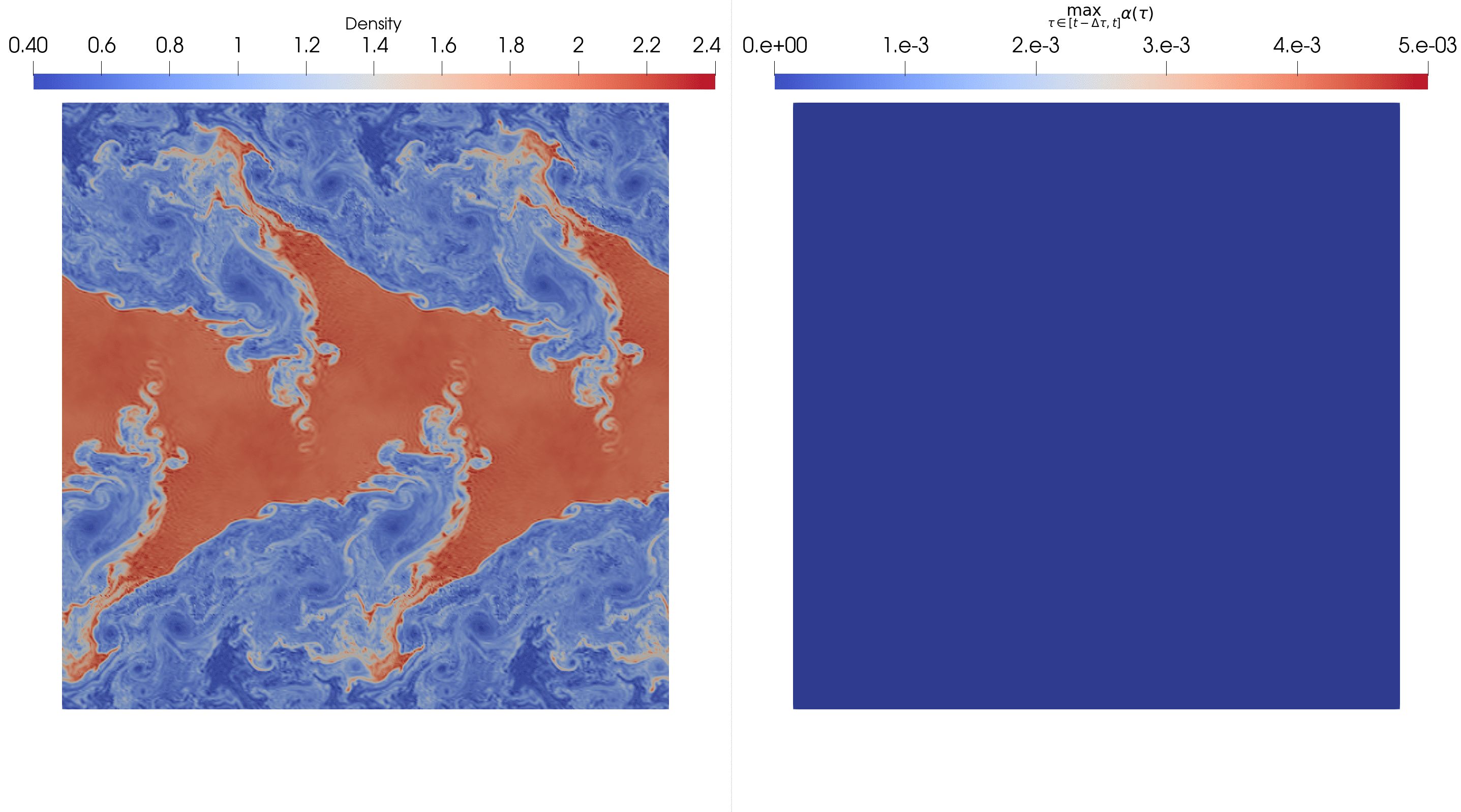}
	\hspace{50pt}
	\caption{Subcell-wise blending, $t=10$}
	\label{fig:khi_dens_subc_3}
\end{subfigure}

\caption{Density contours for the Kelvin-Helmholtz instability simulations with element-wise and subcell-wise blending strategies at three different times. DGSEM results with polynomial degree $N=7$ and $64\times 64$ elements.}
\label{fig:khi_dens}
\end{figure}

Figure~\ref{fig:khi_alpha_contour} shows the distribution of the maximum cumulative blending coefficient for the element-wise and subcell-wise techniques at time $t=6.7$.
Since the low-order method is only activated when the density or pressure computed by the high-order scheme fall below the specified thresholds \eqref{eq:rhop_restriction}, we have to compute the maximum cumulative blending coefficient over a time interval to be able to visualize where dissipation is being added.
It is clear that the subcell-wise technique applies the low-order method in a more localized manner than the element-wise technique, and thus retains the desirable high-order convergence properties as much as possible.

\begin{figure}[h!]
\centering
\includegraphics[trim=1450 1610 0 0 ,clip,width=0.48\linewidth]{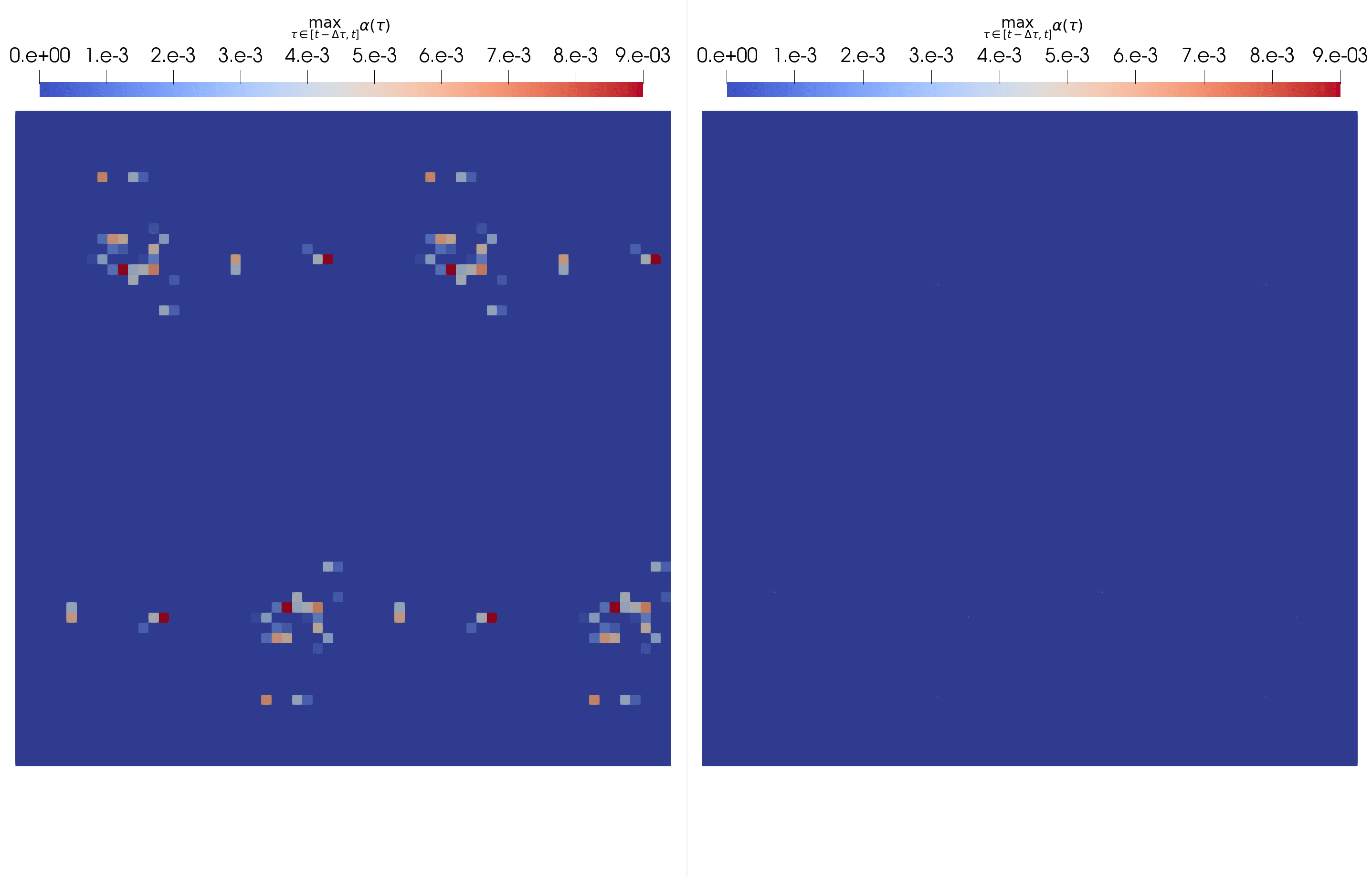}
\\
\begin{subfigure}[b]{0.35\linewidth}
    \centering
	\includegraphics[trim=0 200 1450 210 ,clip,scale=0.1]{figs/KHI/KHI_alpha_cumulative_elem_subcellwise_6_7_tau_0_1}
	\hspace{50pt}
	\caption{Element-wise blending}
	\label{fig:khi_alpha_elem_2}
\end{subfigure}
\begin{subfigure}[b]{0.6\linewidth}
    \centering
	\includegraphics[trim=0 200 50 210 ,clip,scale=0.1]{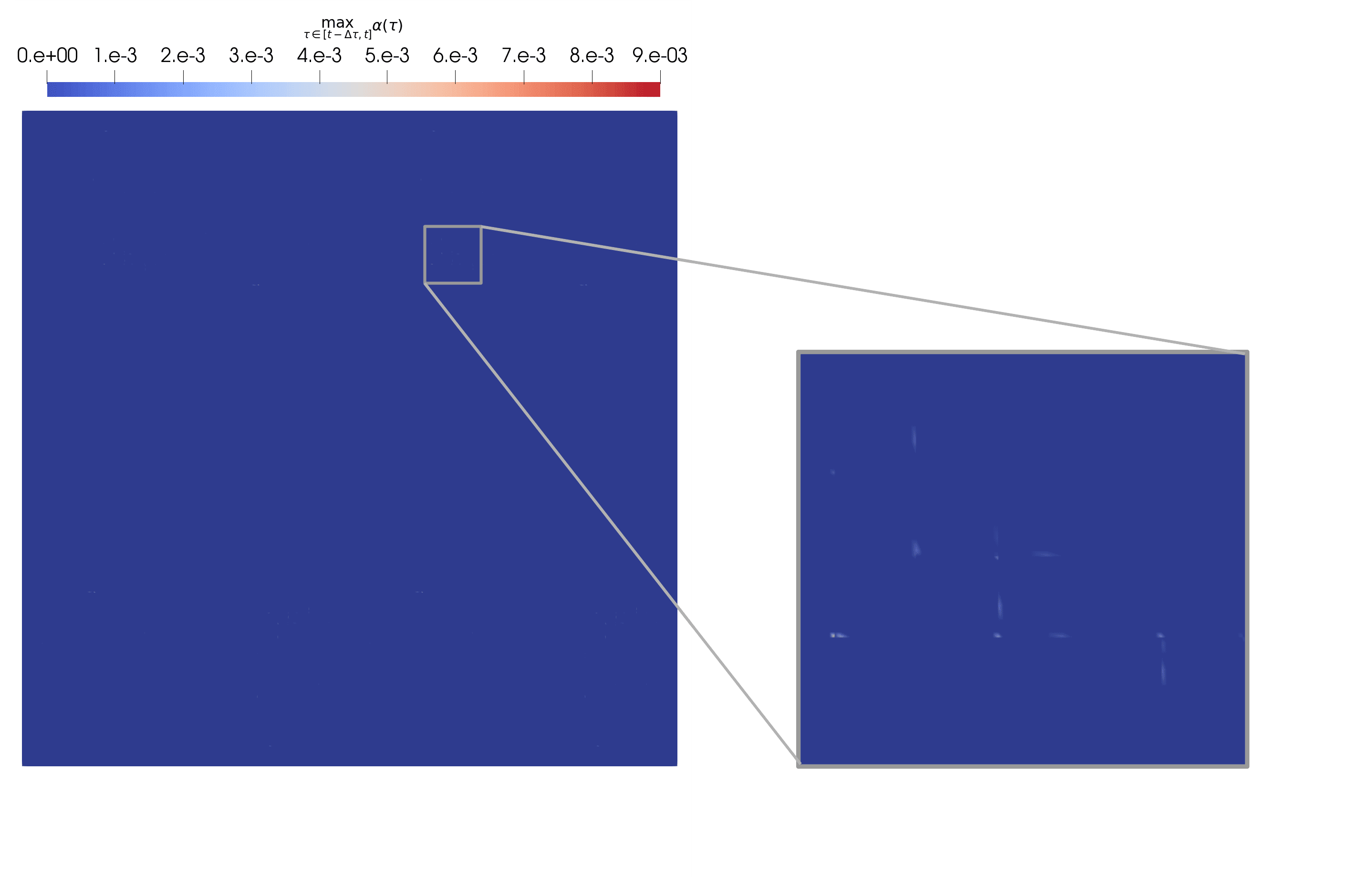}
	\hspace{50pt}
	\caption{Subcell-wise blending and detail}
	\label{fig:khi_alpha_subc_2}
\end{subfigure}

\caption{Contours of the maximum blending coefficient for the Kelvin-Helmholtz instability simulations with element-wise and subcell-wise blending strategies at time $t=6.7$.
The blending coefficient is taken as the maximum over all the SSP-RK stages of a sampling interval $\Delta \tau = 0.1$. DGSEM results with polynomial degree $N=7$ and $64\times 64$ elements.
}
\label{fig:khi_alpha_contour}
\end{figure}

\subsubsection{Inviscid Bow Shock Upstream of a Blunt Body} \label{sec:bowshock}
We now consider the 2D supersonic flow over a blunt body, which produces a detached bow shock.
The aim of this test is to simulate a shock problem using invariant domain preserving techniques with subcell-wise convex blending on a curvilinear grid.

The problem setup was proposed as an advanced test case for the Fifth International Workshop on High-Order CFD Methods \cite{hiocfd5}.
The left boundary of the domain is a circular arc with origin $(3.85,0)$ and radius $5.9$.
The blunt body consists of a flat front of length $1$ and two quarter circles of radius $0.5$.
The meshes were generated with gmsh \cite{geuzaine2009gmsh} using third-order hexahedral elements.
The heat capacity ratio is set to $\gamma=1.4$ and the initial condition is the constant state
\begin{equation}
\rho(x,y) = 1.4, \qquad
p(x,y) = 1, \qquad
v_1(x,y) = 4, \qquad
v_2(x,y) = 0.
\end{equation}

The boundary on the blunt body is a reflecting wall and the other boundaries are characteristics-based inflow/outflow boundaries, where the external state is selected depending on the flow conditions normal to the boundary.
We use the split-form DGSEM with the entropy-stable flux of \citet{Chandrashekar2013} and polynomial degree $N=5$, the Rusanov solver for the surface numerical fluxes, and the first-order FV method.

For this example, we compute the blending coefficient to impose a TVD-like solution on the density from the bar states \eqref{eq:barstates},
\begin{equation} \label{eq:IDPconditionRho}
    \min_{k \in \NN (ij)} \bar{\rho}_{(ij,k)}
    \le \rho_{ij} \le
    \max_{k \in \NN (ij)} \bar{\rho}_{(ij,k)}
\end{equation}
and a discrete minimum principle on a modified specific entropy from the bar states \eqref{eq:barstates},
\begin{equation} \label{eq:IDPconditionEnt}
    \min_{k \in \NN (ij)} \phi(\bar{\state{u}}_{(ij,k)})
    \le \phi({\state{u}}_{ij}),
\end{equation}
where the use of the modified specific entropy,
$\phi = e \rho^{1- \gamma}$,
guarantees the fulfillment of a discrete entropy inequality and is a particularly efficient choice, as $\phi$ is computationally cheaper to evaluate than the \textit{standard} specific entropy, $s=\ln (p \rho^{-\gamma})$, and is better suited for the Newton's method \cite{guermond2019invariant,maier2021efficient}.
Moreover, the fulfillment of conditions \eqref{eq:IDPconditionRho} and \eqref{eq:IDPconditionEnt} guarantees positivity of density and pressure \cite{guermond2019invariant}.

Since we use the LLF solver as the surface numerical flux for the DG and first-order FV method and due to the equivalence between the FV-LLF and the low-order graph Laplacian method shown in Section~\ref{sec:equivalence}, the bounds \eqref{eq:IDPconditionRho} and \eqref{eq:IDPconditionEnt} can always be met, up to machine precision accuracy.

The numerical results for time $t=10$ are plotted in Figure~\ref{fig:bowshock} for two different meshes with 864 and 1536 elements, respectively.
In both mesh resolutions, the shock is resolved sharply, even when it is located inside an element.
It is interesting to note that the FV method is activated primarily upwind of the shock to damp out possible oscillations of the high-order scheme.
About $\bar{\alpha} = 11.03 \%$ of the domain volume is discretized with low order at the end of the simulation for the coarse grid, and about $\bar{\alpha} = 9.50 \%$ for the fine grid.

\begin{figure}[htb]
\centering
\includegraphics[trim=0 1400 1440 0 ,clip,width=0.45\linewidth]{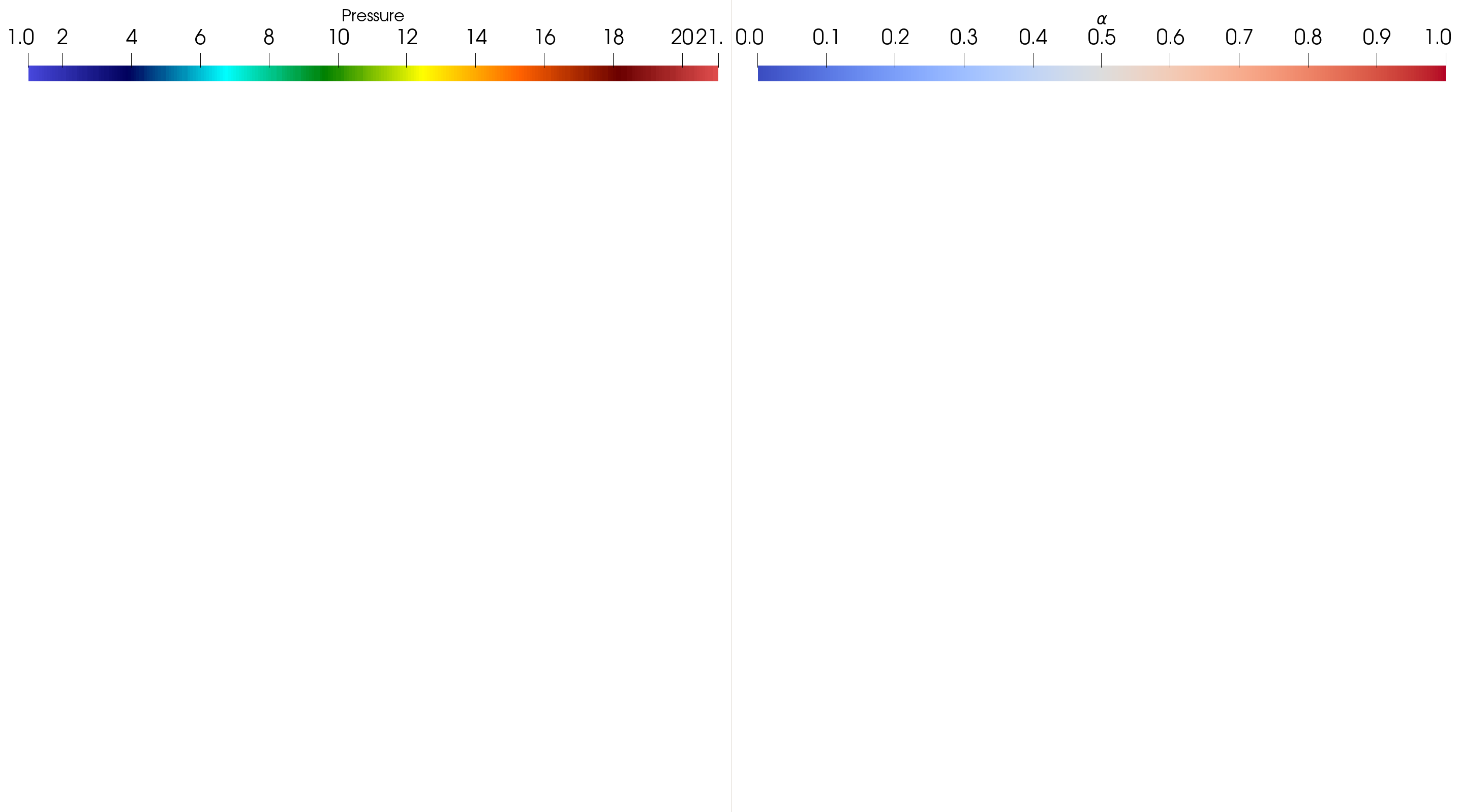}
\includegraphics[trim=1440 1400 0 0 ,clip,width=0.45\linewidth]{figs/BowShock/Legend.png}
\begin{subfigure}[b]{0.45\linewidth}
    \centering
	\includegraphics[trim=550 40 1950 40 ,clip,height=0.9\linewidth]{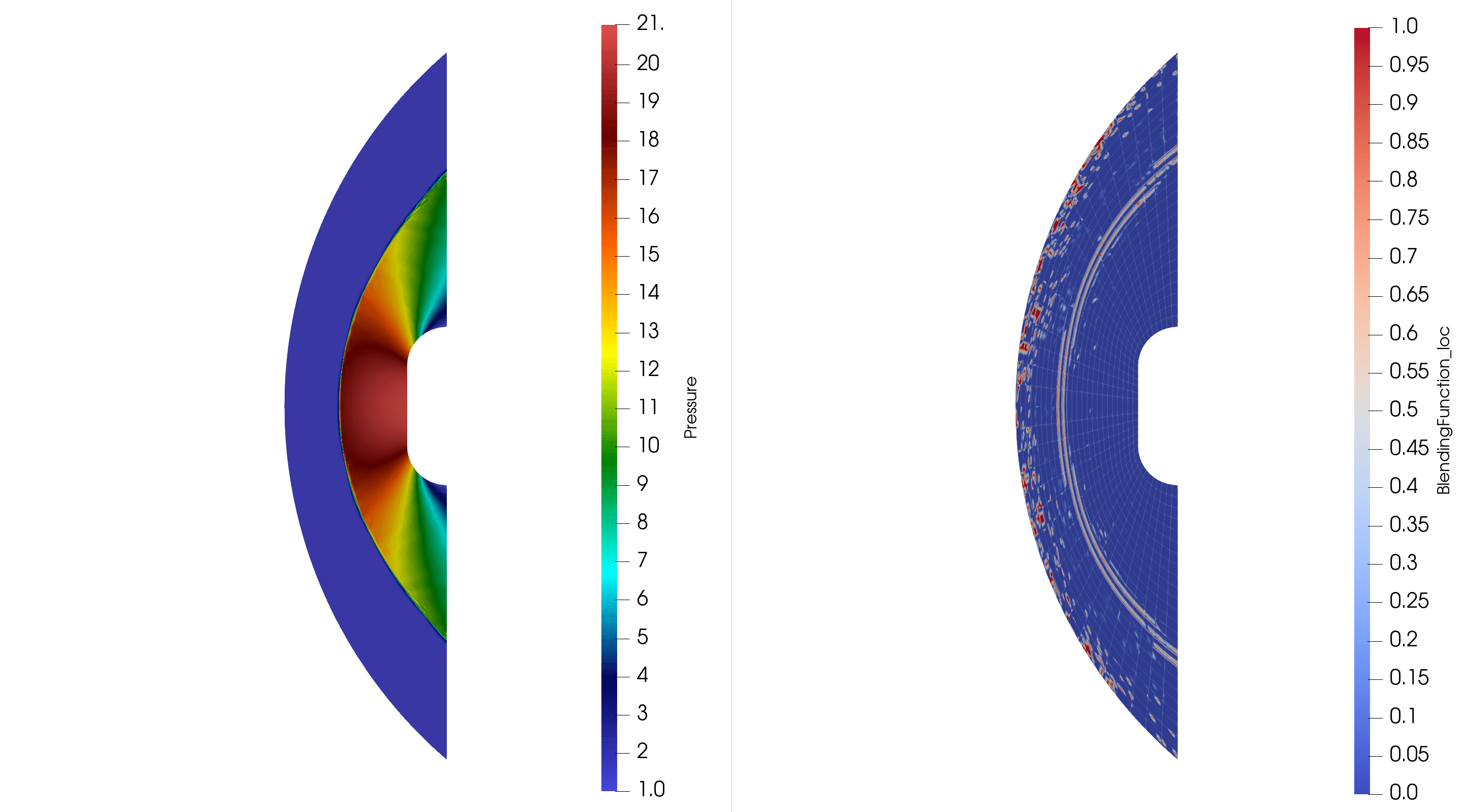}
	\includegraphics[trim=1950 40 550 40 ,clip,height=0.9\linewidth]{figs/BowShock/Bowshock_grid3.png}
	\caption{Mesh with 864 elements (31104 DOFs)}
	\label{fig:bow:1}
\end{subfigure}
\begin{subfigure}[b]{0.45\linewidth}
    \centering
	\includegraphics[trim=550 40 1950 40 ,clip,height=0.9\linewidth]{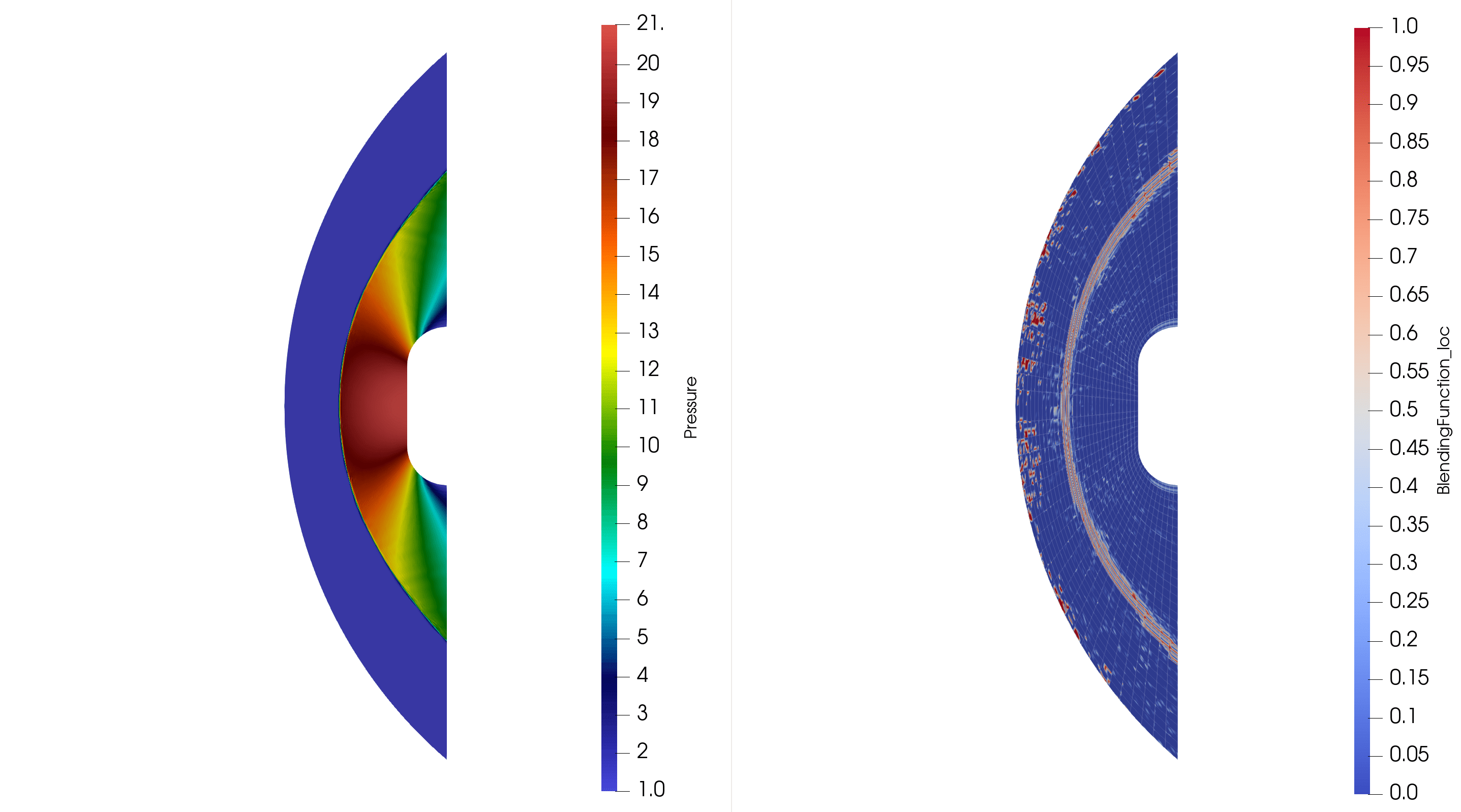}
	\includegraphics[trim=1950 40 550 40 ,clip,height=0.9\linewidth]{figs/BowShock/Bowshock_grid4.png}
	\caption{Mesh with 1536 elements (55296 DOFs)}
	\label{fig:bow_2}
\end{subfigure}
\caption{Simulation results for the bow shock simulation at time $t = 10$.
We plot the pressure, $p$, and the instant local blending coefficient, $\alpha$, for two different mesh resolutions. DGSEM results with polynomial degree $N=5$.}
\label{fig:bowshock}
\end{figure}

\subsubsection{High Mach Astrophysical Jet}

To test the robustness of the subcell limiting strategies for the DGSEM, we simulate an astrophysical jet with a Mach number $\textrm{Ma} \approx 2000$.
This example was originally proposed by \citet{ha2005numerical}, and is considered an extreme benchmark for robust high-order problems \cite{zhang2010positivity,liuoscillation}.

The computation domain is the unit square, $[-0.5,0.5]^2$, tessellated into $256 \times 256$ quadrilateral elements of degree $N=3$.
The top and bottom boundaries are periodic, and the left and right boundaries are characteristics-based inflow/outflow boundaries.
The domain is filled with a monatomic gas ($\gamma = 5/3$) at rest with
\begin{equation}
\rho(x,y) = 0.5, \qquad
p(x,y) = 0.4127, \qquad
v_1(x,y) = 0, \qquad
v_2(x,y) = 0,
\end{equation}
and on the left boundary there is a hypersonic inflow with
\begin{equation}
\rho(-0.5,y) = 5, \qquad
p(-0.5,y) = 0.4127, \qquad
v_1(-0.5,y) = 800, \qquad
v_2(-0.5,y) = 0
\end{equation}
for $y \in [-0.05, 0.05]$, which corresponds to a Mach number of $\textrm{Ma}=2156.91$ with respect to the speed of sound in the jet gas, and $\textrm{Ma}=682.08$ with respect to the speed of sound in the ambient gas.

As in the previous examples, we discretize the Euler equations using the split-form DGSEM, and the entropy-conserving and kinetic energy preserving flux of \citet{Chandrashekar2013} for the volume numerical fluxes, $\state{f}^*$.
Similarly to the case of the bow shock in Section~\ref{sec:bowshock}, we impose TVD-like bounds on the density \eqref{eq:IDPconditionRho} and a local minimum principle on the modified specific entropy \eqref{eq:IDPconditionEnt} based on the bar states \eqref{eq:barstates}.

In this example, we test different low-order methods to assess their impact on the solution when performing element-wise and subcell-wise limiting.
As the baseline low-order scheme, we use a standard FV method with first-order reconstruction (Figure~\ref{fig:FV_1st}) and the LLF method as the surface numerical flux, $\numfluxb{f}$.
Because of the equivalence shown in Section~\ref{sec:equivalence}, this scheme is able to fulfill the bounds for both the element-wise and subcell-wise blending strategies.

\begin{figure}[h!]
\centering
\includegraphics[trim=0 1400 1450 0 ,clip,width=0.48\linewidth]{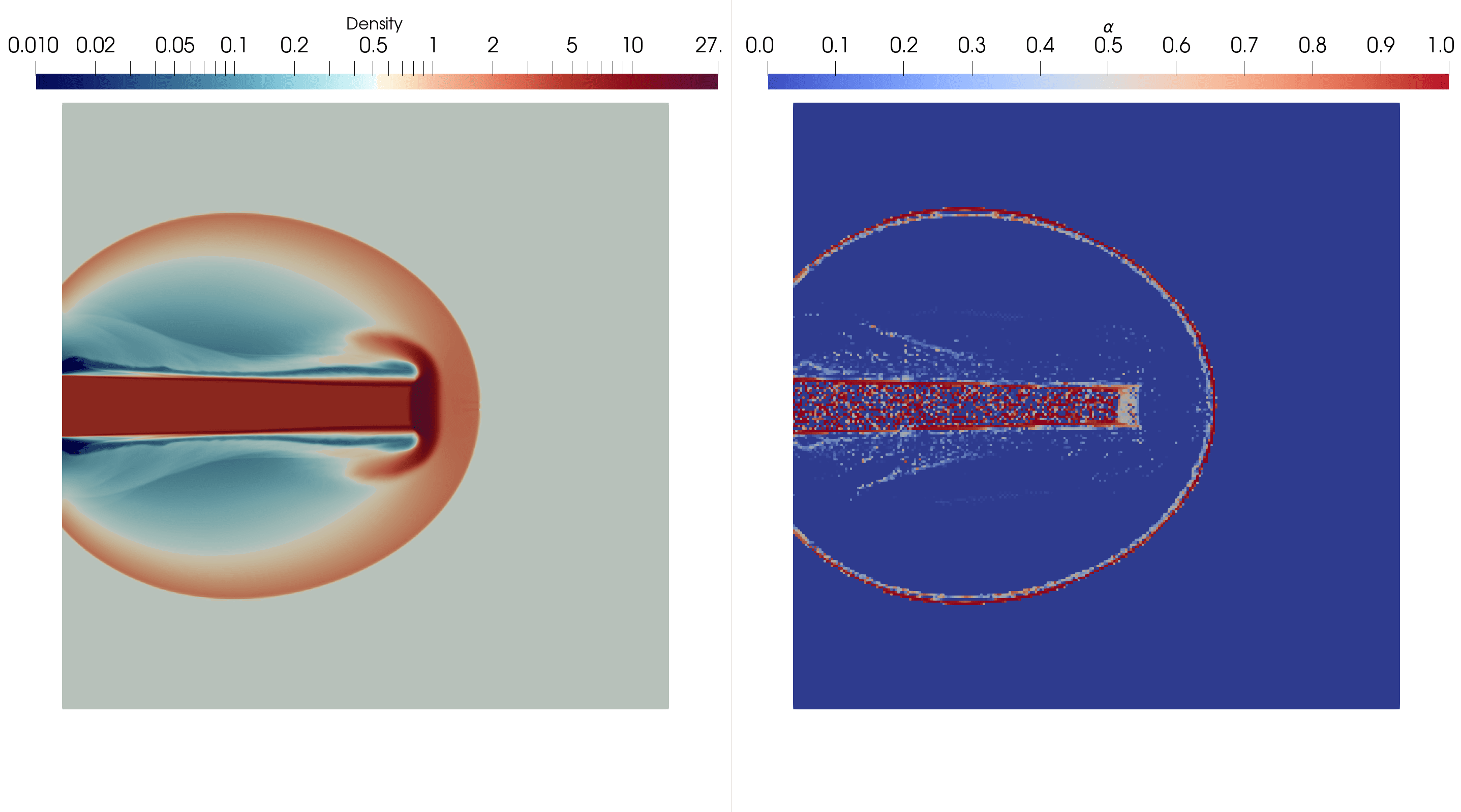}
\includegraphics[trim=1450 1400 0 0 ,clip,width=0.48\linewidth]{figs/Jet/Jet_ElemWise_1st_LLF.png}
\begin{subfigure}[b]{\linewidth}
    \centering
	\includegraphics[trim=110 200 1550 200 ,clip,width=0.35\linewidth]{figs/Jet/Jet_ElemWise_1st_LLF.png}
	\hspace{50pt}
	\includegraphics[trim=1550 200 110 200 ,clip,width=0.35\linewidth]{figs/Jet/Jet_ElemWise_1st_LLF.png}
	\caption{ESDGSEM + first-order FV (LLF)}
	\label{fig:jet_elem_1}
\end{subfigure}
\begin{subfigure}[b]{\linewidth}
    \centering
	\includegraphics[trim=110 200 1550 200 ,clip,width=0.35\linewidth]{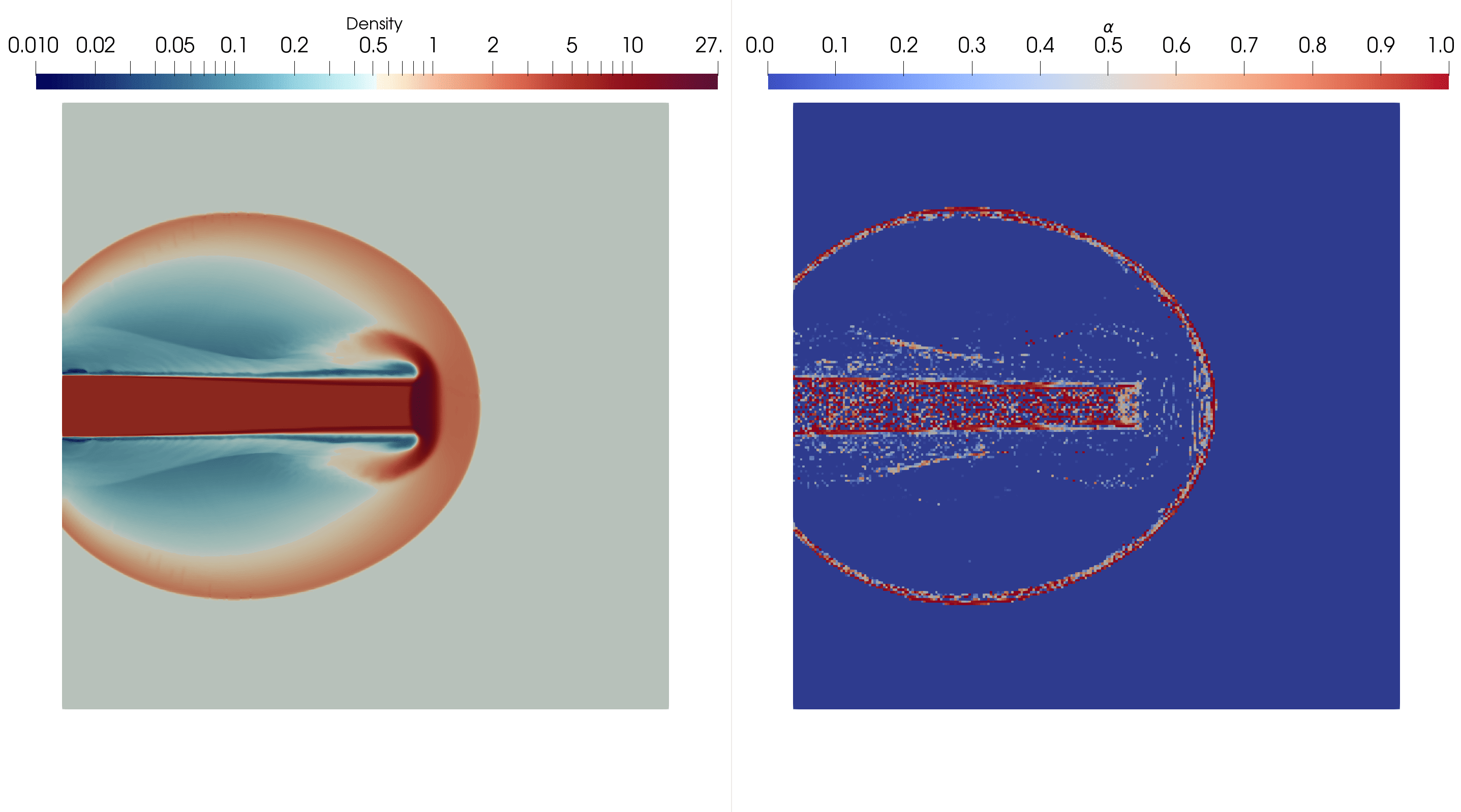}
	\hspace{50pt}
	\includegraphics[trim=1550 200 110 200 ,clip,width=0.35\linewidth]{figs/Jet/Jet_ElemWise_1st_HLLE.png}
	\caption{ESDGSEM + first-order FV (HLLE)}
	\label{fig:jet_elem_2}
\end{subfigure}
\begin{subfigure}[b]{\linewidth}
    \centering
	\includegraphics[trim=110 200 1550 200 ,clip,width=0.35\linewidth]{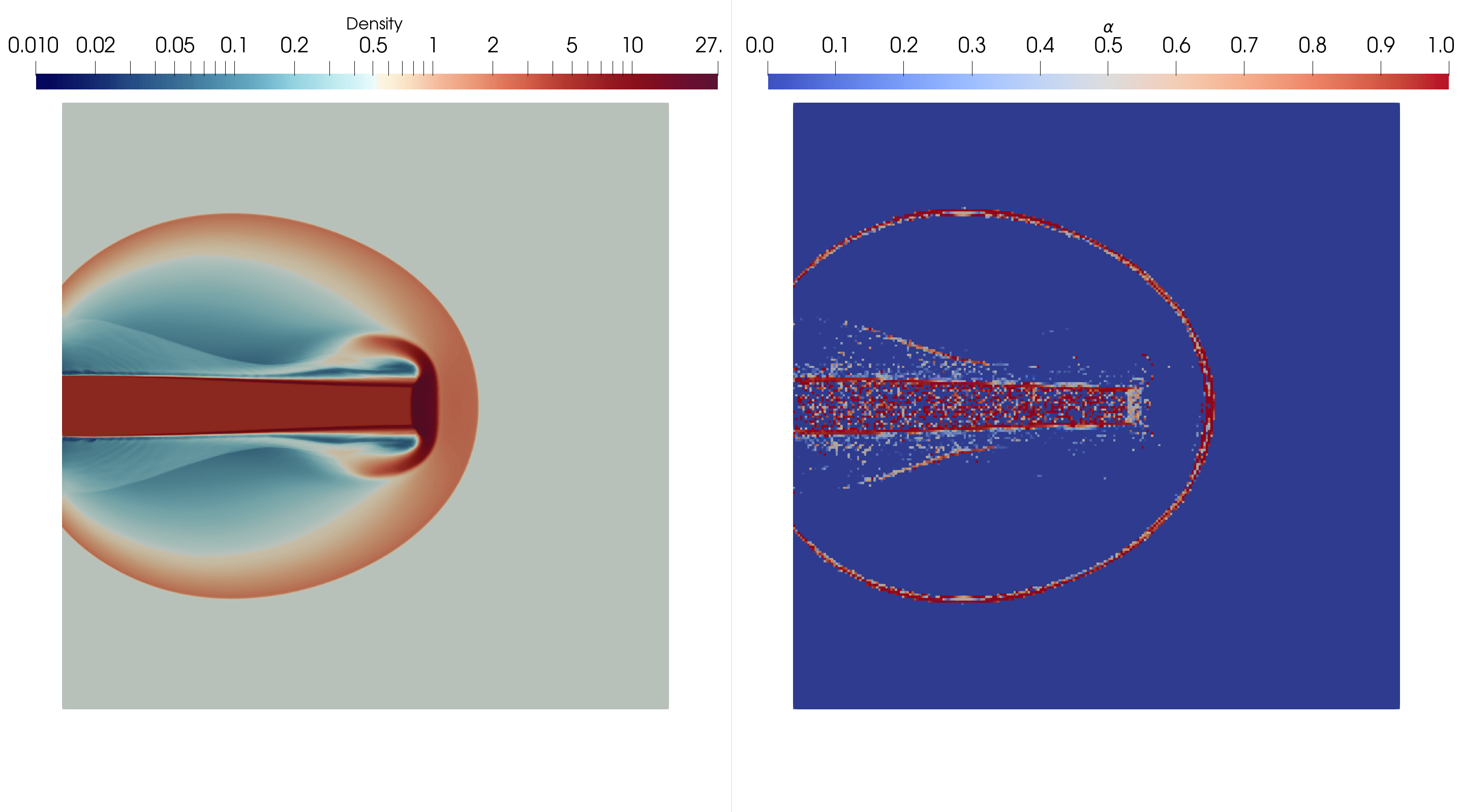}
	\hspace{50pt}
	\includegraphics[trim=1550 200 110 200 ,clip,width=0.35\linewidth]{figs/Jet/Jet_ElemWise_2nd_HLLE.png}
	\caption{ESDGSEM + second-order FV (HLLE)}
	\label{fig:jet_elem_3}
\end{subfigure}

\caption{Simulation results for the astrophysical jet simulation with element-wise convex blending at time $t = 10^{-3}$.
We plot the density, $\rho$, and the instant blending coefficient, $\alpha$, for three different schemes considered. DGSEM results with polynomial degree $N=3$ and $256\times 256$ elements.}
\label{fig:jet_elem}
\end{figure}

\begin{figure}[h!]
\centering
\includegraphics[trim=0 1638 1450 0 ,clip,width=0.48\linewidth]{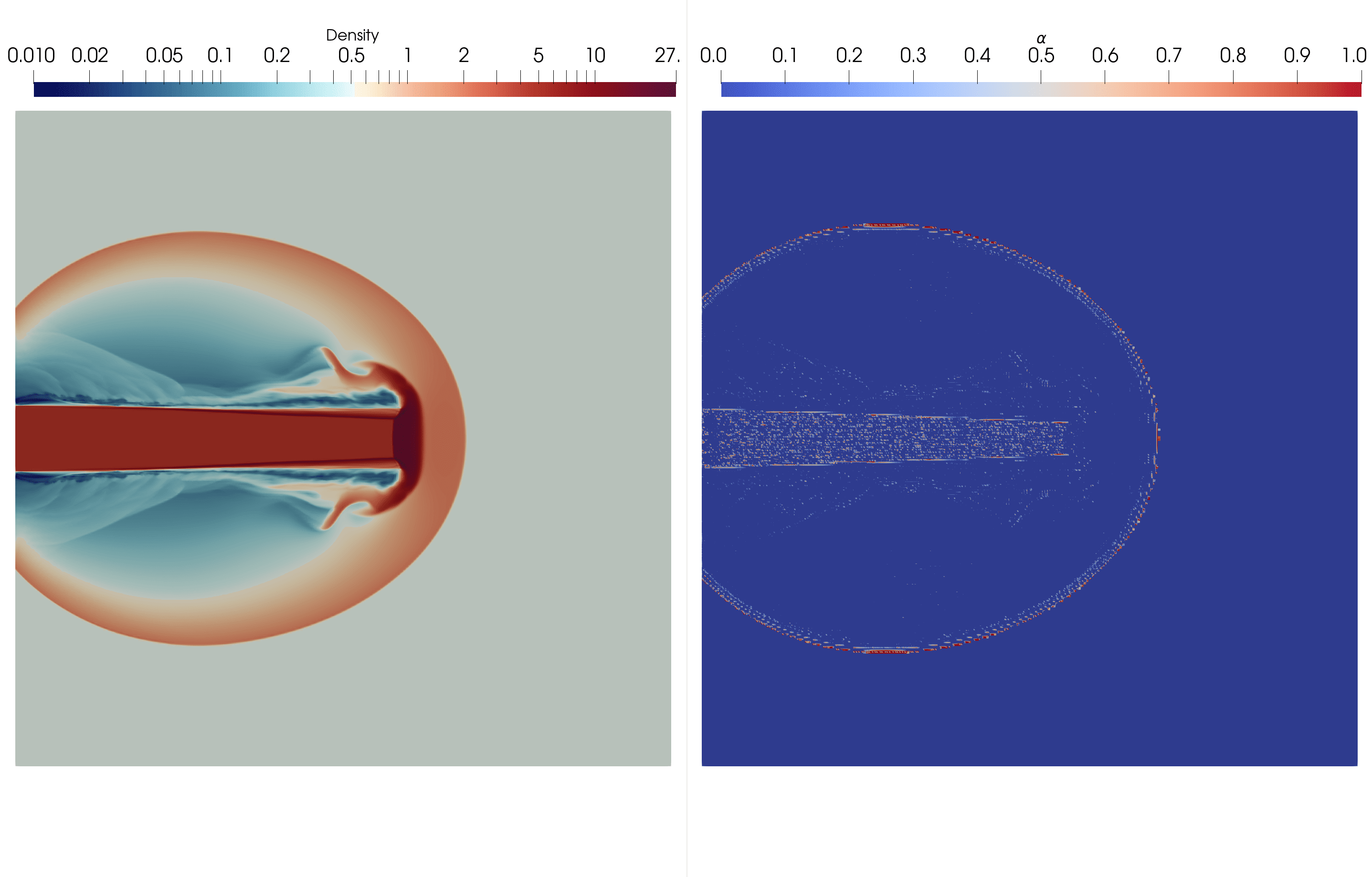}
\includegraphics[trim=1450 1638 0 0 ,clip,width=0.48\linewidth]{figs/Jet/Jet_SubcellWise_1st_LLF.png}
\begin{subfigure}[b]{\linewidth}
    \centering
	\includegraphics[trim=0 220 1450 220 ,clip,width=0.35\linewidth]{figs/Jet/Jet_SubcellWise_1st_LLF.png}
	\hspace{50pt}
	\includegraphics[trim=1450 220 0 220 ,clip,width=0.35\linewidth]{figs/Jet/Jet_SubcellWise_1st_LLF.png}
	\caption{ESDGSEM + first-order FV (LLF)}
	\label{fig:jet_subcell_1}
\end{subfigure}
\begin{subfigure}[b]{\linewidth}
    \centering
	\includegraphics[trim=0 220 1450 220 ,clip,width=0.35\linewidth]{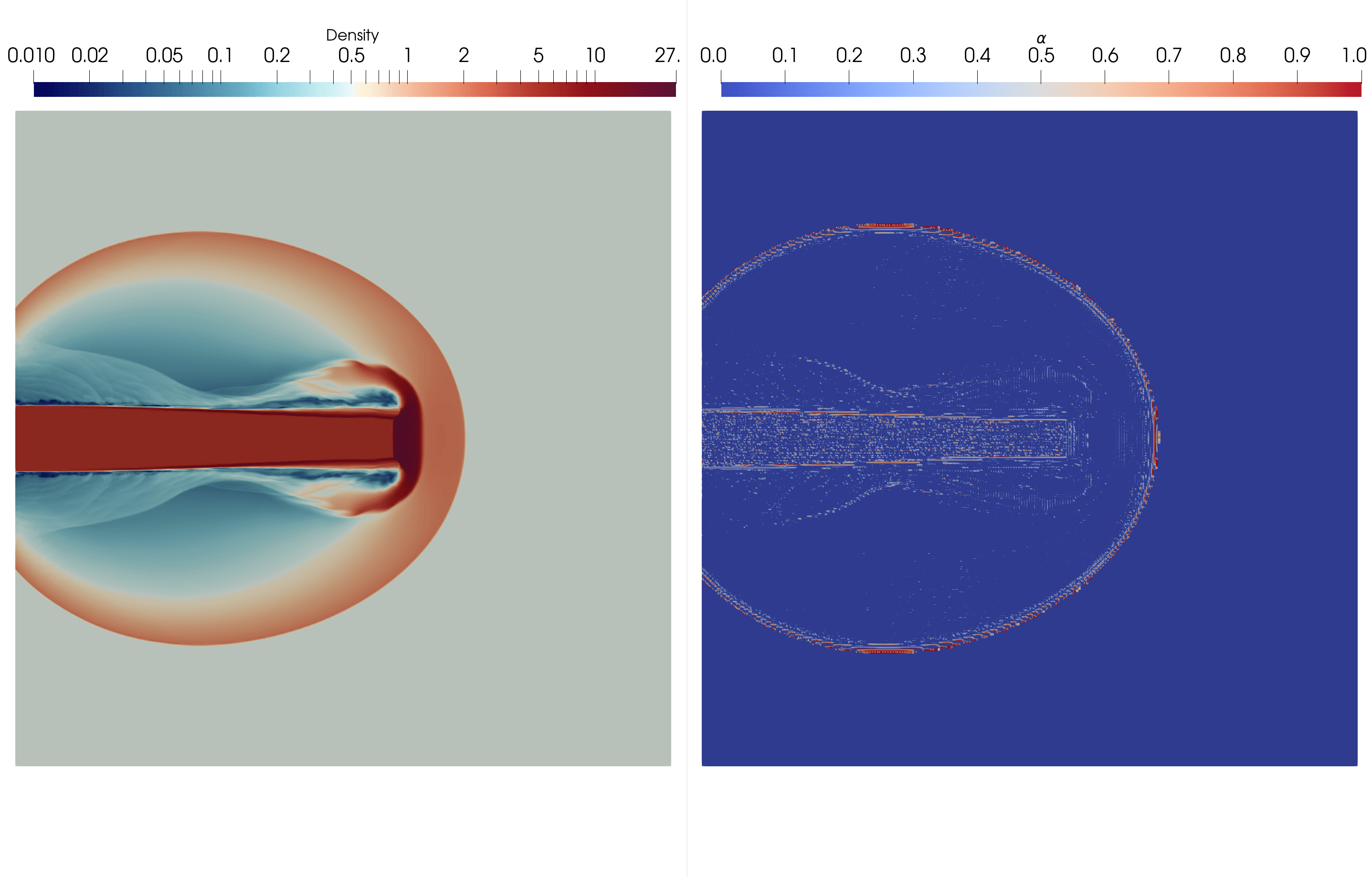}
	\hspace{50pt}
	\includegraphics[trim=1450 220 0 220 ,clip,width=0.35\linewidth]{figs/Jet/Jet_SubcellWise_1st_HLLE.png}
	\caption{ESDGSEM + first-order FV (HLLE)}
	\label{fig:jet_subcell_2}
\end{subfigure}
\begin{subfigure}[b]{\linewidth}
    \centering
	\includegraphics[trim=0 220 1450 220 ,clip,width=0.35\linewidth]{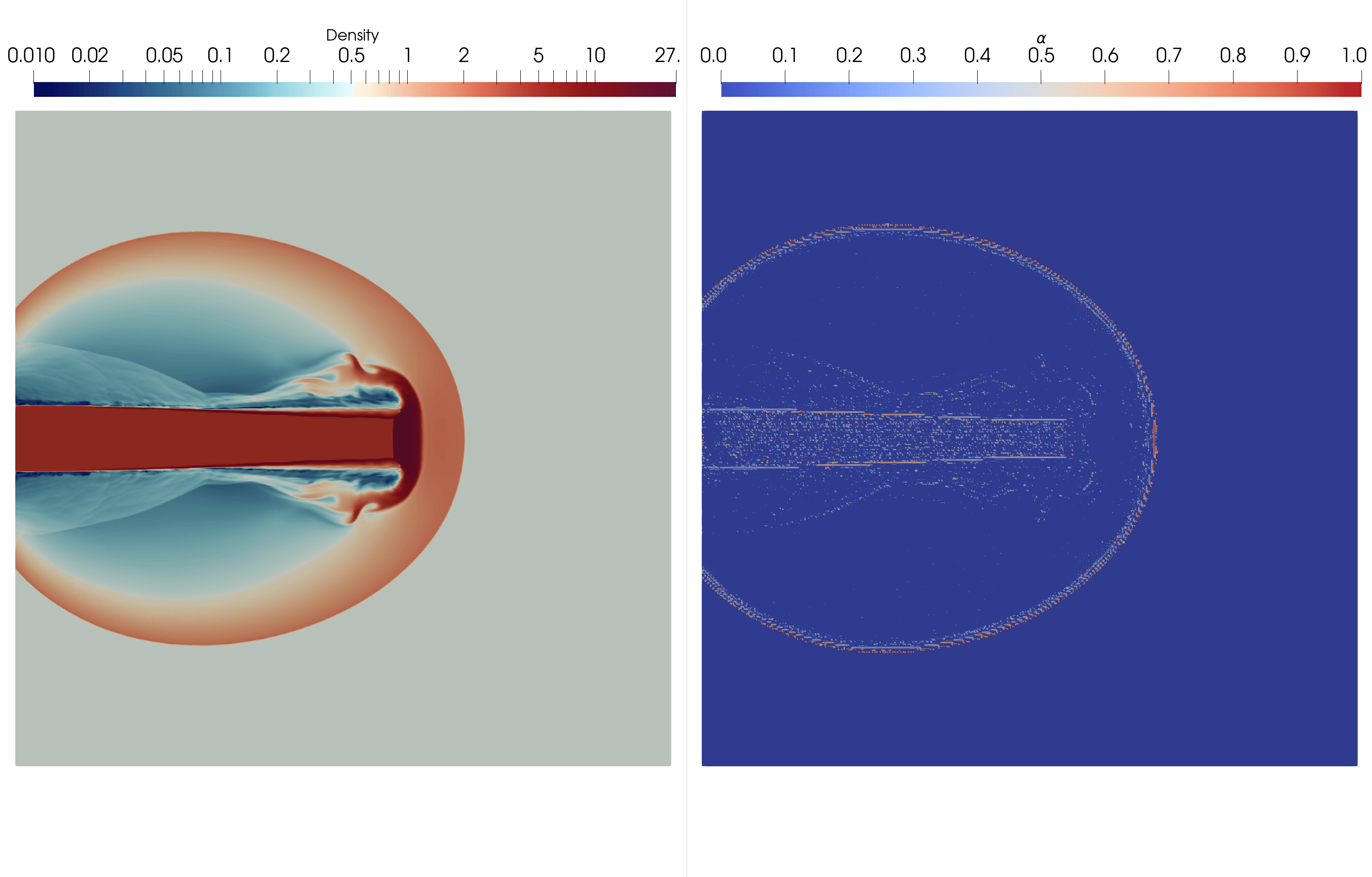}
	\hspace{50pt}
	\includegraphics[trim=1450 220 0 220 ,clip,width=0.35\linewidth]{figs/Jet/Jet_SubcellWise_2nd_HLLE.png}
	\caption{ESDGSEM + second-order FV (HLLE)}
	\label{fig:jet_subcell_3}
\end{subfigure}

\caption{Simulation results for the astrophysical jet simulation with subcell-wise convex blending at time $t = 10^{-3}$.
We plot the density, $\rho$, and the instant blending coefficient, $\alpha$, for three different schemes considered. DGSEM results with polynomial degree $N=3$ and $256\times 256$ elements.}
\label{fig:jet_subcell}
\end{figure}

Besides the baseline scheme, we test a FV scheme with second-order reconstruction on primitive variables ($\rho$, $\vec{v}$ and $p$) and the positivity-preserving HLLE Riemann solver \cite{Einfeldt1988} for the surface numerical fluxes.
We compute the reconstructed quantities in reference space using a \textit{minmod} limiter.
Since there are collocated nodes at the interface between every two neighbor elements, the \textit{minmod} reconstruction procedure is not well-defined for the subcells at the element boundaries.
Therefore, to guarantee a positivity-preserving FV reconstruction, we only use the \textit{minmod} limiter for the inner subcells, and fall back on a first-order reconstruction at the element boundaries, as shown in Figure~\ref{fig:FV_2nd}.

The low-order methods that use a second-order reconstruction and/or the HLLE solver are not guaranteed to satisfy the bounds computed from the bar states, i.e.\ \eqref{eq:IDPconditionEnt} and \eqref{eq:IDPconditionRho}; bounds of this form are closely related to the first-order LLF discretizations.
Nevertheless, we keep \eqref{eq:IDPconditionRho} and \eqref{eq:IDPconditionEnt} in the absence of better bounds and employ the algorithms that determine the blending coefficient to obtain the numerical solution that is closest to the restrictions.

Figure~\ref{fig:jet_elem} shows the density and blending coefficient distributions of the three different schemes considered when using the element-wise convex combination strategy.
It is evident that the shock resolution and the amount of small scales improves when higher-order FV reconstruction procedures and high-fidelity Riemann solvers are used.
Figure~\ref{fig:jet_subcell} shows the density and blending coefficient distributions of the three different schemes considered when using the subcell-wise convex combination strategy.
It is clear that the subcell-wise blending captures the shocks more sharply and allows the development of smaller scales than the element-wise blending.
Moreover, the resolution of small scale features improves again when higher-order FV reconstruction procedures and high-fidelity Riemann solvers are used.

\subsection{Ideal Magnetohydrodynamics}

We now consider the ideal MHD equations with a divergence cleaning mechanism that is based on a generalized Lagrange multiplier (GLM), also known as the GLM-MHD equations \cite{derigs2018ideal,Dedner2002}.
The GLM-MHD equations can be written as a conservative or non-conservative system.
Although the non-conservative form is necessary to show entropy stability \cite{derigs2018ideal}, we use the conservative form in this work since it is more convenient to cast in the framework of the present collection of subcell limiting strategies.

The conserved quantities of the GLM-MHD system are $\state{u} = [\rho, \rho \vec{v}, \rho E]^T$, $\rho$ is the density, $\vec{v} = (v_1, v_2, v_3)^T$ is the velocity, $E$ is the specific total energy, $\vec{B} = (B_1, B_2, B_3)^T$ is the magnetic field, and $\psi$ is the so-called \textit{divergence-correcting field}, a GLM that is added to the original MHD system to minimize the magnetic field divergence.

The flux contains the Euler, ideal MHD and GLM contributions,
\begin{equation}
\blocktensor{f}(\mathbf{u}) = \blocktensor{f}^{\supEuler} +\blocktensor{f}^{\supMHD}
+\blocktensor{f}^{\supGLM}=
\begin{pmatrix}
\rho \vec{v} \\[0.15cm]
\rho (\vec{v}\, \vec{v}^{\,T}) + p\threeMatrix{I} \\[0.15cm]
\vec{v}\left(\frac{1}{2}\rho \left\|\vec{v}\right\|^2 + \frac{\gamma p}{\gamma -1}\right)  \\[0.15cm]
\threeMatrix{0}\\ \vec{0}\\[0.15cm]
\end{pmatrix} +
\begin{pmatrix}
\vec{0} \\[0.15cm]
\frac{1}{2 \mu_0} \|\vec{B}\|^2 \threeMatrix{I} - \frac{1}{\mu_0} \vec{B} \vec{B}^T \\[0.15cm]
\frac{1}{\mu_0} \left( \vec{v}\,\|\vec{B}\|^2 - \vec{B}\left(\vec{v}\cdot\vec{B}\right) \right) \\[0.15cm]
\vec{v}\,\vec{B}^T - \vec{B}\,\vec{v}^{\,T} \\ \vec{0}\\[0.15cm]
\end{pmatrix} +
\begin{pmatrix}
\vec{0} \\[0.15cm]
\threeMatrix{0} \\[0.15cm]
\frac{c_h}{\mu_0} \psi \vec{B} \\[0.15cm]
c_h \psi \threeMatrix{I} \\ c_h \vec{B}\\[0.15cm]
\end{pmatrix},
\label{eq:advective_fluxes}
\end{equation}
where $p$ is the gas pressure, $\threeMatrix{I}$ is the $3\times 3$ identity matrix, $\mu_0$ is the permeability of the medium, and $c_h$ is the \textit{hyperbolic divergence cleaning speed},
The variable $c_h$ is a time-dependent parameter in the simulation, which is computed at every time step as the maximum value that retains CFL-stability.

We close the system with the (GLM) calorically perfect gas assumption,
\begin{equation}
p = (\gamma-1)\left(\rho  E - \frac{1}{2}\rho\left\|\vec{v}\right\|^2 - \frac{1}{2 \mu_0}\|\vec{B}\|^2 - \frac{1}{2 \mu_0}\psi^2\right),
\label{eqofstate}
\end{equation}
where $\gamma$ denotes the heat capacity ratio.

In these equations, the divergence-free condition, $\Nabla \cdot \vec{B} = 0$, is not enforced exactly.
However, the solution evolves towards a divergence-free state \cite{Munz2000,Dedner2002,derigs2018ideal}.

\subsubsection{Orszag-Tang Vortex}

The Orszag-Tang Vortex is an inviscid 2D MHD case that was originally proposed by \citet{orszag1979small}, which is widely used to test the robustness of MHD codes \cite{Chandrashekar2016,derigs2018ideal,Stone2008,kuzmin2020limiting}.
The simulation starts from a smooth initial condition, which evolves into complex shock patterns with several shock-shock interactions and transitions to MHD turbulence.

We tessellate the simulation domain, $\Omega = [0,1]^2$,  with a Cartesian grid and periodic boundary conditions.
The initial conditions are
\begin{align*}
\rho_0(x,y) &= \frac{25}{36 \pi},
&  p_0(x,y) &= \frac{5}{12 \pi}, \\
 v_{1,0}(x,y) &= - \sin (2 \pi y),
&v_{2,0}(x,y) &=   \sin (2 \pi x), \\
 B_{1,0}(x,y) &= -\frac{1}{\sqrt{4 \pi}} \sin (2 \pi y),
&B_{2,0}(x,y) &= -\frac{1}{\sqrt{4 \pi}} \sin (4 \pi x).
\end{align*}

In this example, we use the subcell-wise limiting technique.
As in \cite{kuzmin2020limiting}, we compute the blending coefficient at each RK stage to impose a TVD-like solution on the density,
\begin{equation} \label{eq:IDPconditionOT}
    \min_{k \in \NN (ij)} {\rho}^{\FV}_{k}
    \le \rho_{ij} \le
    \max_{k \in \NN (ij)} {\rho}^{\FV}_{k},
\end{equation}
where $\rho_k^{\FV}$ is the solution at the next RK stage that is obtained with the low-order method for node $k$.

We solve this problem until $t=1$ with $256 \times 256$ elements, the polynomial degree $N=3$, and we use the split-form DGSEM, the volume numerical flux of \citet{Chandrashekar2016}, and the LLF as the surface numerical flux.
Moreover, we use subcell FV schemes with first- and second-order reconstructions on primitive variables ($\rho$, $\vec{v}$, $p$, $\vec{B}$ and $\psi$) as the compatible low-order method.

We note that, since the selected low-order methods are guaranteed to be monotonic for the primitive variables and the density bounds are obtained from the low-order solutions in the FV stencil, \eqref{eq:IDPconditionOT} can always be satisfied by the hybrid DG/FV scheme.

\begin{figure}
\centering
\includegraphics[trim=0 1400 1450 0 ,clip,width=0.48\linewidth]{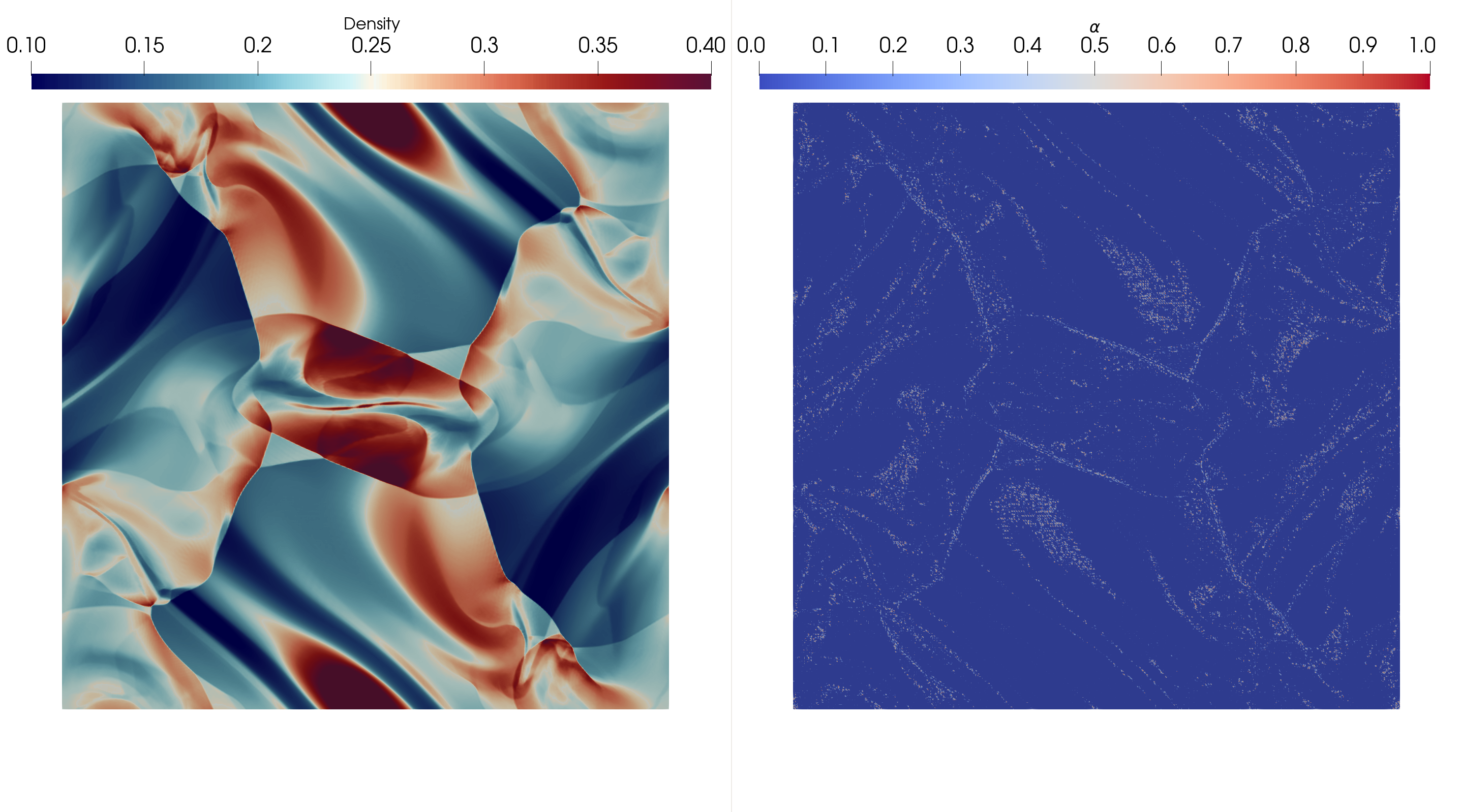}
\includegraphics[trim=1450 1400 0 0 ,clip,width=0.48\linewidth]{figs/OT/OT_1st_LLF_densL.png}
\begin{subfigure}[b]{\linewidth}
    \centering
	\includegraphics[trim=110 200 1550 200 ,clip,width=0.35\linewidth]{figs/OT/OT_1st_LLF_densL.png}
	\hspace{50pt}
	\includegraphics[trim=1550 200 110 200 ,clip,width=0.35\linewidth]{figs/OT/OT_1st_LLF_densL.png}
	\caption{ESDGSEM + first-order FV (LLF)}
	\label{fig:ot_subcell_05_1}
\end{subfigure}
\begin{subfigure}[b]{\linewidth}
    \centering
	\includegraphics[trim=110 200 1550 200 ,clip,width=0.35\linewidth]{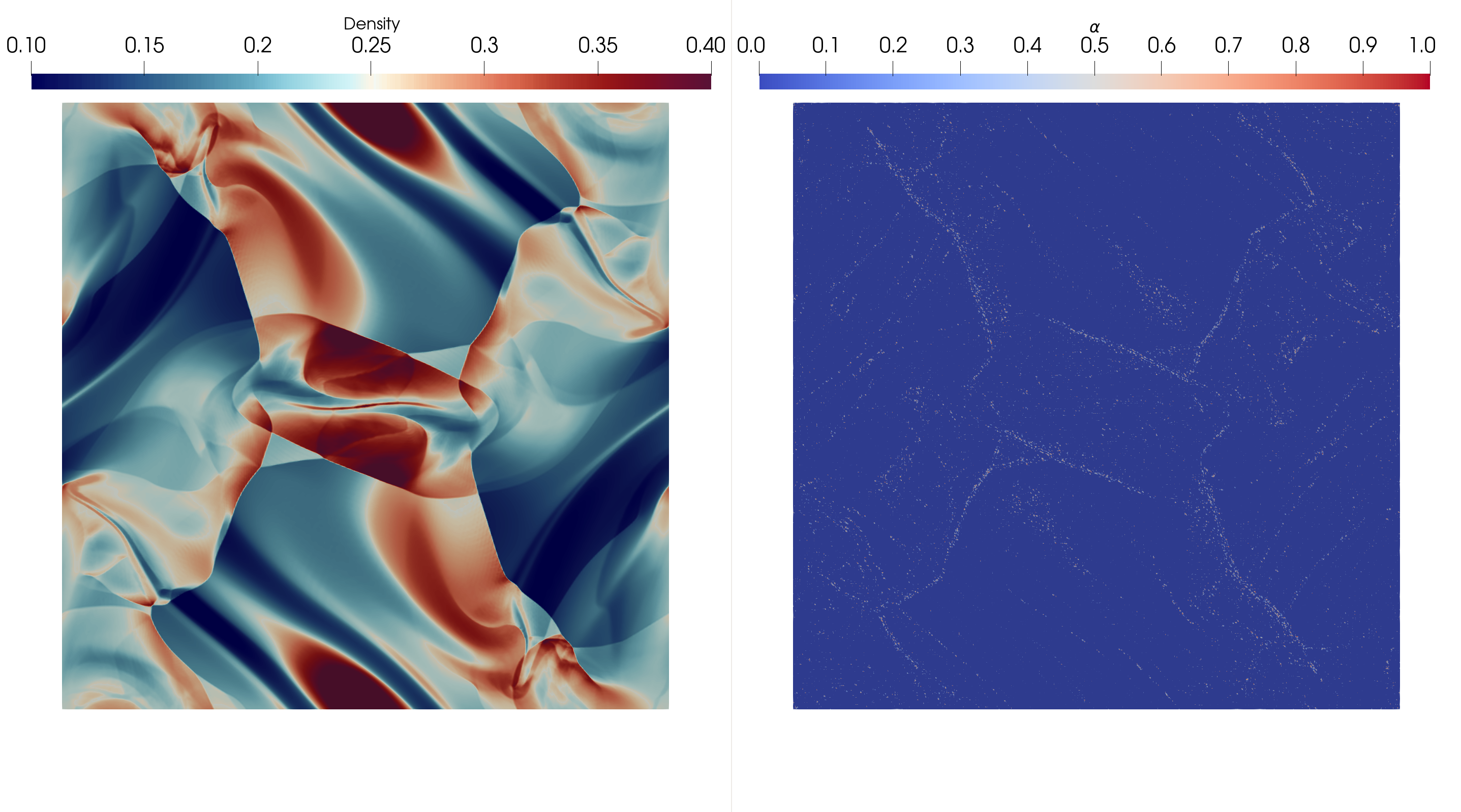}
	\hspace{50pt}
	\includegraphics[trim=1550 200 110 200 ,clip,width=0.35\linewidth]{figs/OT/OT_2nd_LLF_densL.png}
	\caption{ESDGSEM + second-order FV (LLF)}
	\label{fig:ot_subcell_05_2}
\end{subfigure}

\caption{Simulation results for the Orszag-Tang vortex simulation with subcell-wise convex blending at time $t = 0.5$.
We plot the density, $\rho$, and the instant blending coefficient, $\alpha$, for two different schemes considered. DGSEM results with polynomial degree $N=3$ and $256\times 256$ elements.}
\label{fig:ot_subcell_05}
\end{figure}

The main results are shown in Figure~\ref{fig:ot_subcell_05} for time $t=0.5$ and in Figure~\ref{fig:ot_subcell_10} for time $t=1.0$.
Since the second-order FV method is less dissipative than its first-order counterpart, the bounds that it imposes are less strict, which leads to less limiting (as can be observed clearly in Figure~\ref{fig:ot_subcell_05}).
As a result, the scheme that uses second-order limiting is less dissipative and allows the appearance of more turbulent structures, as can be observed in Figure~\ref{fig:ot_subcell_10}.

\begin{figure}
\centering
\includegraphics[trim=0 1400 1450 0 ,clip,width=0.48\linewidth]{figs/OT/OT_1st_LLF_densL.png}
\includegraphics[trim=1450 1400 0 0 ,clip,width=0.48\linewidth]{figs/OT/OT_1st_LLF_densL.png}
\begin{subfigure}[b]{\linewidth}
    \centering
	\includegraphics[trim=110 200 1550 200 ,clip,width=0.35\linewidth]{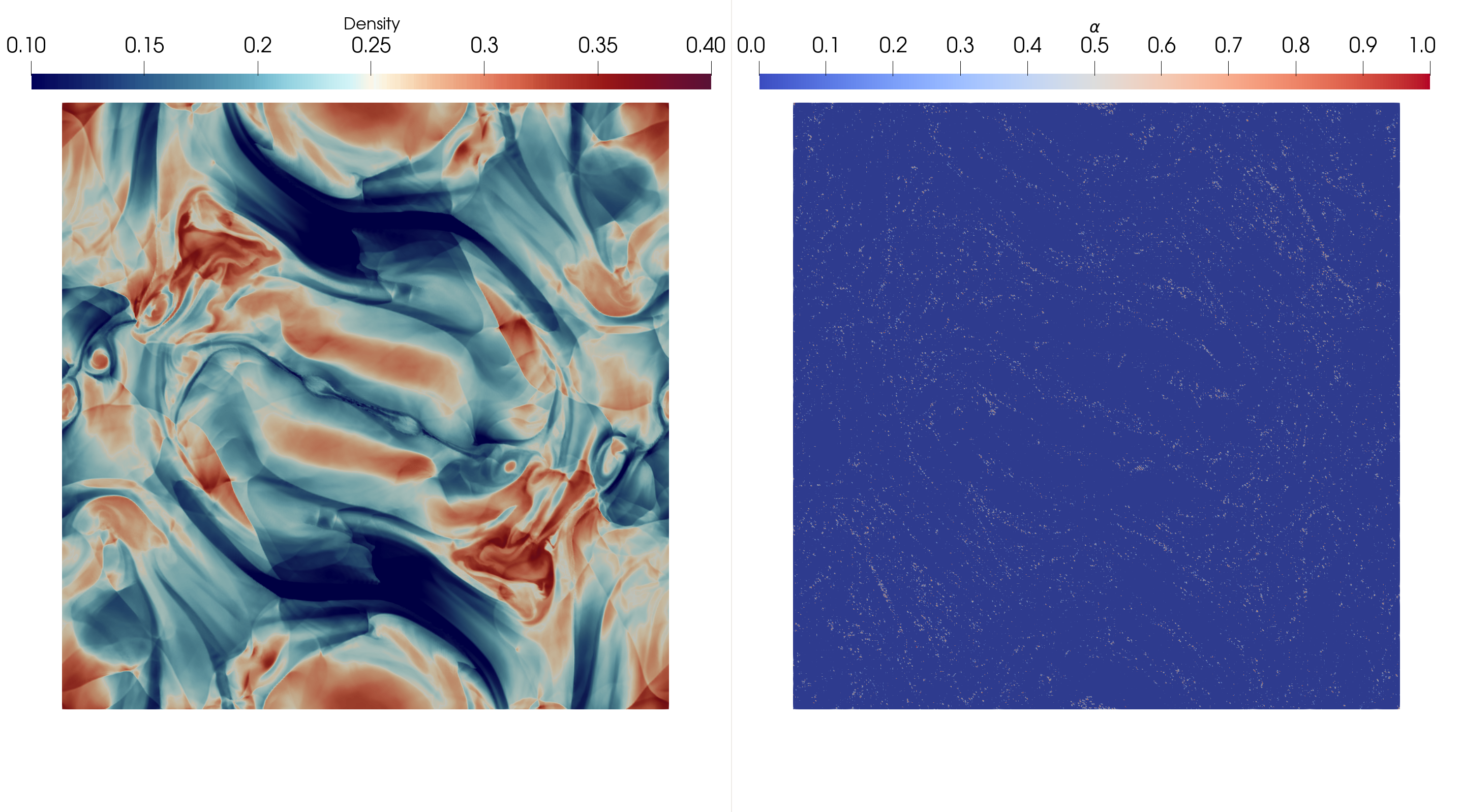}
	\hspace{50pt}
	\includegraphics[trim=1550 200 110 200 ,clip,width=0.35\linewidth]{figs/OT/OT_1st_LLF_densL_final.png}
	\caption{ESDGSEM + first-order FV (LLF)}
	\label{fig:ot_subcell_10_1}
\end{subfigure}
\begin{subfigure}[b]{\linewidth}
    \centering
	\includegraphics[trim=110 200 1550 200 ,clip,width=0.35\linewidth]{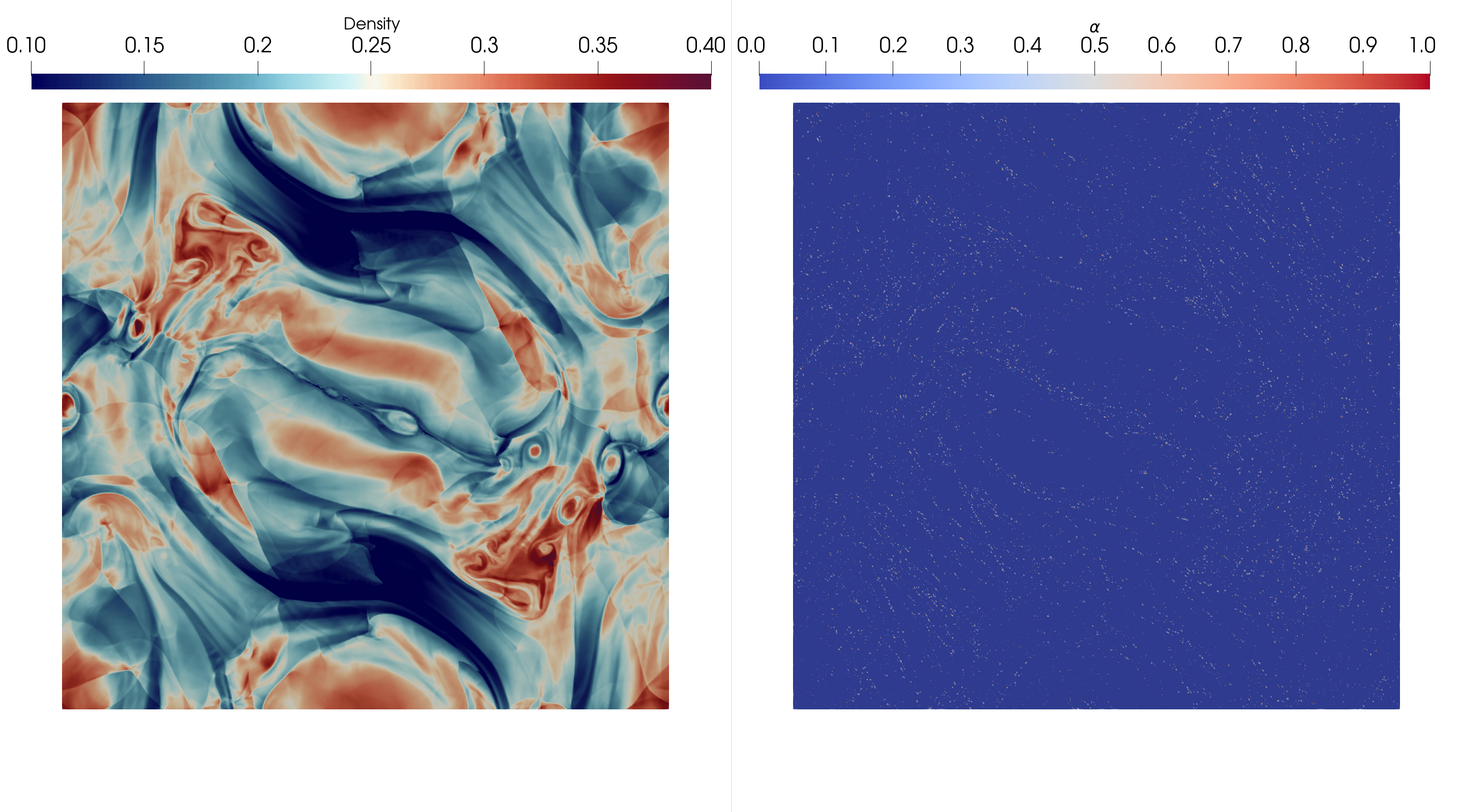}
	\hspace{50pt}
	\includegraphics[trim=1550 200 110 200 ,clip,width=0.35\linewidth]{figs/OT/OT_2nd_LLF_densL_final.png}
	\caption{ESDGSEM + second-order FV (LLF)}
	\label{fig:ot_subcell_10_2}
\end{subfigure}

\caption{Simulation results for the Orszag-Tang vortex simulation with subcell-wise convex blending at time $t = 1.0$.
We plot the density, $\rho$, and the instant blending coefficient, $\alpha$, for two different schemes considered. DGSEM results with polynomial degree $N=3$ and $256\times 256$ elements.}
\label{fig:ot_subcell_10}
\end{figure}

\section{Conclusions}\label{sec:conclusion}

In this paper, we discuss strategies for subcell limiting of high-order DG methods based on convex blending on curvilinear meshes.
We first enumerate the four components necessary to formulate the method.
The main ingredients are high-order DGSEM and a compatible subcell FV scheme that evolves the same degrees of freedom as the high-order scheme.
The high-order scheme admits split-formulations that can be, for instance, kinetic energy preserving or entropy consistent.
The low-order scheme allows for the application of a large number of standard techniques from the FV literature, e.g., allowing flexibility in the choice of approximate Riemann solver and the type of slope reconstruction.
One further needs to decide whether an entire element is blended, or if the nodal values within an element are individually blended.
While element-wise blending is simpler and allows for a global entropy estimate of the hybrid discretization, it is also typically more dissipative than the subcell blending.
Lastly, suitable indicators to compute the blending coefficients are necessary.
Options range from simple and cheap heuristic options based on modal energy estimates as troubled cell indicators up to FCT-like a posteriori approaches that guarantee positivity of density and pressure and preservation of convex invariants throughout the duration of the simulation.

By choosing specific options for these ingredients, we have demonstrated that existing limiting approaches from the literature may be recovered using the unified framework.
In addition to recovering existing versions, the unified point of view allows to further mix and match the four ingredients to obtain shock-capturing and positivity-preserving methods tailored to the needs of the user.
Some of the resulting combinations allow for strong theoretical statements, e.g., provable entropy dissipation, provable positivity of density and pressure, as well as other properties of invariant domain preservation.
Other combinations depend more on heuristic choices, but empirical results suggest that they may substantially improve real-world performance in terms of amount of dissipation added to the discretization and/or computational complexity.

To investigate the versatility of the convex blending strategies in the context of the unified framework, we consider several challenging test cases from the literature and demonstrate the performance of the resulting high-order DGSEM discretizations. For instance, we demonstrate that for the KPP problem (featuring a non-convex flux function), it is beneficial to use an entropic high-order DGSEM in order to guarantee convergence to the unique entropy solution. We also demonstrated that it is possible to robustly simulate a challenging Mach 2000 jet flow inspired by an astrophysical setup and maintain positivity of the solution without excessive dissipation. Lastly, we also show that the strategies can be applied for more complex hyperbolic systems, such as the ideal MHD system with divergence cleaning.

\section*{Acknowledgments}

This work has received funding from the European Research Council through the ERC Starting Grant “An Exascale aware and Uncrashable Space-Time-Adaptive Discontinuous Spectral Element Solver for Non-Linear Conservation Laws” (Extreme), ERC grant agreement no. 714487 (Gregor J. Gassner and Andrés Rueda-Ramírez). This work was
performed on the Cologne High Efficiency Operating Platform for Sciences (CHEOPS) at the Regionales Rechenzentrum K\"oln (RRZK) and on the group cluster ODIN.
We thank RRZK for the hosting and maintenance of the clusters.
This work was performed under the auspices of the U.S.\ Department of Energy by Lawrence Livermore National Laboratory under Contract DE-AC52-07NA27344 (LLNL-JRNL-831293).

\bibliographystyle{model1-num-names}
\bibliography{Biblio}

\end{document}